\documentclass[aos]{imsart}

\RequirePackage{amsthm,amsmath,amsfonts,amssymb}
\RequirePackage[numbers]{natbib}

\usepackage{bbm}

\RequirePackage[OT1]{fontenc}
\RequirePackage[colorlinks,linkcolor=blue,citecolor=blue,urlcolor=blue]{hyperref}
\usepackage{graphicx}
\usepackage{caption}
\usepackage{amssymb}
\usepackage{extarrows}
\usepackage{multirow}
\usepackage{stmaryrd}
\usepackage{color}
\usepackage{enumitem}
\usepackage{ulem} 
\usepackage{dsfont}
\usepackage{enumerate}
\numberwithin{equation}{section}

\theoremstyle{plain} 
\newtheorem{thm}{Theorem}[section]
\newtheorem{lem}[thm]{Lemma}

\newtheorem{pro}[thm]{Proposition}
\newtheorem{assumption}[thm]{Assumption}
\newtheorem{defn}[thm]{Definition}

\newtheorem*{theorem}{Theorem}

\theoremstyle{remark}
\newtheorem{rem}[thm]{Remark}

\newcommand{\mycomment}[1]{}
\def\Tr{\mathrm{Tr}}

\renewcommand{\Re}{\mathrm{Re}\,}
\renewcommand{\Im}{\mathrm{Im}\,}

\newcommand{\im}{\mathrm{Im}\,}

\newcommand{\E}{{\mathbb E }}

\newcommand{\ii}{\mathrm{i}}
\newcommand{\deq}{\mathrel{\mathop:}=}

\newcommand{\bs}{\boldsymbol}

\newcommand{\nc}{\normalcolor}

\renewcommand{\mathbf}[1]{\bs{#1}}

\newcommand{\eps}{\epsilon}

\startlocaldefs

\endlocaldefs

\begin{document}

\begin{frontmatter}
\title{Signal detection from spiked noise via asymmetrization}
\runtitle{Signal detection via asymmetrization}

\begin{aug}
\author[A]{\fnms{Zhigang}~\snm{Bao}\ead[label=e1]{zgbao@hku.hk}},
\author[B]{\fnms{Kha Man}~\snm{Cheong}\ead[label=e2]{kmcheong@connect.ust.hk}}
\author[A]{\fnms{Jaehun}~\snm{Lee}\ead[label=e4]{jaehun@hku.hk}}
\and
\author[A]{\fnms{Yuji}~\snm{Li}\ead[label=e3]{u3011732@connect.hku.hk}}
\address[A]{University of Hong Kong\printead[presep={, }]{e1,e4,e3}}

\address[B]{Hong Kong University of Science and Technology\printead[presep={, }]{e2}}
\end{aug}

\begin{abstract}
The signal plus noise model $H=S+Y$ is a fundamental model in signal detection when a low rank signal  $S$ is polluted by   noise  $Y$. In the high-dimensional setting, one often uses the leading singular values and corresponding singular vectors of $H$ to conduct the statistical inference of the signal  $S$. Especially, when $Y$ consists of iid random entries, the singular values of $S$ can be estimated from those of $H$ as long as the signal $S$ is strong enough.  However, when the $Y$ entries are heteroscedastic or heavy-tailed, this standard approach may fail. Especially in this work, we consider a situation that can easily arise with heteroscedastic or heavy-tailed noise but is particularly difficult to address using the singular value approach, namely, when the noise $Y$ itself may create spiked singular values.  It has been a recurring question how to distinguish the signal 
$S$ from the spikes in $Y$, as this seems impossible by examining the leading singular values of $H$. Inspired by the work \cite{CCF21}, we turn to study the eigenvalues of an asymmetrized model when two samples $H_1=S+Y_1$ and $H_2=S+Y_2$ are available. We show that by looking into the leading eigenvalues (in magnitude) of the asymmetrized model $H_1H_2^*$, one can easily detect $S$. Unlike \cite{CCF21}, we show that even if the spikes from $Y$ is much larger than the strength of $S$, and thus the operator norm of $Y$ is much larger than that of $S$, the detection is still effective. Second, we establish the precise detection threshold. Third, we do not require any structural assumption on the singular vectors of $S$. Finally, we derive precise limiting behaviour of the leading eigenvalues and eigenvectors of the asymmetrized model. Based on the limiting results, we propose a completely data-based approach for the detection of $S$.  We will primarily discuss the heteroscedastic case and then discuss the extension to the heavy-tailed case. As a byproduct, we also derive the fundamental result regarding the outlier of non-Hermitian random matrix in \cite{Tao} under the minimal 2nd moment condition.
\end{abstract}



\end{frontmatter}

\section{Introduction}

In this paper, we consider the following signal-plus-noise model 
\begin{align*}
H=S+Y\equiv S+\Sigma X,
\end{align*}
where $X=(x_{ij})\in \mathbb{R}^{p\times n}$ is a random matrix with independent mean $0$ entries. We further denote the variance profile of the matrix by
\begin{align*}
T=(t_{ij}):=n(\text{Var}(x_{ij}))
\end{align*}
which is the matrix with entries given by the variances of $\sqrt{n}x_{ij}$'s.  Throughout the paper, we make the following assumption on the variance profile: there exist two positive constants $t_*\leq t^*$  such that
\begin{align}
t_*\leq \min_{i,j} t_{ij} \leq \max_{i,j} t_{ij}\leq t^*. \label{flatness_assumption}
\end{align}
 Here we assume that $S\in \mathbb{R}^{p\times n}$ is a fixed rank matrix  and $\Sigma\in \mathbb{R}^{p\times p}$ is a fixed rank perturbation of identity, i.e., 
\begin{align}
S=\sum_{i=1}^k d_i u_iv_i^*=:UDV^*, \qquad \Sigma= I +\sum_{j=1}^r \sigma_j\xi_j\theta_j^*=: I +\Xi \Delta \Theta^*, \label{022701} 
\end{align}
where $k$ and $r$ are fixed nonnegative integers, and $\{u_i\}$, $\{v_i\}$, $\{\xi_i\}$ and $\{\theta_i\}$ are 4 classes of deterministic orthonormal vectors. Here $D=\text{diag}(d_1, \ldots, d_k)$ and $\Delta=\text{diag}(\sigma_1, \ldots, \sigma_r)$, where $\{d_i\}'s$ and $\{\sigma_i\}'s$ are two collections of nonnegative numbers,  both ordered in descending order.  Throughout the paper, we always assume that $\Sigma$ is invertible and further assume $\|\Sigma^{-1}\|_{\text{op}}\leq K$ for some constant $K>0$. We interpret $S$ as a signal, which is polluted by the noise part $\Sigma X$.  Let $\mathbbm{1}$ be the all-one matrix. Differently from the classical setting in many previous literature where $(\Sigma, T)=(I, \mathbbm{1})$, here we consider the general setting when the noise part $\Sigma X$ itself may be heteroscedastic or even correlated. We remark here that in case $\Sigma$ is diagonal, in principle we can simply write $\Sigma X$ and $\widetilde{X}$, where the latter still has independent entries and a general variance profile $\widetilde{T}$, analogously to $T$. Even in this case, we would prefer to keep the writing $\Sigma X$, as this will give us the flexibility of choosing $\sigma_i$'s to be even $n$-dependent, as in this case $\widetilde{T}$-entries no longer satisfy the assumption (\ref{flatness_assumption}). At this moment, we also assume that all moment of $x_{ij}$'s exist, although by some standard truncation technique the condition can be relaxed. Later, we will also consider a heavy-tailed scenario when only 2nd moment exists. 

In case $(\Sigma, T)=(I, \mathbbm{1})$, a standard approach to detect $S$ from $H$ is to investigate the leading singular values of $H$, as natural estimators of the counterpart of $S$. In the high-dimensional setting when $p$ and $n$ are proportional, there has been a vast of literature on the singular value approach; see \cite{Ding20, BDW22, Capitaine, CD16, BN12, LK25, GD14, CTT17, JN17, GD17, JCL21} for instance.  Specifically, as a prominent example of the famous Baik-Ben Arous-P\'{e}ch\'{e} (BBP) phase transition phenomenon \cite{BBP}, one knows that the $i$-th leading singular value of $H$ will jump out of the support of the Marchenko-Pastur law, and converges to a limiting location which is a function of $d_i$, if the $d_i$ is sufficiently large. From the limiting location of the leading singular value of $H$, one can recover the value of $d_i$.  We also refer to \cite{BS06, BN11, BY12, CD16, Paul, BPZ, CD08, CD12, BDWW, KY14, KY13, BW22} and the reference therein for the study of BBP transition on other models such as deformed Wigner matrices and spiked covariance matrices.  The BBP transition has found numerous applications in various statistical problems; see \cite{JP18, WY17, PLY17, LWY17, BDMN11, O09} for instance. 

The BBP transition for the signal-plus-noise model can be further extended to the case when $\Sigma$ is general but itself does not create any outliers in the singular value distribution of $\Sigma X$; see \cite{DY22} for instance. In this case, if $d_i$ is sufficiently large, one can still observe outliers in the singular value distribution of $H$. Nevertheless, this time, the limiting location of the leading singular value will depend on the parameters in $\Sigma$ and $T$ as well, which are often unknown and hard to estimate in applications. This prevents one from using these limiting locations to estimate $d_i$'s. As the signal plus noise model with heteroscedastic noise is ubiquitous (see \cite{LK25, B19, CH77, F09, HS19, SH14, W06} for instance), alternative approach to detect the signal from such general model is in high demand.  

In order to solve the signal detection problem for the heteroscedastic case, in \cite{CCF21} the authors proposed  an asymmetrized non-Hermitian random matrix model for this purpose when two samples of the data are available. Actually, the work \cite{CCF21} primarily focuses on the deformed Wigner matrices. But the strategy can be generalized to signal-plus-noise model as well, as mentioned in \cite{CCF21}. More specifically, if one has two independent samples $H_1=S+\Sigma X_1$ and $H_2=S+\Sigma X_2$, one can turn to consider the random matrix model $H_1H_2^*$ or a linearization of it. Under an assumption that the operator norm of the noise part $\Sigma X$ is significantly smaller than the signal part $S$, the authors in \cite{CCF21}  find that the leading eigenvalue (in magnitude) of the non-Hermitian model contains the precise information of $d_i$'s in case $\Sigma=I$, and the  heteroscedasticity of $\Sigma X$ only shows up in a subleading order, and thus the unknown parameters do not matter if one aims for the first order estimate of the signals. From the mathematical point of view, the closeness between the leading eigenvalues of the signal plus non-Hermitian noise matrix model and its signal part was previously revealed in \cite{Tao}, where the outlier of the low rank deformation of non-Hermitian square matrix with iid entries is studied. In particular, it reveals a striking difference from the Hermitian case, where the outlier exhibits an order-1 bias toward the true signal.  We also refer to \cite{BC16, BBCC21,Ra15, Han24} for further study on the outliers of the deformed non-Hermitian random matrices.  The work \cite{CCF21} provides a very interesting application of this closeness and also provide a non-asymptotic analysis, under general assumption on $T$. We also refer to \cite{BCN} for an application of asymmetrization to the sparse matrix completion problem.  In this work, we will continue this line of research to show that the asymmetrization technique is powerful even various conditions in \cite{CCF21} are not satisfied, and it can be used to tackle other challenging questions in signal detection. Especially, we will answer a recurring question on how to detect the signal part $S$ from $H$ when $\Sigma X$ itself may create large spiked singular values. Although distinguishing the outlying singular values of $H$ caused by the signal part $S$ and the noise part $\Sigma$ is difficult, we find again, the asymmetrization approach is effective for this purpose. Unlike the situation in \cite{CCF21}, where detecting the existence of the signal is not an issue and only obtaining a precise estimate is challenging using the singular value approach, in our case, even detecting the existence of the signal using the singular value approach is difficult. In addition,  differently from \cite{CCF21}, we show that even when the spikes in $\Sigma$ is much  larger than $d_i$'s in $S$, and thus the operator norm of the noise part is much larger than $S$, one can still effectively detect $S$. More specifically, we will provide the optimal detectability threshold, $\sqrt{n^{-1}\|T\|_{\text{op}}}$,  for $d_i$'s, as long as the spikes in $\Sigma$ is not $n^{1/4}$ times larger than $d_i$'s. Even when $\Sigma=I$, our threshold in terms of $T$ is more precise than the sufficiently large signal-to-noise ratio imposed in \cite{CCF21}. It particularly shows that the operator norm of the noise matrix $\Sigma X$ itself is not essential in the detectability of $d_i$'s. Instead, the essential threshold is given by $\sqrt{n^{-1}\|T\|_{\text{op}}}$ which could be significantly smaller than $\|\Sigma X\|_{\text{op}}$. We then take a step further to identify the precise fluctuation of the outliers of $H_1H_2^*$ around the limiting location. 

 Another typical scenario when a spiked singular value can be created by the noise is the heavy-tailed case. In this work, under the classical setting of $(\Sigma, T)=(I, \mathbbm{1})$, we will also consider the signal detection problem when $X$ entries only have the 2nd moment. It is known that when the 4th moment is absent, the largest eigenvalues of various classical Hermitian random matrix models will jump out of the support of the limiting spectral distribution, and form a Poisson process on a much larger scale. Especially, for the Wigner matrices and the sample covariance matrices, the largest eigenvalue is diverging and follows a Fr\'{e}chet law; see \cite{ABP09,Sosh04}. Actually, in this case, there are a diverging number of outliers in the spectrum of the Hermitian random matrices. Back to our signal-plus-noise model, in case $X$ only has the 2nd moment but not the  4th moment, similarly, there will be a diverging number of spikes from the noise if one considers the singular values of  $H$. In contrast, unlike the Hermitian random matrices, for which a 4th moment condition is necessary for the convergence of the largest eigenvalue as shown by the famous Bai-Yin law \cite{BY93}, a striking fact revealed by Bordenave, Chafa\"{i}, Garc\'{i}a-Zelada \cite{BCG22} is that a 2nd moment condition is already sufficient for the convergence of the spectral radius of the iid non-Hermitian random matrices. This fundamental difference again inspires us to use the asymmetrization approach to attack the signal detection for $H$ in the heavy-tailed case. In particular, we will show that for $H_1H_2^*$ or its linearization, only the signal part $S$ will create outlying eigenvalues, but the noise part $X$ will not, as long as the 2nd moment of $X$ entries exist. Our derivation in this part can be easily extended to the model of low rank deformation of iid non-Hermitian matrix. As a byproduct, we extend the fundamental limiting result of Tao \cite{Tao} regarding the outliers of i.i.d. random matrices to the minimal 2nd moment condition. This result was originally proved in \cite{Tao} under a 4th moment condition and was very recently extended in \cite{Han24} to the heavy-tailed regime when the deformation satisfies strong structural assumption; see Remark \ref{rmk.071201}.

We will primarily focus on the discussion of the heteroscedastic case when $(\Sigma, T)$ is general but $X$ is light-tailed. Later, we will extend our discussion to the heavy-tailed case, but for the classical setting $(\Sigma,T)=(I, \mathbbm{1})$. 

In general, we will work with the following non-Hermitian random matrix 
\begin{align}
\mathcal{Y}=\left(
\begin{array}{ccc}
~ & H_1\\
H_2^* &~
\end{array}
\right)=\left(
\begin{array}{ccc}
~ & \Sigma X_1\\
(\Sigma X_2)^* &~
\end{array}
\right)+\left(
\begin{array}{ccc}
~ & S\\
S^* &~
\end{array}
\right)=: \mathcal{X}+\mathcal{S}  \label{model}
\end{align}
which is a linearization of $H_1H_2^*$. Note that, due to the block structure,  the eigenvalues of $\mathcal{Y}$ are in pair. We denote the non-zero eigenvalues of $\mathcal{Y}$ by $\lambda_{\pm i}\equiv \pm \lambda_i$,  $i=1, \ldots, n\wedge p$. We make the convention that $\lambda_{1}, \ldots, \lambda_{n\wedge p}$ are those eigenvalues with arguments in $(-\pi/2, \pi/2]$. Further,  $\lambda_{1}, \ldots, \lambda_{n\wedge p}$ are in descending order (in magnitude). Due to the fact that we are considering real matrix, we have another symmetry of the eigenvalues, namely, all non-real eigenvalues in $\{\lambda_{1}, \ldots, \lambda_{n\wedge p}\}$ can find their complex conjugates in this collection. Hence, regarding the ordering according to magnitude, we further make the convention that the one with argument in $(0, \pi/2)$ is followed by its complex conjugate in $(-\pi/2, 0)$.

Throughout the paper, we will be working with the following assumption in the heteroscedastic case. 
\begin{assumption}\label{main assump}  We make the following assumptions. 

\vspace{1ex}
\noindent (i) (\textit{On dimensionality}): We assume that $p\equiv p(n)$ and $n$ are comparable, i.e.  there exists a constant $c$, such that 
    \begin{align*}
       c_n\equiv  \frac{p}{n} \rightarrow c  \in (0, \infty) \text{ as } n \rightarrow \infty.
    \end{align*}
  
  \vspace{1ex}  
\noindent (ii) (\textit{On signal $S$}): We assume that $S$ is a low rank matrix with rank $k$ and admits the singular value decomposition, i.e. 
    \begin{align*}
        S = \sum_{i=1}^k d_i u_i v_i^* =: UDV^*
    \end{align*}
   Here $k \geq 0$ is fixed, $D = \text{diag}(d_1, \ldots, d_r)$ with $C > d_1 \geq \ldots\geq d_k \geq  0$ for some constant $C>0$, and $u_i$'s and $v_i$'s are the associated unit left and right singular vectors, respectively. 

\vspace{1ex}
\noindent (iii) (\textit{On $\Sigma$}): We assume that $\Sigma$ is a rank $r$ perturbation of identity,  i.e. 
    \begin{align*}
        \Sigma = I + \sum_{j=1}^r \sigma_j\xi_j\theta_j^* =: I + \Xi\Delta\Theta^*, 
    \end{align*}
    where $r\geq 0$ is fixed; $\Delta = \text{diag}(\sigma_1, \ldots, \sigma_r)$ with $\|\Delta\|_{op} \leq n^{1/4-\varepsilon_0}$ for some small but fixed $\varepsilon_0 > 0$; $\xi_i$'s and $\theta_i$'s are the associated unit left and right singular vectors of the low rank matrix $\Sigma - I$, respectively. In particular, $\|\Sigma\|_{\text{op}}$ can be much larger than the constant order.  Finally, we always assume $\Sigma$ to be invertible and $\|\Sigma^{-1}\|_{op} \leq K$, for some constant $K > 0$. 

\vspace{1ex}
\noindent (iv) (\textit{On matrix $X$}): We assume that $X=(x_{ij})$ has independent entries and have general variance profile $\frac{1}{n}T=\frac{1}{n}(t_{ij})$. Specifically,  the entries $x_{ij}$ are real random variables with
    \begin{align*}
        \mathbb{E}x_{ij} = 0, \qquad \mathbb{E}x_{ij}^2 = \frac{t_{ij}}{n}.
    \end{align*}
We restate (\ref{flatness_assumption}) below:  there exist two positive constants $t_*\leq t^*$  such that
\begin{align*}
t_*\leq \min_{i,j} t_{ij} \leq \max_{i,j} t_{ij}\leq t^*. 
\end{align*} 
For simplicity, we further assume that all moments of $x_{ij}$'s exist, i.e.  for any integer $p\geq3$, there exists a constant $C_p>0$, such that
\begin{align*}
   \max_{i,j} \mathbb{E}|\sqrt{n}x_{ij}|^p\leq C_p<\infty.
\end{align*}
\end{assumption}

We can also derive the fluctuation of the outliers based on the following additional assumption.

\begin{assumption} \label{assump.fluctuation} Suppose that Assumption \ref{main assump} holds.  We further assume that $d_i$'s are well separated, i.e., 
   \begin{align*}
   \min_{i\neq j} |d_i-d_j|\geq c
   \end{align*}
   for some small constant $c>0$. 
\end{assumption}

In the heavy-tailed case, we make the following assumptions.

\begin{assumption}\label{main assump 2}  We assume that Assumption \ref{main assump} (i), (ii) still hold, and we assume $(\Sigma, T)=(I, \mathbbm{1})$. Further, we assume

\noindent (iv') (\textit{On matrix $X$}):  $X=(x_{ij})$ has iid real entries with
    \begin{align*}
        \mathbb{E}x_{ij} = 0, \qquad \mathbb{E}x_{ij}^2 = \frac{1}{n}.
    \end{align*}
\end{assumption}


\begin{rem} Another popular setting considered in the literature is $(\Sigma, T) = (\text{general}, \mathbbm{1})$. Here, by "general," we mean that $\Sigma$ is not necessarily a fixed-rank perturbation of $I$; see \cite{LPZZ22}, for instance. We expect similar phenomena to occur under this setting. However, technically, it requires a separate derivation. We choose to work under Assumption \ref{main assump} for the heteroscedastic case mainly to facilitate a more direct comparison with the results in \cite{CCF21}. It is also possible to extend our discussion to the case where the noise components of the two samples have different $(\Sigma, T)$, which can occur in reality when the variance profile of the noise depends on the sampling. We leave all these extensions for future discussion. For the heavy-tailed case, we choose the classical setting $(\Sigma, T) = (I, \mathbbm{1})$ due to technical reason. 
\end{rem}

For simplicity, in the sequel, we denote
\begin{align}
\sigma_{\max}:=\|\Sigma\|_{\text{op}}. \label{def of sigma_max}
\end{align}

Throughout this paper, we adopt the notion of stochastic domination introduced in \cite{EKY2013}, which allows a loss of boundedness, up to a small power of $N$, with high probability.
\begin{defn}(Stochastic domination) \label{stochastic dom} Let
\begin{align*}
X=\big(X_N(u): N\in \mathbb{N}, u\in U\big), \quad Y=\big(Y_N(u): N\in \mathbb{N}, u\in U_N\big)
\end{align*}
 be two families of real random variables, where $Y$ is nonnegative, and $U_N$ is a possibly $N$-dependent parameter set. We say that \textit{$X$ is stochastically dominated by $Y$, uniformly in $u$}, if for arbitrary small $\epsilon > 0$, and large $D > 0$,
 \begin{align*}
 \sup_{u\in U_N}\mathbb{P}\big(|X_N(u)|>N^{\epsilon}Y_N(u)\big)\leq N^{-D}
 \end{align*}
 for large $N\geq N_0(\epsilon, D)$. We write $X=O_\prec(Y)$ or $X\prec Y$ when $X$ is stochastically bounded by $Y$ uniformly in $u$. Note that in the special case when $X$ and $Y$ are deterministic, $X\prec Y$ means for any given $\epsilon>0$,  $|X_{N}(u)|\leq N^{\epsilon}Y_{N}(u)$ uniformly in $u$, for all sufficiently large $N\geq N_0(\epsilon)$.
 
 In addition, we say that an event $\mathcal{E}\equiv \mathcal{E}(N)$ holds with high probability if 
 \begin{align*}
 \mathbb{P}(\mathcal{E})\geq 1-N^{-D}
 \end{align*}
 for any large constant $D>0$ when $N$ is large enough. 
\end{defn}

In this paper, we will mainly focus on the eigenvalue behavior, but eigenvector can be studied via similar approach.
We state the eigenvalue results in the main part, and leave some discussions about the eigenvector to Appendix \ref{s eigenvector}. 

For the heteroscedastic case, our main results are stated as follows.  

\begin{thm}[First order limit--heteroscedastic case] \label{thm: first order}
	Suppose that Assumption \ref{main assump} holds. If there exists a positive integer $\tilde{k} \leq k$ such that $d_1 \geq \ldots \geq d_{\tilde{k}} \geq \sqrt{n^{-1}\|T\|_{op}+\delta}$ for any small (but fixed) $\delta>0$, then the spectrum of $\mathcal{Y}$ has $\tilde{k}$ pairs of outliers. Additionally, such outlying eigenvalues converge in probability to the strength of the signal $ d_i$ in magnitude. Specifically, for $i =1, \ldots, \tilde{k}$, with high probability, we have for any $\epsilon>0$, 
    \begin{align*}
        |\lambda_i-d_i|\leq n^{-\frac12+\epsilon}\sigma_{\max}^2. 
    \end{align*} 
\end{thm}

\begin{rem}  In the null case when  $S=0$ and $\Sigma=I$, we have a random matrix $\mathcal{Y}$ with independent entries and two diagonal 0 blocks.  For non-Hermitian matrix with general flat variance profile, i.e., the variances of all entries are comparable to order $n^{-1}$, it is known from \cite{AEK18} that the spectral radius of the matrix is given by the square root of the spectral radius of the variance profile. If the result extended to the case with 0 diagonal blocks, we would see that the spectral radius of $\mathcal{Y}$ exactly converges to $\sqrt{n^{-1}\|T\|_{op}}$, in the null case. This indicates that our threshold is optimal. It is possible to prove (in the null) the convergence of the spectral radius of $\mathcal{Y}$ to $\sqrt{n^{-1}\|T\|_{op}}$ by adapting the discussions in \cite{AEK18, AEK21}, for instance. But for our results on outliers, it is enough to show an upper bound of the spectral radius in the null case.
\end{rem}

In the sequel, we 
denote by $\vec{T}_{k \cdot}$ the $k$-th row of $T$, and by $\vec{T}_{\cdot k}$ the $k$-th column of $T$.

\begin{thm}[Second order fluctuation--heteroscedastic case] \label{thm: second order}
Suppose that Assumption \ref{assump.fluctuation}  holds. Recall that $u_i$ and $v_i$ are the left and right singular vectors of signal $S$ associated with singular value $d_i$. If $d_i \geq \sqrt{n^{-1}\|T\|_{op}+\delta}$, we have  the expansion to the fluctuation order 
\begin{align}
\lambda_i = d_i+u_i^*\Sigma X_1 v_i + v_i^* (\Sigma X_2)^*u_i+\frac{1}{\sqrt{n}}\mathfrak{g}(1+o_p(1)). \label{041001}
\end{align}
Here $\mathfrak{g}$ is independent of $u_i^*\Sigma X_1 v_i + v_i^* (\Sigma X_2)^*u_i$, and it is a centered Gaussian random variable with variance
\begin{align*}
    \text{Var}(\mathfrak{g}) &= \frac{n}{d_i^4}\sum_{\alpha \beta}\Bigg[\frac{d_i^2}{n^2}V(\Sigma^* u_iu_i^*\Sigma, \Sigma^* u_iu_i^*\Sigma)M(T, \alpha, \beta) + \frac{2}{n^3}V(\Sigma^* u_iu_i^*\Sigma, v_iv_i^*)_{\alpha, \beta}N(TT^*, \alpha, \beta) \\
    &\qquad + \frac{d_i^2}{n^2}V(v_iv_i^*, v_iv_i^*)_{\alpha, \beta}M(T^*, \alpha, \beta) \Bigg],
\end{align*}
where
\begin{align*}
    &M(T, \alpha, \beta) = \vec{T}_{\alpha \cdot} \left[I-\frac{1}{n^{2}|z|^{2}}\left(T^*T\right)\right]^{-1}\vec{T}_{\beta \cdot} , \notag
    \\
    &N(TT^*, \alpha, \beta) = \Bigg(TT^*\left[I-\frac{1}{n^{2}|z|^{2}}\left(TT^*\right)\right]^{-1}\vec{T}_{\cdot \beta}\Bigg)_{\alpha} , \notag \\
    &V(pq^*, rs^*)_{\alpha, \beta} = (pq^*)_{\alpha\alpha}(rs^*)_{\beta\beta}
\end{align*}
for some vector $p, q, r, s \in \mathbb{R}^p$. 
\end{thm}

\begin{rem} Note that the fluctuation order  is not necessarily order $1/\sqrt{n}$ due to the possible $n$-dependence of $\sigma_i$'s. It could be as large as  $n^{-1/2}\sigma_{\max}^2$, depending on $\Sigma^*u_i$. 
The above theorem reveals a non-universal feature of the limiting distribution of the outlier, which has been previously observed in other Hermitian or non-Hermitian model; see \cite{CD08,KY13, Ra15} for instance. Particularly, the distribution of the outlier is (asymptotically) a convolution of  the distribution of a linear combination of $X_1, X_2$ entries and a Gaussian, which may not be Gaussian in case $u_i, v_i, \Sigma^* u_i$ are all localized, i.e., only a fixed number of components of them are nonzero. But apart from this case, the limiting distribution is still Gaussian by CLT. 
\end{rem}

\begin{rem} Our discussion can also be easily extended to the case when there are some multiple $d_i$, i.e, some of $d_i$'s are equal, while the distinct $d_i$'s are sufficiently well separated. In this case, a supercritical $d_i$ with multiplicity $k_i$ will create $k_i$ corresponding outliers $\lambda_i$'s. The joint distribution of these $k_i$ eigenvalues is given by that of the eigenvalues of a $k_i$ by $k_i$ random matrix, whose entry distribution can also be analyzed by our derivations. As a consequence, on fluctuation level, some of these $\lambda_i$'s will be truly complex, i.e., the imaginary part is not $0$. It can also be seen in our simulation study in Section \ref{s.simulation} and Appendix \ref{s eigenvector}.  For brevity, we leave the detailed discussion to future study. 
\end{rem}

Next, we state our result in the heavy-tailed case. In this case, we provide first order result for eigenvalues only. 

\begin{thm}[First order limit--heavy tailed case]\label{thm: first order opt}
	Suppose that Assumption \ref{main assump 2} holds. If there exist a positive integer $\tilde{k}\le k$  such that $d_{1}\ge\cdots\ge d_{\tilde{k}}>\sqrt{(p/n)^{1/2}+\delta}$ for small (but fixed) $\delta>0$.
Then,  for $i = 1, \ldots, \tilde{k}$, the following estimate holds in probability for large $n$: 
\begin{equation*}
	|\lambda_{i}-d_{i}|=o(1).
\end{equation*}
\end{thm}

We wish to highlight that our proof strategy for Theorem \ref{thm: first order opt} carries over to the classical low-rank deformation of an i.i.d.\ non-Hermitian random matrix, whose structure is even simpler. As a consequence, we strengthen the result of \cite{Tao} to the optimal 2nd moment assumption, stated in the following theorem. 

\begin{thm}\label{thm: out iid}
	Let $X$ be $n\times n$ random matrix with iid entries with mean 0 and variance $\frac{1}{n}$. Let $k$ be a fixed positive integer. Choose any constant $\delta>0$. Let $s^{*}>1+\delta$ be a constant. Consider a deterministic $n\times n$ matrix $C$ of bounded rank and bounded operator norm such that $\textnormal{rank}(C)=k$ and $\lVert C\rVert\le s^{*}$. Assume that for all sufficiently large $n$ there are $\tilde{k}$ eigenvalues of $C$, denoted by $c_{1},\cdots,c_{\tilde{k}}$, in the region $\{z:|z|>1+\delta\}$, and all other eigenvalues of $C$ are included in the region $\{z:|z|\le 1+\delta/3\}$. Then, for every $j\in\{1,2,\cdots,\tilde{k}\}$, there is exactly one eigenvalue of $X+C$ which converges to $c_{i}$ in probability. Moreover, in probability, for large $n$, all other eigenvalues of $X+C$ are included in the region $\{z:|z|< 1+\delta/2\}$.
\end{thm}

\begin{rem} We remark here that the eigenvalues of the deformation $C$ in the above theorem are not necessarily real. Actually, for our block non-Hermitian matrix $\mathcal{Y}$ in (\ref{model}), all our argument can be easily extended to the case when $\mathcal{S}$ is non-Hermitian and has complex eigenvalues. But from the application point of view, we display our results and derivations for the case when $\mathcal{S}$ only has real eigenvalues. 
\end{rem}

\begin{rem} \label{rmk.071201}
	In \cite{Han24}, Han obtained the same first-order limit for the outliers under the same 2nd-moment condition, but imposed strong structural conditions on the perturbation $C$: either the number of its non-zero entries is bounded, or every entry is $O(n^{-1})$.
	In contrast, our result applies to any perturbation $C$ of bounded rank and bounded operator norm, without further restrictions.
\end{rem}

\subsection{Proof Strategy} 

In this section, we briefly describe our proof strategy for the main results. 

We will start with the heteroscedastic case. We shall first show that, in the null case $(S, \Sigma)=(0, I)$, there is no outlier. More specifically, we denote by 
\begin{align}
	\mathcal{X}_0=\left(
\begin{array}{ccc}
~ & X_1\\
X_2^* &~
\end{array}
\right), \label{def of X_0}
\end{align}
which can be regarded as a linearization of $X_1X_2^*$. 
The variance profile of $\mathcal{X}_0$ is
\begin{align}
	\mathcal{V}=\frac{1}{n}\left(
\begin{array}{ccc}
~ &T\\
T^* &~
\end{array}
\right). \label{variance profile}
\end{align}
 This non-Hermitian matrix model can be regarded as a special case of the models considered in \cite{AEKN18}, where the authors consider a general Kronecker random matrix, which is a linear combination of Kronecker products of deterministic matrices and random matrices with independent entries but general variance profile.  The results in \cite{AEKN18} shows that under rather general assumption, the spectrum of the Kronecker random matrix is contained in the self-consistent $\tau$-pseudospectrum for any $\tau>0$. In our case, the self-consistent $\tau$-pseudospectrum is defined as  
\begin{align*}
	\mathbb{D}_{\tau}:=\{z\in \mathbb{C}:\text{dist}(0,\text{supp }\rho^z)\leq \tau\},
\end{align*}
where $\rho^z$ is  the so called self-consistent density of states, which is the deterministic approximation of the spectral distribution  of $\mathcal{X}_0-z$'s Hermitization, i.e., 
\begin{align*}
	\mathbf{H}^z=\left(
\begin{array}{ccc}
~ & \mathcal{X}_0-z\\
\mathcal{X}_0^*-\bar{z}
 &~
\end{array}
\right).
\end{align*}
One calls $\text{supp }\rho^z$ the self-consistent spectrum of $\mathbf{H}^z$. Hence, in order to prove that for any $\delta>0$, the spectral radius of $\mathcal{X}_0$ is bounded by $\sqrt{n^{-1}\|T\|_{\text{op}}+\delta}$ and thus that of $X_1X_2^*$ is bounded by $n^{-1}\|T\|_{\text{op}}+\delta$, it suffices to show that $z\not\in \mathbb{D}_{\tau}$ for some $\tau>0$, given that $|z|\geq \sqrt{n^{-1}\|T\|_{\text{op}}+\delta}$. Such a conclusion can be obtained by analyzing the Hermitized Dyson equation following the strategy in \cite{AEK18}. In our case, the Dyson equation boils down to a system of  four vector equations.  The analysis of the Hermitized random matrix $\mathbf{H}^z$ does not only lead to the precise upper bound of the spectral radius of $\mathcal{X}_0$ and $X_1X_2^*$, it also provides the high probability upper bound of the operator norms of the Green function
\begin{align*}
\|G(\lambda)\|_{\text{op}}\equiv \|(X_1X_2^*-\lambda)^{-1}\|_{\text{op}} \leq C, \qquad \text{if  } \lambda\geq n^{-1}\|T\|_{\text{op}}+\delta. 
\end{align*}

We then proceed with studying the case $S=0$ but $\Sigma$ is with deformation as defined in (\ref{022701}). In this case, we consider the noise part $\mathcal{X}$ in (\ref{model}), which can be regarded as a multiplicative deformation of $\mathcal{X}_0$. Our aim is to show that as long as $\sigma_{\max}\leq n^{1/4-\varepsilon_0}$ for some small constant $\varepsilon_0>0$, the results  proved for $\mathcal{X}_0$, including the upper bound of the spectral radius and also the upper bound for Green function,  still hold. That means, the multiplicative deformation does not change the spectrum of $\mathcal{X}_0$ significantly. In particular, we show the spectral radius of $\mathcal{X}$ is also bounded by $\sqrt{n^{-1}\|T\|_{\text{op}}+\delta}$ w.h.p, and  the operator norm of the Green function of $\Sigma X_1X_2^*\Sigma^*$ is order $1$.  To this end, we shall prove that $\det(\mathcal{X}-z)$ is uniformly nonzero for all $|z|\geq \sqrt{n^{-1}\|T\|_{\text{op}}+\delta}$. Based on the result for $\mathcal{X}_0$, it will be sufficient to show the smallness of the following centered quadratic form of Green function of $X_1X_2^*$: For any deterministic unit vectors $u$ and $v$, any $|\lambda|\geq n^{-1}\|T\|_{\text{op}}+\delta$, and any small constant $\varepsilon>0$
\begin{align}
|u^*\underline{G}(\lambda)v|\prec n^{-1/2+\varepsilon}, \qquad \underline{G}(\lambda)=G(\lambda)+\frac{1}{\lambda}.  \label{022710}
\end{align}
The above type of estimates has been considered in previous literature, such as \cite{Tao,Ra15, BC16, CCF21}, for other non-Hermitian random matrix models under various assumptions and with varying levels of precision. But all of them are carried out by rather delicate combinatorial moment method after applying a Neumann expansion. In our work, we 
prove (\ref{022710}) via a rather robust cumulant expansion approach, which is applied to the Green function directly. The cumulant expansion approach has been widely used to study Green function in random matrix theory; see \cite{KKP, LS18, HK17, EKS19} for instance. It is a bit surprising that it seems it has never been used to study the limiting behavior of the outliers of non-Hermitian matrix in the literature, despite the fact that it has been widely used for the counterpart of the Hermitian matrices; see the recent survey \cite{BHY} and the reference therein.  

After establishing these bounds for $\mathcal{X}_0$ and $\mathcal{X}$, we then turn to investigate the outliers of our model  $\mathcal{Y}$, created by the signal part $\mathcal{S}$. The asymptotic behaviour of the outliers of $\mathcal{Y}$ eventually boils down to the analysis of the quadratic form $u^*\underline{G}(\lambda)v$. From the estimate in (\ref{022710}) one can conclude Theorem \ref{thm: first order} via a standard argument in \cite{Tao}. The proof of the eigenvector result in Theorem \ref{thm: eigenvector} can be done similarly. 

Regarding the fluctuation, we find that $\lambda_i-d_i$ can be written as a linear combination of Green function quadratic forms of the form $u^*\underline{G}(\lambda)v$ and $\psi^*X_a\phi, a=1,2$ for various different choices of $u$, $v$, $\psi$ and $\phi$. A key fact is that $\psi^*X_a\phi$ is not necessarily asymptotic Gaussian, depending on the structure of $u$ and $v$, but $u^*\underline{G}(\lambda)v$ is asymptotically Gaussian.  Hence, we shall show both the asymptotic Gaussianity of the linear combinations of  $u^*\underline{G}(\lambda)v$'s and the asymptotic independence between it and the $\psi^*X_a\phi$ terms.  To this end, we turn to study the joint characteristic function of $u^*\underline{G}(\lambda)v$'s and $\psi^*X_a\phi$'s and study its limit. Again, the limiting characteristic function is obtained via the robust cumulant expansion approach. In contrast, even in the simpler case of iid non-Hermitian random matrix, the distribution of the outlier was previously obtained in \cite{Ra15, BC16} via a rather involved moment method. Our approach is much more straightforward, and robust under the general variance profile assumption, thanks to the inputs from \cite{AEKN18}.

We now outline the strategy when only a finite 2nd moment is assumed (the ``heavy-tailed'' case).
First, as in Theorem \ref{thm: first order} of the light-tail regime, we follow the argument in \cite{Tao}. We first start with the null case, i.e., $\mathcal{X}_0$ and study its spectral radius. Staring from the characteristic polynomial approach in \cite{BCG22},
we introduce a truncation of the matrix model and 
replace the original matrix by its truncated version,
which requires a comparison between the characteristic polynomials of these two matrices. The limiting behaviour of the characteristic polynomial of the truncated matrix is governed by the joint behaviour of the tracial quantities of the truncated $\mathcal{X}_0$. Instead of using a combinatorial argument for the tracial quantites in \cite{BCG22}, we can again rely on a very robust cumulant expansion approach, as we used in the heteroscedastic case for the quadratic forms.  After we obtain the spectral radius of $\mathcal{X}_0$, we then turn to study the outliers of $\mathcal{X}_0$, created by the signal $\mathcal{S}$.
A key technical hurdle is bounding the modulus of the characteristic polynomial of $\mathcal{X}$ both from below and above.
Especially for the upper bound, a key novel input is the following determinant identity
\begin{equation}\label{eq: 488}
	\det(\mathsf{A}+\mathsf{B}) = \sum_{\substack{I,J\subseteq\{1,\dots,n\} \\ |I|=|J|}}(-1)^{s(I,J)}\det([\mathsf{A}]_{I,J})\det([\mathsf{B}]_{I^{c},J^{c}}), \quad
	\mathsf{A},\mathsf{B}\in\mathbb{C}^{n},
\end{equation}
where \([W]_{I,J}:=(w_{ij})_{i\in I,\;j\in J}\) denotes the \((I,J)\)-submatrix of a
matrix \(W=(w_{ij})\), and \(s(I,J)\) is determined by the sign of the permutation that
first lists the indices in \(I\) followed by those in \(I^{c}\) and maps this
ordering to the analogous ordering determined by \(J\cup J^{c}\).
We find that the identity \eqref{eq: 488} is especially effective here
as long as the perturbation has both bounded rank and bounded operator norm.
The same strategy extends readily to the more classical model, that is, a low-rank perturbation of a non-Hermitian random matrix with i.i.d.~entries.

\subsection{Notation} Throughout this paper, we regard $n$ as our fundamental large parameter. Any quantities that are not explicit constant or fixed may depend on $n$; we often omit the argument $n$ from our notation.  We further use   $\|A\|_{\text{op}}$ for the operator norm of a matrix $A$.  We use $C$ or $K$ to denote some generic (large) positive constant.   For any positive integer $m$, let $\llbracket m\rrbracket$ denote the set $\{1,\ldots,m\}$. We use $\mathbf{1}$ to denote the all one vector, whose dimension is often apparent from the context, and thus is omitted from the notation. For any vector $u\in \mathbb{C}^m$, we denote $\langle u\rangle$ the average of its components.

\subsection{Organization} The rest of the paper is organized as follows. In Section \ref{s.simulation}, we propose an approach to detect the signal and discuss some simulation study.  In Section \ref{pf of first order}, we prove Theorem \ref{thm: first order}, based on Propositions \ref{pro.spectralradius} and \ref{quadraticformbound}. In Section \ref{Section: second order}, we prove Theorem \ref{thm: second order}, based on  Proposition \ref{jointCLT}, which will be proved in Section \ref{section. proof of CLT}. Theorem \ref{thm: first order opt} is proved in Section \ref{sec: 997}
{The proofs of Propositions \ref{pro.spectralradius}, \ref{quadraticformbound}, Theorem \ref{thm: out iid}  and some technical lemmas are stated in Appendices \ref{s.spectralradius}, \ref{s.proof of quadratic form bound}, \ref{sec: 2425}, and \ref{app.lemmas} respectively. In Appendix \ref{s eigenvector}, we state and prove the first order behaviour of the eigenvector projection for the heteroscedastic case.

\section{Simulation study} \label{s.simulation}

In this section, we present some simulation study, in order to compare the singular value approach based on the symmetric model  and the eigenvalue approach based on the asymmetric model.

We may apply our result in the following way. We define the following random domain, which is purely data based
\begin{align*}
D_N:=\Big\{ z\in \mathbb{C}: \Re z\geq \lambda_{\max}^{s}+N^{-1/2}\Big\}, 
\end{align*}
where we used the notation
\begin{align}
\lambda_{\max}^{s}:=\max_{i}|\lambda_i|\mathbf{1}\Big(\text{arg}(\lambda_i)\in [\frac{\pi}{\log N}, \frac{\pi}{2}]\Big). \label{042501}
\end{align}
Our detection criterion is as follows: If $\lambda_i\in D_N$, we identify it  as a signal . The rough idea is to check if a leading eigenvalue (in magnitude), $\lambda_i$,  is significantly away from the threshold. 
Here we use the purely data based $\lambda_{\max}^{s}$ as a random approximation of the limiting spectral radius, $\sqrt{n^{-1}\|T\|_{\text{op}}}$, which is unknown in reality. Such an approximation is possible since  the limiting spectral distribution of $\mathcal{X}_0$ is rotationally symmetric, which  is a consequence of the fact that our Matrix/Vector Dyson equation in (\ref{MDE2}) depends on $|z|^2$ only; see the discussion in Section \ref{s.spectralradius}. Hence, $\lambda_{\max}^s$ serves as a random approximation of the true threshold. Also notice that, according to our main results, the outliers are almost real, i.e, close to $d_i$'s. Hence, $\mathbf{1}\Big(\text{arg}(\lambda_i)\in [\frac{\pi}{\log N}, \frac{\pi}{2}]\Big)$ in the definition of $\lambda_{\max}^{s}$ can help us avoid the true signals in the definition of the approximate threshold. Our choice of the additional shift $N^{-1/2}$ is inspired by the recent result of the iid random matrix \cite{CEX23}. Especially, in \cite{CEX23}, it is known that the fluctuation of the spectral radius of iid random matrix (without deformations) is of order $o(N^{-1/2})$.  It is not clear if this fluctuation order still applies to our model with a block structure, also with a multiplicative perturbation $\Sigma$. But we expect that it should be still true at least  when $\sigma_{\max}\sim 1$. Certainly, if $d_i$ is sufficiently away from the threshold, by a constant order distance $\epsilon>0$, as we assumed in our main theorems, we do not have to choose a correction as delicate as $N^{-1/2}$ in the definition of $D_N$. Such a choice is mainly for the detection of those weak signals close to the threshold. The reason why we start from $\pi/\log N$ in the definition of $\lambda_{\max}^s$ is because that on fluctuation level, the outliers can indeed have nonzero imaginary part, when we have some multiple $d_i$'s or close $d_i$'s. But in any case, the fluctuation order is no larger than $N^{-\varepsilon}$ according to our assumption of $\sigma_{\max}$. Hence, the choice of the lower phase bound $\pi/\log N$ in (\ref{042501}) is enough to distinguish the true signals from the other eigenvalues.

 In the simulation study, we fix $(p, n)=(800, 2000)$. We consider two settings of variance profile $T$, 
\begin{align*}
T_1=\mathbbm{1}, \qquad T_2= (I_{p/2}\oplus 1.5 I_{p/2})\mathbbm{1}.
\end{align*}
We further consider the following two distribution types of $x_{ij}$: (i) $\sqrt{n}x_{ij}\sim N(0, t_{ij})$; (ii) $\sqrt{n}x_{ij}$ follows Student's t distribution with degree of freedom $\nu=2.2$, normalized to be mean 0 and variance $t_{ij}$. 
For the choices of $S$, we primarily consider the case of simple $d_i$, but we also perform simulation for multiple $d_i$. Specifically, we choose 
\begin{align*}
S= d_1e_3\tilde{e}_4^*+d_2e_4\tilde{e}_5^*+d_3e_5\tilde{e}_6^*
\end{align*}
and consider various choices of $(d_1,d_2,d_3)$. Here $e_i\in \mathbb{R}^p, \tilde{e}_j\in \mathbb{R}^n$ represent the standard basis of respective dimensions.   For $\Sigma$, we consider 
\begin{align*}
\Sigma=I+\sigma_1 e_1e_1^*+\sigma_2e_2e_2^*.
\end{align*}

All the figures are displayed in Appendix \ref{s.fig}.

In all figures, we simply call all singular values/eigenvalues which are close to the support of the limiting singular value/eigenvalue distributions the {\it singular values/eigenvalues}, and we call those away from the support of limiting distributions the {\it outlying singular values/ outlying eigenvalues}. We determine if an eigenvalue is an outlying eigenvalue by the criteria if $\lambda_i\in D_N$, and we also color the negative copy of it. In all eigenvalue figures, we denote by $\text{thre}_s$ the random threshold $\lambda_{\max}^s$. We determine if a singular value is an outlying one  by simply checking if it is away from the theoretical right end point of (general) Marchenko Pastur law.

In Fig \ref{fig:Gaussian_iid_SV}, we plot the singular values of $H$ under the choice  $(T, X)=(T_1, \text{Gaussian})$, and in Fig \ref{fig:Gaussian_iid_EV}, we plot the eigenvalues of $\mathcal{Y}$ (c.f. (\ref{model})) under the same choice of $(T,X)$. In both figures, we choose simple $d_i$'s. The results for multiple $d_i$'s are presented in Fig \ref{fig:Gaussian_iid_SV_multiple} and \ref{fig:Gaussian_iid_EV_multiple}. In Fig \ref{fig:Gaussian_general_SV} and \ref{fig:Gaussian_general_EV}, we consider the setting $(T, X)=(T_2, \text{Gaussian})$.  From these figures, we can clearly see that the eigenvalue approach can always detect the signals and also the locations of the outlying eigenvalues precisely tell the value of $d_i$ which are above the threshold. In contrast, 
the spiked $\Sigma$ can create additional outliers when one use the singular value approach, and thus it could be falsely detected as signals. 

%
%
%
%
%
%
%
%
%
%
%
%
%
%
%
%
%
%
%
%
%
%
%
%
%
%
%

In light of Theorem \ref{thm: first order opt}, we also present the simulation results in Fig. \ref{fig:Heavy_iid_SV} and Fig. \ref{fig:Heavy_iid_EV} under the choice $(T, X) = (T_1, \text{Student's t})$ with $\Sigma=I$ to illustrate the robustness of our approach in the heavy-tailed setting.  In this case, there are many outlying singular values due to the fatness of the distribution tail. We can also view them as spiked singular values, although they are not created by $\Sigma$. Detecting the signal from all these outliers seems impossible.   
In contrast, we observe that the eigenvalue approach remains robust in this case, as long as the 2nd moments of the matrix entries exist.  

%
%
%
%
%
%
%
%
%

\nc

\section{First order of the heteroscedastic case: Proof of Theorem \ref{thm: first order}} \label{pf of first order}
Recall the matrix model $\mathcal{Y}$ from (\ref{model}). Denote the eigendecomposition of $\mathcal{S}$ by 
\begin{align}\label{062704}
\mathcal{S}=W\mathfrak{D}W^*:=\sum_{i=\pm 1, \ldots, \pm k} d_i w_i w_i^*, 
\end{align}
where 
\begin{align*}
d_{\pm i}=\pm {d_i},  \quad w_{\pm i}=\frac{1}{\sqrt{2}}\binom{u_i}{\pm v_i}, \qquad i=1, \ldots, k.
\end{align*}
By considering the characteristic polynomial, if $\lambda$ is an eigenvalue of $\mathcal{Y}$ but not an eigenvalue of $\mathcal{X}$, we can derive from 
\begin{align*}
\det\Big(\mathcal{X}+W\mathfrak{D}W^*-\lambda\Big)=\det\big(\mathcal{X}-\lambda\big)\det\Big(I+\mathfrak{D}W^*(\mathcal{X}-\lambda)^{-1}W\Big)=0
\end{align*} 
the following identity
\begin{align}
\det\Big(I+\mathfrak{D}W^*(\mathcal{X}-\lambda)^{-1}W\Big)=0. \label{102901}
\end{align} 
We further write
\begin{align}
(\mathcal{X}-\lambda)^{-1}=\left(
\begin{array}{ccc}
\lambda G_\sigma(\lambda^2)&\Sigma X_1\mathcal{G}_\sigma(\lambda^2)\\
 \mathcal{G}_\sigma(\lambda^2) X_2^*\Sigma^* & \lambda \mathcal{G}_\sigma(\lambda^2)
\end{array}
\right), \label{021803}
\end{align}
where 
\begin{align*}
G_\sigma(z):= (\Sigma X_1X_2^*\Sigma^*-z)^{-1}, \qquad \mathcal{G}_\sigma(z):=( X_2^*\Sigma^*\Sigma X_1-z)^{-1}.
\end{align*}
We further set 
 \begin{align}
 G(z):=\left( X_{1} X_{2}^{*} -z\right)^{-1}, \qquad  \mathcal{G}(z):=\left(X_{2}^{*}  X_{1}-z\right)^{-1}. \label{070701}
 \end{align} 
 We denote by $\rho(A)$ the spectral radius of a square matrix $A$.  Recall the notion ``with high probability" from Definition \ref{stochastic dom}.  We have the following proposition regarding the spectral radius and Green functions of $X_{1} X_{2}^{*}$ and $\Sigma X_1X_2^*\Sigma^*$.

\begin{pro}\label{pro.spectralradius} Under Assumption \ref{main assump}, for any small constant $\delta>0$, we have 
\begin{align*}
\rho(X_{1} X_{2}^{*})\leq n^{-1}\|T\|_{\text{op}}+\delta,  \qquad \rho(\Sigma X_1X_2^*\Sigma^*)\leq n^{-1}\|T\|_{\text{op}}+\delta
\end{align*}
with high probability.  Further, uniformly in $z$ with $|z|\geq n^{-1}\|T\|_{\mathrm{op}}+\delta$, we have for some large constant $C>0$, 
\begin{align*}
\| G(z)\|_{\text{op}}, \|G_{\sigma}(z)\|_{\text{op}}, \|\mathcal{G}(z)\|_{\text{op}}, \|\mathcal{G}_{\sigma}(z)\|_{\text{op}}\leq C
\end{align*}
with high probability. 
\end{pro}
The proof of the above proposition will be stated in Appendix \ref{s.spectralradius}.

We then introduce the centered Green functions 
\begin{align}
\underline{G}(z):=G(z)+z^{-1}, \qquad \underline{\mathcal{G}}(z):=\mathcal{G}(z)+z^{-1}. \label{def of G under}
\end{align}
We set for any small but fixed $\varepsilon>0$.
\begin{align}
q_n\equiv q_n(\varepsilon)=:n^{\frac12-\varepsilon}.  \label{def of q}
\end{align}

Here $\varepsilon$ shall always be chosen to be sufficiently small according to $\varepsilon_0$ in Assumption \ref{main assump} of (iii). 

The following proposition provides the estimates of the quadratic forms of centered $G$ and $\mathcal{G}$,  which will be one of our main technical inputs. 

\begin{pro} \label{quadraticformbound} Under Assumption \ref{main assump}, 
for any $|z|\geq n^{-1}\|T\|_{\mathrm{op}}+\delta$, and any deterministic vectors $u, v\in S_{\mathbb{C}}^{p-1}$ and $\psi, \phi\in S_{\mathbb{C}}^{n-1}$, we have 
\begin{align}
&\langle u, \underline{G}(z) v \rangle=O_{\prec} (q_n^{-1}), \qquad  \langle \psi, \underline{\mathcal{G}}(z) \phi \rangle=O_{\prec} (q_n^{-1})\notag\\
&\langle u, X_1\mathcal{G}(z)\phi\rangle=O_\prec(q_n^{-1}), \qquad \langle \psi, \mathcal{G}(z) X_2^* v\rangle=O_\prec(q_n^{-1}).  \label{011510}
\end{align}
\end{pro}
The proof of Proposition \ref{quadraticformbound} will be deferred to Appendix \ref{s.proof of quadratic form bound}. We then proceed with the proof of Theorem \ref{thm: first order}.

Notice that $\Sigma^{-1}$ is also a low rank perturbation of $I$. 
We then first  write
\begin{align*}
    G_{\sigma}(z)
     =(\Sigma^*)^{-1} \left[ X_{1} X_{2}^{*} -z+z\left(I- (\Sigma^*\Sigma)^{-1} \right)\right]^{-1} \Sigma^{-1}
   =:   G(z) + \mathcal{E}(z).
\end{align*}
The low rank matrix $I-(\Sigma^*\Sigma)^{-1}$  admits a spectral decomposition 
\begin{align}
I-(\Sigma^*\Sigma)^{-1}=: L\Gamma L^*, \label{032401}
\end{align}  
where $\Gamma$ is a low rank diagonal matrix consisting of the non-zero eigenvalues of $I-(\Sigma^*\Sigma)^{-1}$. Then, by  Woodbury matrix identity, we can write 
\begin{align*}
&\left[ X_{1} X_{2}^{*} -z+z\left(I- (\Sigma^*\Sigma)^{-1} \right)\right]^{-1}=\left[ X_{1} X_{2}^{*} -z+zL\Gamma L^*\right]^{-1}\notag\\
& \qquad =G-zGL\big[I+z\Gamma L^* GL\big]^{-1}\Gamma L^*G=G-zGL\big[I-\Gamma+z\Gamma L^* \underline{G}L\big]^{-1}\Gamma L^*G ,
\end{align*}
which gives 
\begin{align*}
    \mathcal{E}(z) &= \left((\Sigma^*)^{-1}-I\right)G(z) + G(z)\left(\Sigma^{-1}-I\right) + \left((\Sigma^*)^{-1}-I\right) G(z)\left(\Sigma^{-1}-I\right) \notag \\
    &\qquad -z(\Sigma^*)^{-1}GL\left[I - \Gamma + z\Gamma L^*\underline{G}L\right]^{-1}\Gamma L^*G\Sigma^{-1}.
\end{align*}
We remark here that if we replace $G(z)$ by $-1/z$ in the definition of ${\mathcal{E}}(z)$, it gives $0$. Similarly, we can also write 
\begin{align*}
    \mathcal{G}_{\sigma}(z) &=
     \mathcal{G}(z) - \mathcal{G}(z)X_2^*L\Lambda (I + zL^*\underline{G}L\Lambda)^{-1}L^*X_1\mathcal{G}(z)  =: \mathcal{G}(z)+\mathcal{F}(z),
\end{align*}
where $\Lambda = (I-\Gamma)^{-1} - I$. Hence,  with the above notations, we can rewrite (\ref{021803}) as
\begin{align}
(\mathcal{X}-\lambda)^{-1}= \left(
\begin{array}{ccc}
\lambda G(\lambda^2) & \Sigma X_1\mathcal{G}(\lambda^2) \\
 \mathcal{G}(\lambda^2) X_2^*\Sigma^*   & \lambda \mathcal{G}(\lambda^2)
\end{array}
\right) 
+\left( \begin{array}{ccc}
\lambda \mathcal{E}(\lambda^2)& \Sigma X_1\mathcal{F}(\lambda^2) \\
\mathcal{F}(\lambda^2)X_2^*\Sigma^* & \lambda\mathcal{F}(\lambda^2). 
\end{array}\right). \label{021802}
\end{align}

Recall the definition of $\sigma_{\max}$ from (\ref{def of sigma_max}). 
Further, we introduce the notation $\mathsf{R}_i(z), i =1,2$ to include all the  matrices in the remainder terms after applying expansion, which satisfy that for any given unit vectors $u, v$
\begin{align}
    \langle u, \mathsf{R}_i(z) v \rangle = O_{\prec}(q_n^{-i}\sigma_{max}^{2i}). \label{021801}
\end{align}

By our assumption on $\Sigma$, it is easy to check $\sigma_{\max}^2 \leq Cn^{1/2-\varepsilon_0}$ for some $\varepsilon_0>\varepsilon>0$. Hence, the eigenvalues of $I-\Gamma$ are no smaller than $n^{-1/2+\varepsilon_0}$. Meanwhile, the entries of $ L^* \underline{G}L$ are $O_\prec(q_n^{-1})$ according to Proposition \ref{quadraticformbound}. Consequently,  $\|(I - \Gamma)^{-1}\Gamma L^*\underline{G}L\|_{\text{op}}= o(1)$ with high probability, we can thus express $\big[I-\Gamma+z\Gamma L^* \underline{G}L\big]^{-1}$ as a Neumann series and bound the remainder terms, i.e. 
\begin{align}\label{Neumann exp}
    &z(\Sigma^*)^{-1}GL\big[I-\Gamma+z\Gamma L^* \underline{G}L\big]^{-1}\Gamma L^*G\Sigma^{-1} 
    \notag \\ &\quad = z(\Sigma^*)^{-1}G(\Sigma^*\Sigma)\left(I - (\Sigma^*\Sigma)^{-1}\right)G\Sigma^{-1} \notag \\
    & \qquad -z^2(\Sigma^*)^{-1}G(\Sigma^*\Sigma)\left(I - (\Sigma^*\Sigma)^{-1}\right)\underline{G}(\Sigma^*\Sigma)\left(I - (\Sigma^*\Sigma)^{-1}\right)G\Sigma^{-1} + \mathsf{R}_2(z),
\end{align}
where $\mathsf{R}_2(z)$ is defined in (\ref{021801}). 

With the above expansion and Proposition \ref{quadraticformbound}, we can expand $G(\lambda^2)$ and $\mathcal{G}(\lambda^2)$ around $-1/\lambda^2$ in (\ref{021802}) and obtain
\begin{align}
\left(W^*(\mathcal{X}-\lambda)^{-1}W\right) &=-\frac{1}{\lambda}+W^* \left(
\begin{array}{ccc}
\lambda \underline{G}(\lambda^2) & \Sigma X_1\mathcal{G}(\lambda^2) \\
 \mathcal{G}(\lambda^2) X_2^*\Sigma^* & \lambda \underline{\mathcal{G}}(\lambda^2)
 \end{array}
 \right)W
  \notag\\
 &+ W^*\left(
\begin{array}{ccc}
\lambda \underline{\mathcal{E}}(\lambda^2) & 0 \\
0  & 0
\end{array}
\right) W+ W^*\mathsf{R}_2(\lambda^2)W, \label{102903}
\end{align}
where 
\begin{align}\label{def of E underline}
    \underline{\mathcal{E}}(z) = & A^*\underline{G} + \underline{G}A + A^*\underline{G}A + B^*\underline{G}\Sigma^{-1} + (\Sigma^*)^{-1}\underline{G}B + B^* \underline{G}B + \mathsf{R}(z),
\end{align}
\begin{align*}
    \mathsf{R}(z) 
    &= -z(\Sigma^*)^{-1} \underline{G}B\Sigma \underline{G}B\Sigma \underline{G}\Sigma^{-1} - z(\Sigma^*)^{-1}\underline{G}B\Sigma \underline{G}\Sigma^{-1} - z(\Sigma^*)^{-1}\underline{G}B\Sigma \underline{G}B \notag \\
    &\qquad - zB^*\underline{G}B\Sigma \underline{G}\Sigma^{-1} +\mathsf{R}_2(z),
\end{align*}
Here we introduced the shorthand notations
\begin{align*}
	A \deq \Sigma^{-1} - I, \qquad B \deq (\Sigma^* \Sigma) (I - (\Sigma^* \Sigma)^{-1}) \Sigma^{-1} = \Sigma^* - \Sigma^{-1}.
\end{align*}
Note that Proposition \ref{quadraticformbound} and (\ref{Neumann exp}) imply that the first six terms in  (\ref{def of E underline}) are $\mathsf{R}_1(z)$, while $\mathsf{R}(z)$ is an $\mathsf{R}_2(z)$. Then,  we have
\begin{align}
\left(W^*\Big[(\mathcal{X}-\lambda)^{-1} + \frac{1}{\lambda} + \frac{\mathcal{X}}{\lambda^2}\Big]W\right) &= W^* \left(
\begin{array}{ccc}
\lambda \underline{G}(\lambda^2) & \Sigma X_1 \underline{\mathcal{G}}(\lambda^2)) \\
\underline{\mathcal{G}}(\lambda^2) X_2^*\Sigma^* & \lambda \underline{\mathcal{G}}(\lambda^2)
 \end{array}
 \right)W
  \notag\\
 &+ W^*\left(
\begin{array}{ccc}
\lambda \underline{\mathcal{E}}(\lambda^2)& 0 \\
0  & 0
\end{array}
\right) W + W^*\mathsf{R}_2(\lambda^2)W, 
\end{align}
 With Proposition \ref{quadraticformbound}, one can follow the argument in \cite{Tao} to show that 
\begin{align*}
\lambda_{\pm i}(\mathcal{Y})=\pm \lambda_i(\mathcal{Y})=\pm d_{i}+O_{\prec}(q_n^{-1}\sigma_{\max}^2). 
\end{align*}
This concludes the proof  of Theorem \ref{thm: first order}.

\section{Second order of the heteroscedastic case: Proof of Theorem \ref{thm: second order}}\label{Section: second order}
In this section, we take a step further to study the fluctuation of $\lambda_i$. 
Based on (\ref{102901}) and Proposition \ref{pro.spectralradius}, we further expand $\lambda_i$ around $d_i$, 
\begin{align}\label{102902}
0 &=\det\Big(I+\mathfrak{D}W^*(\mathcal{X}-\lambda)^{-1}W\Big)\notag\\
&=\det\Big(I+\mathfrak{D}W^*(\mathcal{X}-d_i)^{-1}W+(d_i - \lambda_i)\mathfrak{D}W^*(\mathcal{X}-d_i)^{-2}W+O(|\lambda_i-d_i|^2)\Big)\notag\\
&= \det\Big(I+\mathfrak{D}W^*(\mathcal{X}-d_i)^{-1}W+(d_i - \lambda_i)d_i^{-2}\mathfrak{D}+O(|\lambda_i-d_i|q_n^{-1}\sigma_{\max}^2)\Big)\notag\\
&=\Big[1+d_iw_i^*(\mathcal{X}-d_i)^{-1}w_i+(d_i - \lambda_i)d_i^{-1}\Big]\prod_{j\neq i}\Big(1-\frac{d_j}{d_i}\Big)\notag\\
&\qquad+O(|\lambda_i-d_i|q_n^{-1}\sigma_{\max}^2)+O(q_n^{-2}\sigma_{\max}^2\|u_i^*\Sigma\|_2^2), 
\end{align}
where in the third step the estimate of $W^*(\mathcal{X}-d_i)^{-2}W$ follows simply from that of $W^*(\mathcal{X}-d_i)^{-1}W$ by applying a contour integral around $d_i$, and the last step follows from the expansion of the determinant and fact that the contribution from off-diagonal entries is small. For simplicity, we will write the error term in (\ref{102902}) as 
\begin{align*}
\mathcal{E}_i:=O(|\lambda_i-d_i|q_n^{-1}\sigma_{\max}^2)+O(q_n^{-2}\sigma_{\max}^2\|u_i^*\Sigma\|_2^2). 
\end{align*}
In the sequel, we will also use $\mathcal{E}_i$ to denote any generic term of the above order. 

From (\ref{102902}), we have 
\begin{align}
\lambda_i-{d}_i &= {d}_i\left(1+ {d}_iw_i^*(\mathcal{X}-{d}_i)^{-1}w_i \right)+\mathcal{E}_i\notag\\
&= {d}_i^2 w_i^*\Big[(\mathcal{X}-{d}_i)^{-1} + \frac{1}{{d}_i} + \frac{\mathcal{X}}{{d}_i^2}\Big]w_i - w_i^* \mathcal{X} w_i+\mathcal{E}_i. \label{010701}
\end{align}
By further using (\ref{102903}), we have 
\begin{align}
\lambda_i-{d}_i =& {d}_i^2 w_i^* \left(
\begin{array}{ccc}
{d}_i \underline{G}({d}_i^2) &  \Sigma X_1 \underline{\mathcal{G}}({d}_i^2) \\
\underline{\mathcal{G}}({d}_i^2) X_2^*\Sigma^* & {d}_i \underline{\mathcal{G}}({d}_i^2)
 \end{array}
 \right)w_i +{d}_i^2 w_i^*\left(
\begin{array}{ccc}
{d}_i \underline{\mathcal{E}}({d}_i^2)& 0 \\
0  & 0
\end{array}
\right)w_i - w_i^* \mathcal{X} w_i
+\mathcal{E}_i. \label{approx eigenvalue}
\end{align}
Hence, our main task is to establish the joint distribution of the following quadratic forms
$$ u^*\underline{G}(z)v, \quad \psi^*\underline{\mathcal{G}}(z)\phi, \quad s^* X_{1} \underline{\mathcal{G}}(z) \zeta, \quad r^* \underline{\mathcal{G}}(z) X_{2}^{*} \eta $$
and also the linear (in $\mathcal{X}$) term $w_i^* \mathcal{X} w_i$. 
We then have the following key proposition.
\begin{pro}\label{jointCLT} Let $|z|\geq n^{-1}\|T\|_{\mathrm{op}}+\delta$ for any small (but fixed) $\delta>0$.  Let $m_i, i=1,\dots, 4$ be fixed positive integers. For any collection of unit vectors $u_i, v_i, \psi_j, \phi_j, q_k, \gamma_k, \eta_\ell, r_\ell$ with $i\in \llbracket m_1\rrbracket, j\in \llbracket m_2\rrbracket, k\in \llbracket m_3\rrbracket, \ell \in \llbracket m_4\rrbracket$, the collection of random variables 
\begin{align}
&\Big\{\sqrt{n}u_i^* \underline{G}(z) v_i,  \sqrt{n}\psi_j^*\underline{\mathcal{G}}(z)\phi_j,  \sqrt{n}q_k^*  X_1\underline{\mathcal{G}}(z) \gamma_k,  \sqrt{n}\eta_\ell^* \underline{\mathcal{G}}(z)X_2^* r_\ell: \notag\\
&\hspace{10ex}i\in \llbracket m_1\rrbracket, j\in \llbracket m_2\rrbracket, k\in \llbracket m_3\rrbracket, \ell \in \llbracket m_4\rrbracket \Big\} \label{collection of qf}
\end{align}
converges to the collection of jointly Gaussian variables
\begin{align*}
\{\mathcal{A}_i, \mathcal{B}_j, \mathcal{C}_k, \mathcal{D}_\ell: i\in \llbracket m_1\rrbracket, j\in \llbracket m_2\rrbracket, k\in \llbracket m_3\rrbracket, \ell \in \llbracket m_4\rrbracket\}
\end{align*}
with mean 0 and  the covariance structure given by 
\begin{align}
&\mathrm{Cov}(\mathcal{A}_i, \mathcal{A}_j)=\frac{1}{n|z|^4}\sum_{\alpha,\beta} u_{i\alpha}{u_{j\alpha}}v_{i\beta}{v}_{j\beta}\vec{T}_{\alpha\cdot}^\top\Big[I-\frac{1}{n^2|z|^2} T^*T\Big]^{-1} \vec{T}_{\beta\cdot}, \notag\\
& \mathrm{Cov}(\mathcal{B}_i, \mathcal{B}_j)=\frac{1}{n|z|^4}\sum_{\alpha,\beta} \psi_{i\alpha}{\psi_{j\alpha}}\phi_{i\beta}{\phi}_{j\beta}\vec{T}_{\cdot \alpha}^\top\Big[I-\frac{1}{n^2|z|^2} TT^*\Big]^{-1} \vec{T}_{\cdot \beta}, \notag\\
&\mathrm{Cov}(\mathcal{C}_i, \mathcal{C}_j)=\frac{1}{n^2|z|^4} \sum_{\alpha,\beta} q_{i\alpha}q_{j\alpha}\gamma_{i\beta}\gamma_{j\beta}\Bigg(TT^* \Big[I-\frac{1}{n^2|z|^2} TT^*\Big]^{-1}\vec{T}_{\cdot \beta}\Bigg)_{\alpha} , \notag\\
& \mathrm{Cov}(\mathcal{D}_i, \mathcal{D}_j)=\frac{1}{n^2|z|^4} \sum_{\alpha,\beta} r_{i\alpha}r_{j\alpha}\eta_{i\beta}\eta_{j\beta}\Bigg(TT^* \Big[I-\frac{1}{n^2|z|^2} TT^*\Big]^{-1}\vec{T}_{\cdot \beta}\Bigg)_{\alpha}. \label{limitgaussian}
\end{align}
The collections $\{\mathcal{A}_i\}, \{\mathcal{B}_j\}, \{\mathcal{C}_k\}, \{\mathcal{D}_\ell\}$ are mutually independent.
Further, the collection (\ref{collection of qf}) is asymptotically independent of any collection of finitely many linear terms of the form $\sqrt{n}a^* X_i b, i=1,2$ for any deterministic unit vectors $a, b$. 
\end{pro}

The proof of Proposition \ref{jointCLT} will be stated in Section \ref{section. proof of CLT}. Now, we proceed with the proof of Theorem  \ref{thm: second order}.  
 It amounts to the estimate of the variance of the Gaussian part in  (\ref{approx eigenvalue}), as we have already shown the asymptotic independence between the Gaussian part and $w_i^* \mathcal{X} w_i$ in Proposition \ref{jointCLT}. 
Notice that we have 
\begin{align}
    \lambda_i - d_i &= d_i^2(d_i u_i^*\underline{G}(d_i^2)u_i + u_i^*\Sigma X_1\underline{\mathcal{G}}(d_i^2)v_i + v_i^*\underline{\mathcal{G}}(d_i^2)X_2^*u_i + d_iv_i^*\underline{\mathcal{G}}(d_i^2)v_i + d_iu_i^*\underline{\mathcal{E}}(d_i^2)u_i) \notag \\
    & \qquad - u_i^*\Sigma X_1 v_i - v_i^* X_2^*\Sigma^* u_i + \mathcal{E}_i, \label{010703}
\end{align}
where $\underline{\mathcal{E}}$ is defined in (\ref{def of E underline}). 
A simplification leads to 
\begin{align}
    u_i^*\underline{\mathcal{E}}(d_i^2)u_i &= (\Sigma^{-1}u_i-u_i)^*\underline{G}(d_i^2)u_i + u_i^*\underline{G}(d_i^2)(\Sigma^{-1}u_i-u_i) + (\Sigma^{-1}u_i-u_i)^*\underline{G}(d_i^2)(\Sigma^{-1}u_i-u_i)  \notag \\ 
    & \qquad + (\Sigma^*u_i-\Sigma^{-1}u_i)^*\underline{G}(d_i^2)\Sigma^{-1}u_i + \Sigma^{-1}u_i^*\underline{G}(d_i^2)(\Sigma^*u_i-\Sigma^{-1}u_i) \notag \\
    & \qquad + (\Sigma^*u_i-\Sigma^{-1}u_i)^*\underline{G}(d_i^2)(\Sigma^*u_i-\Sigma^{-1}u_i) + \mathcal{E}_i \notag \\
    &= u_i^*\Sigma\underline{G}(d_i^2)\Sigma^*u_i -u_i^*\underline{G}(d_i^2)u_i + \mathcal{E}_i. \label{041501}
\end{align}
The Gaussian part of (\ref{010701}),  comes from the first line of (\ref{010703}). For notational simplicity, we set 
\begin{align*}
   & M(T, \alpha, \beta) = \vec{T}_{\alpha \cdot} \left[I-\frac{1}{n^{2}|z|^{2}}\left(T^*T\right)\right]^{-1}\vec{T}_{\beta \cdot} , \notag
    \\
  &  N(TT^*, \alpha, \beta) = \Bigg(TT^*\left[I-\frac{1}{n^{2}|z|^{2}}\left(TT^*\right)\right]^{-1}\vec{T}_{\cdot \beta}\Bigg)_{\alpha} , \notag \\
   & V(pq^*, rs^*)_{\alpha, \beta} = (pq^*)_{\alpha\alpha}(rs^*)_{\beta\beta}, 
\end{align*}
for some vector $p, q, r, s \in \mathbb{R}^p$. Combining (\ref{010703}), (\ref{041501}), and (\ref{limitgaussian}) in Proposition \ref{jointCLT}, we can conclude the proof of Theorem \ref{thm: second order}.

\section{Gaussianity of Green function quadratic forms: Proof of Proposition \ref{jointCLT}} \label{section. proof of CLT}
In this section, we prove Proposition \ref{jointCLT}, based on Propositions \ref{pro.spectralradius} and \ref{quadraticformbound}.  For brevity, we consider the following linear combination of the quadratic forms
\begin{equation}\label{Q def}
    \begin{aligned}
        Q &\deq \sqrt{n}\left(zu^*\underline{G}(z)v +  z\psi^*\underline{\mathcal{G}}(z)\phi +  q^*  X_{1} \underline{\mathcal{G}}(z) \gamma + \eta^* \underline{\mathcal{G}}(z) X_{2}^{*} r \right) \\
    &= \sqrt{n}\left( u^* X_1X_2^* G(z)v + \psi^* X_2^*X_1 \mathcal{G}(z)\phi +  q^*  X_{1} \underline{\mathcal{G}}(z) \gamma + \eta^* \underline{\mathcal{G}}(z) X_{2}^{*} r \right)
    \end{aligned}  
\end{equation}
In order to prove Proposition \ref{jointCLT}, we shall consider an arbitrary linear combination of all the quadratic forms in (\ref{collection of qf}). Nevertheless, the following derivation for the CLT of the above $Q$ is sufficient to show the mechanism. The generalization is straightforward. After establishing the CLT for $Q$, we will further comment on how to involve the linear term of $\mathcal{X}$ and show the asymptotic independence simultaneously. 

We further define the smooth cutoff function 
\begin{align}
\chi_G \equiv \chi(n^{-1}\text{Tr} GG^*), \label{011502}
\end{align}
where $\chi(t)$ is a smooth cutoff function which takes $1$ when $|t|\leq K$  for some sufficiently large constant $K>0$ and takes $0$ when $|t|\geq 2K$ and interpolates smoothly in between so that $|\chi^{(k)}(t)|\leq C_k$ for some positive constant $C_k$ for any fixed integer $k\geq 0$.  By the upper bound of $\|G\|_{\text{op}}$ in Propositions \ref{pro.spectralradius}, we have $|n^{-1}\text{Tr} GG^*|\leq K$ w.h.p.,  when $K$ is large enough. Hence, $\chi_G=1$ w.h.p. and $\chi^{(k)}_G\equiv \chi^{(k)}(n^{-1}\text{Tr} GG^*)=0 $ w.h.p. for any fixed $k\geq 1$. Further, we have the deterministic bound 
\begin{align}
\|G\|_{\text{op}}\chi^{(k)}_G\leq 2C_kKn, \qquad \text{for any fixed  } k\geq 0.  \label{101501}
\end{align}
Note that $Q=Q\chi_G$ w.h.p. Hence, it suffices to establish the CLT for $Q\chi_G$ instead.

Our aim is to establish
\begin{align}
\mathbb{E} f'(t)=-\mathfrak{d}^2t \mathbb{E} {f(t)}+o(1), \qquad \text{  where   } {f(t)}=\exp(\ii tQ\chi_G) \label{010322}
\end{align}
for some $\mathfrak{d}>0$.  We will use the cumulant expansion formula \cite[Lemma 1.3]{HK17} to establish the above equation.
The reason why we include the $\chi_G$ factor is to make sure that $G$-factors in the derivation has a deterministic upper bound so that we can take expectations. 

In the sequel, for simplicity, we omit $z$ from the notations of the Green functions. We further write
\begin{align}
 \mathbb{E} f'(t) & = \ii \mathbb{E} Q\chi_G{f(t)} 
= \ii \sqrt{n}\mathbb{E}  \left( u^* X_1X_2^* Gv +  \psi^* X_2^*X_1 \mathcal{G}\phi +  q^*  X_{1} \underline{\mathcal{G}} \gamma + \eta^* \underline{\mathcal{G}} X_{2}^{*} r \right)\chi_G{f(t)} \notag\\
&=: I+II+III+IV \label{010301}
\end{align}
In the sequel, we show the estimate of the first term $I$ in details, and the other terms can be estimated similarly. We write via cumulant expansion \cite[Lemma 1.3]{HK17}
\begin{align}
&I=\ii \sqrt{n}\mathbb{E}  u^* X_1X_2^* Gv f(t) =\ii \sqrt{n} \sum_{ij} \mathbb{E} x_{1,ij} \Big[(X_2^* G vu^*)_{ji} \chi_Gf(t)\Big]\notag\\
&\qquad\qquad= \ii \sqrt{n} \sum_{\alpha=1}^m \sum_{ij} \frac{\kappa_{\alpha+1}(x_{1,ij})}{\alpha !} \mathbb{E} \partial_{1,ij}^{\alpha} \Big[(X_2^*Gvu^*)_{ji} \chi_Gf(t)\Big]+\mathcal{R} \label{010723}
\end{align} 
where $\mathcal{R}=O_\prec(n^{-C})$ for a large constant $C>0$ when $m$ is large enough. The estimate can be done similarly to (\ref{remainderterm}).  

As the probability for the $\chi_G\neq 1$ or $\chi_G^{(k)}\neq 0$ for $k\geq 1$ is extremely small, any term involving the derivatives of $\chi_G$ can be neglected easily in the cumulant expansion. Hence, in the sequel, we will focus on the derivatives of other factors, and always include the terms involving the $\chi_G$ derivatives into the error terms without further explanation.

The first order ($\alpha=1$) term in (\ref{010723}) reads 
\begin{align}
&\frac{\ii}{\sqrt{n}}  \sum_{ij} \mathbb{E}\partial^1_{1,ij} \Big[ (X_2^* G v u^* )_{ji}{\chi_Gf(t)} \Big] \notag\\
&=\frac{\ii}{ \sqrt{n}} \sum_{ij} t_{ij} \mathbb{E} \Big[ \Big(-(X_2^*G)_{ji}(X_2^* Gvu^*)_{ji}+\ii t \sqrt{n} (X_2^* G v u^* )_{ji}\Big( -zu^*GE_{ij} X_2^* G v 
 \notag \\
&\qquad -z\psi^*\mathcal{G}X_2^* E_{ij}\mathcal{G}\phi+ q^* E_{ij} \underline{\mathcal{G}}\gamma  -q^* X_1\mathcal{G} X_2^* E_{ij} \mathcal{G}\gamma-\eta^* \mathcal{G} X_2^* E_{ij} \mathcal{G} X_2^* r\Big)\Big)\chi_G f(t)\Big] \notag\\
&=:-t\sum_{ij} t_{ij} \mathbb{E} \Big[ \Big((X_2^* G v u^* )_{ji}u^* E_{ij} X_2^* G v\Big)\chi_Gf(t) \Big]+\mathcal{R}_1.\notag\\ \label{010722}
\end{align}
Here $E_{ij}=(\delta_{ki}\delta_{\ell j})_{k,\ell}\in \mathbb{R}^{p\times n}$, i.e., $E_{ij}=e_i \tilde{e}_j^*$, where $e_i\in \mathbb{R}^p, \tilde{e}_j\in \mathbb{R}^n$ represent the standard basis of respective dimensions.  
We claim that in the RHS of the above equation, the first term is the leading term and $\mathcal{R}_1$ is the remainder term that is negligible. More precisely, with the aid of Proposition \ref{quadraticformbound} and Proposition \ref{prop_upper_bound}, we can easily show  
\begin{align}
	\mathcal{R}_1= O_\prec(q_n^{-1}). \label{010303}
\end{align}

We then proceed with showing that all higher order ($\alpha\geq 2$) terms in (\ref{010723}) are negligible. Notice that the second order term is proportional to  
\begin{align}\label{2nd order exp}
   &\sqrt{n} \sum_{ij} \kappa_3(x_{1,ij})\partial^2_{1,ij} \Big[ (X_2^* G v u^* )_{ji}\chi_G f(t) \Big] \notag\\
   &= \sqrt{n}\sum_{ij}  \kappa_3(x_{1,ij})\Big[\partial^2_{1,ij} (X_2^* G v u^* )_{ji} + 2\ii t\partial^1_{1,ij} (X_2^* G v u^* )_{ji} \partial^1_{1,ij}Q \notag \\ 
    &+ \ii t (X_2^* G v u^* )_{ji}\partial^2_{1,ij} Q - t^2 (X_2^* G v u^* )_{ji}(\partial^1_{1,ij} Q)^2\Big]\chi_G f(t) + \mathcal{R}_1 \notag \\
    & = (1) + (2) + (3) + (4) + \mathcal{R}_1,
\end{align}
Using \eqref{quadraticformbound} and the fact that $\kappa_3(x_{1,ij})=O(n^{-\frac32})$, one can simply bound (1) by
\begin{align*}
    |(1)| \leq Cn^{-1}\sum_{ij}  |(X_2^* G)_{ji}|^2  |\tilde{e}_j^* X_2^* G v||u_i| \overset{C.S.}{\leq} \sqrt{n} q_n^{-3},
\end{align*}
As the other terms in (\ref{2nd order exp}) all involve the derivative of $Q$, we first derive  
\begin{equation}\label{Q 1st der}
    \begin{aligned}
        \partial^1_{1,ij}Q &= \sqrt{n}\Big[u^* E_{ij} X_2^*  G v - u^* X_1X_2^* G E_{ij} X_2^*Gv \\
    &+  \psi^* X_2^*  E_{ij} \mathcal{G} \phi -  \psi^* X_2^*  X_1 \mathcal{G}X_2^*E_{ij} \mathcal{G} \phi \\
    &+ q^* E_{ij} \underline{\mathcal{G}}\gamma - q^* X_1\mathcal{G}X_2^* E_{ij} \mathcal{G}\gamma- \eta^* \mathcal{G}X_2^* E_{ij} X_2^* \mathcal{G}r \Big].
    \end{aligned}
\end{equation}
It is easy to obtain $|\partial^1_{1,ij}Q| = O(\sqrt{n} q_n^{-1})$ by Proposition \ref{quadraticformbound}.
Then for (2), by extracting  an factor $(X_2^* G)_{ji}$ from $\partial^1_{1,ij} \Big(X_2^* G v u^* \Big)_{ji}$, we have 
\begin{align*}
    |(2)| &\leq Cn^{-1}\sum_{ij} |(X_2^* G)_{ji}| \Big|\Big(X_2^* G v u^* \Big)_{ji} \partial^1_{1,ij} Q \Big|\notag\\
    &  \leq C n^{-1/2} q_n^{-1}\sum_{ij} |(X_2^* G)_{ji}| |(X_2^* G v u^* )_{ji}| \notag\\
    &\leq Cn^{-1/2} q_n^{-1} \sqrt{\text{Tr} X_2^*GG^*X_2\cdot v^*G^* X_2X_2^* G v}\prec q_n^{-1},
\end{align*}
where in the last two steps we used Cauchy Schwarz inequality and the fact that the operator norms of the matrices are $O_{\prec}(1)$. For term (3),we observe that in comparison to compare to $\partial^1_{1,ij} Q $, $\partial^2_{1,ij} Q$ will bring each term an additional factor,  which can be bounded crudely by O(1). Furthermore, in each term of $\partial_{1,ij}^1 Q$, there is a factor of the form $\theta^* e_i$ for some vector $\theta$ with $\|\theta\|_2\prec 1$, and another factor of the form $\tilde{e}_j^* \xi$ for some vector $\xi$ with $\|\xi\|_2\prec 1$. We further write $(X_2^* G vu^*)_{ji}=\tilde{e}_j^* X_2^*Gv\cdot u^*_i$, and apply Cauchy Schwarz inequality for the $i$-sum and $j$-sum respectively, one can easily get
\begin{align*}
    |(3)| \leq Cn^{-1} \sum_{ij} \Big|\Big(X_2^* G v u^* \Big)_{ji} \partial^1_{1,ij} Q \Big| \leq C n^{-1} \sum_{ij} \Big|u^*_i \theta^*e_i \tilde{e}_j^* X_2^*Gv  \tilde{e}_j^*\xi \Big| \prec n^{-1/2} . 
\end{align*}
For the term $(4)$, since $(\partial^1_{1,ij} Q)^2$ will produce an additional factor of $n$, we need to deal with this term in a finer way. Bounding one $\partial^1_{1,ij} Q$ by $\sqrt{n}q_n^{-1}$, and for the rest applying the same reasoning as in (3), we obtain
\begin{align*}
    |(4)| \leq Cn^{-1} \sum_{ij} \Big|\Big(X_2^* G v u^* \Big)_{ji} (\partial^1_{1,ij} Q)^2  \Big|\leq C q_n^{-1} \sum_{ij} \Big|u^*_i \theta^*e_i \tilde{e}_j^* X_2^*Gv  \tilde{e}_j^*\xi \partial^1_{1,ij} Q \Big|\lesssim q_n^{-1} . 
\end{align*}
Altogether, the second order term can be bounded by $q_n^{-1} = o(1)$. For higher order terms, $o(1)$ bound can be obtained similarly and actually more easily as $\kappa_{a+1}(x_{ij})$ decays by $n^{-1/2}$ if $\alpha$ increases by 1. Hence, we omit the details and claim 
\begin{align}
I=-t \sum_{ij} |u_i|^2 t_{ij}\mathbb{E}\Big[(X_2^* G vv^* G^*X_2)_{jj}\chi_G f(t) \Big]+O_\prec(q_n^{-1}). \label{010302}
\end{align}

Similarly, for the other terms in (\ref{010301}), we also have  
\begin{align}
&II=-t\sum_{ij} |\phi_j|^2 t_{ij} \mathbb{E}\Big[(X_2\mathcal{G}^*\psi\psi^* \mathcal{G}X_2^*)_{ii} \chi_Gf(t) \Big]+ O_{\prec}(q_n^{-1}),\notag\\
&III
    = -\frac{t}{|z|^2}\sum_{ijk} |q_i|^2 t_{ij}t_{kj} \mathbb{E}\Big[ (X_1\mathcal{G}\zeta\zeta^*\mathcal{G}^*X_1^*)_{kk} \chi_Gf(t) \Big] + O_{\prec}(q_n^{-1}), \notag\\
&IV
    = -\frac{t}{|z|^2}\sum_{ijk} |r_j|^2 t_{ji}t_{ki}\mathbb{E}\Big[ (X_2\mathcal{G}^*\eta\eta^*\mathcal{G}X_2^*)_{kk} \chi_Gf(t) \Big] + O_{\prec}(q_n^{-1}).   \label{010320}
\end{align}
In the following lemma, we provide a further estimate of the leading term in (\ref{010302}). The  terms in (\ref{010320}) can be estimated similarly. 
\begin{lem}\label{leading I} 
With the above notations, we have 
\begin{align*}
    &-t \sum_{ij} |u_i|^2 t_{ij}\mathbb{E}\Big[(X_2^* G vv^* G^*X_2)_{jj} \chi_Gf(t) \Big]\notag\\
    &\qquad= -\frac{t}{n|z|^2}\sum_{\alpha,\beta} |u_{\alpha}|^2|v_{\beta}|^2\vec{T}_{\alpha\cdot}^\top\Big[I-\frac{1}{n^2|z|^2} T^*T\Big]^{-1} \vec{T}_{\beta\cdot}{\mathbb{E}f(t)}+O_\prec(q_n^{-1}).
\end{align*}
\end{lem}
\begin{proof}[Proof of Lemma \ref{leading I}] By cumulant expansion, we have 
\begin{align}
    &\mathbb{E} \Big[\left(X_{2}^{*} G v v^{*} G^* X_{2}\right)_{jj} \chi_Gf(t) \Big]= \sum_{k} \frac{t_{k j}}{n} \mathbb{E} \partial_{2, kj} \Big[\left(G v v^{*} G^* X_{2} \right)_{k j}\chi_G f(t)\Big]+O_\prec(\frac{1}{nq_n}) \notag\\ 
    &= \frac{1}{z} \sum_{k} \frac{t_{k j}}{n} \mathbb{E}\Big[\Big(\left(X_{1} X_{2}^{*} G v v^{*} G^*\right)_{k k}-\left(v v^{*} G^*\right)_{k k}\Big)\chi_Gf(t)\Big]+O_\prec(\frac{1}{nq_n}),  \label{010310}
\end{align}
where the error terms come from the estimates of the higher order terms in the expansion, and their estimate can be done similarly as before, and thus the details are omitted.  Further, by applying another cumulant expansion, we have 
\begin{align}
    &\mathbb{E}\Big[\left(X_{1} X_{2}^{*} G v v^{*} G^*\right)_{k k}\chi_G f(t)\Big]= \sum_{b} \frac{t_{kb}}{n} \mathbb{E} \partial_{1, kb} \Big[\left(X_{2}^{*} G v v^{*} G^*\right)_{bk} \chi_G f(t)\Big]  \notag\\ 
        =& -\sum_{b} \frac{t_{kb}}{n}\mathbb{E} \Big[\Big(\left(v^{*} G^* X_{2} e_{b} e_{b}^{*} X_{2}^{*} G e_k e_k^{*} G v \right) + \left(X_{2}^{*} G v v^{*} G^* X_{2}\right)_{bb} G^*_{k k} \Big) \chi_G f(t)\Big] +O_\prec(\frac{1}{nq_n}). \label{010311}
\end{align}

Note that when we plug (\ref{010311}) to (\ref{010310}), the contribution from the first term in (\ref{010311}) is negligible,  since 
\begin{align*}
   & \frac{1}{n^2}\sum_{k, b} t_{k j}t_{k b} v^{*} G^* X_{2} e_{b} e_{b}^{*} X_{2}^{*} G e_k e_k^{*} G v 
    = \frac{1}{n^2}\sum_{k} t_{kj} v^{*} G^* X_{2} \mathfrak{T}_{k \cdot} X_{2}^{*} G e_k e_k^{*} G v  \\ &\qquad\qquad= \sum_{k} t_{kj} v^{*} A_k e_k e_k^{*} G v  = \frac{1}{n^2}v^{*}B_jv =O_\prec(\frac{1}{n^2})
\end{align*}
where $\mathfrak{T}_{i \cdot} = diag(\{t_{ij}\}_{j=1}^n), \mathfrak{T}_{\cdot j} = diag(\{t_{ij}\}_{i=1}^p)$, and $A_i, B_i$ are some matrices that depends on $\mathfrak{T}_{i \cdot}$ or $ \mathfrak{T}_{\cdot i}$ with $O_\prec(1)$ operator norm bounds. Hence, approximating $G^*_{k k}$ in (\ref{010311}) and $\left(v v^{*} G^*\right)_{k k}$ in (\ref{010310}) by $-1/\bar{z}$ and $-|v_k|^2/\bar{z}$ respectively, we can derive from (\ref{010310}) that 
\begin{align*}
    &\mathbb{E} \Big[\left(X_{2}^{*} G v v^{*} G^* X_{2}\right)_{jj} \chi_Gf(t)\Big]\\
    =& \frac{1}{n^2 |z|^2} \sum_{k, b} t_{k j}t_{k b} \mathbb{E}\Big[(X_{2}^{*} G v v^{*} G^* X_{2})_{bb} \chi_Gf(t)\Big]+ \frac{1}{n |z|^2}\sum_{k} t_{kj} |v_k|^2\mathbb{E}[f(t)] + O_{\prec}(\frac{1}{nq_n}).
\end{align*}
A further estimate of the variances of $(X_{2}^{*} G v v^{*} G^* X_{2})_{jj}$'s via cumulant expansion leads to 
\begin{align}
  & \mathbb{E} \Big[\left(X_{2}^{*} G v v^{*} G^* X_{2}\right)_{jj}\chi_G f(t)\Big] =\mathbb{E} \Big[\left(X_{2}^{*} G v v^{*} G^* X_{2}\right)_{jj}\chi_G\Big]\mathbb{E} [f(t)]+ O_{\prec}(\frac{1}{nq_n}) \notag\\
 &\qquad   = \frac{1}{n^2 |z|^2} \sum_{k, b} t_{k j}t_{k b} \mathbb{E}\Big[(X_{2}^{*} G v v^{*} G^* X_{2})_{bb}\chi_G\Big]\mathbb{E} [f(t)]\notag\\
 &\qquad \qquad+ \frac{1}{n |z|^2}\sum_{k} t_{kj} |v_k|^2\mathbb{E}[f(t)] + O_{\prec}(\frac{1}{nq_n}).
 \label{010315}
\end{align}
It is also easy to see that the second equation still holds if we cancel out the $\mathbb{E}[f(t)]$ factors. In light of this fact, we 
denote by $\vec{T}_{k \cdot}$ the $k$-th row of $T$, and by $\vec{T}_{\cdot k}$ the $k$-th column of $T$. Further, we set   $\vec{\mathcal{M}}=\left(\mathcal{M}_{j}\right)_{j=1}^{n}$ with $\mathcal{M}_{j}=\mathbb{E}\left(n X_{2}^{*} G v v^{*} G^* X_{2}\right)_{j j}$. Then we have the self consistent equation 
\begin{align} \label{M sc eqt}
\vec{\mathcal{M}} 
&= \frac{1}{|z|^2} \sum_{k} |v_k|^2 \left[I-\frac{1}{n^{2}|z|^{2}}\left(T^*T\right)\right]^{-1} \vec{T}_{k \cdot} + \vec{\varepsilon},
\end{align}
where $\vec{\varepsilon}$ is an error vector with $\|\vec{\varepsilon}\|_{\infty}=O_{\prec}(q_n^{-1})$. Plugging (\ref{M sc eqt}) back to (\ref{010315}), we get 
\begin{align*}
&-t \sum_{ij} |u_i|^2 t_{ij}\mathbb{E}\Big[(X_2^* G vv^* G^*X_2)_{jj}\chi_G f(t) \Big] \notag\\
&= -\frac{t}{n|z|^2} \sum_{ik} |u_i|^2|v_k|^2 \vec{T}_{i \cdot}^{\top}\left[I-\frac{1}{n^{2}|z|^{2}}\left(T^*T\right)\right]^{-1}\vec{T}_{k \cdot} \mathbb{E}f(t)+ O_{\prec}(q_n^{-1}).
\end{align*}
This concludes the proof of Lemma \ref{leading I}.
\end{proof}

We then proceed with the proof of Proposition \ref{jointCLT} for the quadratic forms in (\ref{Q def}). Similarly to the proof of Lemma \ref{leading I}, we can also estimate other terms in (\ref{010320}).  In summary, we have 
\begin{align*}
    &I=-\frac{t}{n|z|^2} \sum_{ik} |u_i|^2|v_k|^2 \vec{T}_{i \cdot}^{\top}\left[I-\frac{1}{n^{2}|z|^{2}}\left(T^*T\right)\right]^{-1}\vec{T}_{k \cdot} \mathbb{E}f(t)+ O_{\prec}(q_n^{-1}), \notag\\
    & II=-\frac{t}{n|z|^2} \sum_{jk}|\phi_j|^2|\psi_k|^2\vec{T}_{\cdot j}^{\top}\left[I-\frac{1}{n^{2}|z|^{2}}\left(TT^*\right)\right]^{-1}\vec{T}_{\cdot k} \mathbb{E}f(t)+ O_{\prec}(q_n^{-1}), \notag\\
    & III=-\frac{t}{n^2|z|^4}\sum_{ik} |q_i|^2 |\zeta_k|^2 \Big(TT^*\Big[I - \frac{1}{n^2|z|^2}(TT^*)\Big]^{-1}\vec{T}_{\cdot k}\Big)_i \mathbb{E}f(t) + O_{\prec}(q_n^{-1}),\notag\\
    &IV=-\frac{t}{n^2|z|^4}\sum_{jk} |r_j|^2 |\eta_k|^2 \Big(TT^*\Big[I - \frac{1}{n^2|z|^2}(TT^*)\Big]^{-1}\vec{T}_{\cdot k}\Big)_j \mathbb{E}f(t) + O_{\prec}(q_n^{-1}). 
\end{align*}
Plugging the above estimates into (\ref{010301}), we get an estimate of the form in (\ref{010322}), and thus a CLT follows for the particular linear combination in (\ref{Q def}). 

In order to prove the general joint CLT for the quadratic forms in  (\ref{collection of qf}), we shall consider the characteristic function of more general linear combination of the form 
\begin{align*}
\sqrt{n}\left(z\sum_{i=1}^{m_1}c_{1i}u_i^*\underline{G}(z)v_i +  z\sum_{j=1}^{m_2} c_{2j}\psi_j^*\underline{\mathcal{G}}(z)\phi_j +  \sum_{k=1}^{m_3} c_{3k}q_k^*  X_{1} \underline{\mathcal{G}}(z) \gamma_k + \sum_{\ell=1}^{m_4} c_{4\ell}\eta_\ell^* \underline{\mathcal{G}}(z) X_{2}^{*} r_\ell \right). 
\end{align*}
The derivation is a straightforward extension of that is done for $Q$ in (\ref{Q def}). For brevity, we omit the details and claim  (\ref{limitgaussian}).

What remains is to show the asymptotic  independence of the Green function quadratic forms in (\ref{collection of qf}) and the linear (in $X_i$) term $a^*X_ib$. It  can be proved via a slight modification of the above derivation of the CLT.  We illustrate the necessary modification as follows, again based on the simple linear combination $Q$ defined in (\ref{Q def}). We now involve the term of the form $a^*X_1 b$, for instance, and define
\begin{align*}
f(t,s)=\exp(\ii t Q\chi_G+\sqrt{n}\ii s a^*X_1 b). 
\end{align*}
We can add the term of the form $c^* X_2 d$ as well. But for brevity, we restrict ourselves to the above quantity. Instead of (\ref{010322}), we now need to show 
\begin{align}
\frac{\partial}{\partial t}\mathbb{E} f(t, s)=-\mathfrak{d}^2t \mathbb{E} f(t,s)+o(1).  \label{010710}
\end{align}
Solving the above equation with $\mathbb{E}f(0,s)=\mathbb{E}\exp(\sqrt{n}\ii s a^* X_1 b )$ gives
\begin{align*}
\mathbb{E} f(t, s)=\exp(-\frac{\mathfrak{d}^2t^2}{2})\mathbb{E}\exp(\sqrt{n}\ii s a^* X_1 b )+o(1)
\end{align*}
which proves the asymptotic Gaussianity of $Q$, and the asymptotic independence between $Q$ and $\sqrt{n} a^* X_1b$ simultaneously. 
Hence, it suffices to illustrate how to adapt the proof of (\ref{010710}) from that of (\ref{010322}). Similarly to (\ref{010301}), we can write 
\begin{align*}
\frac{\partial}{\partial t}\mathbb{E}f(t,s)&=\ii \mathbb{E} Qf(t, s)\notag\\
&= \ii \sqrt{n}\mathbb{E}  \left( u^* X_1X_2^* Gv +  \psi^* X_2^*X_1 \mathcal{G}\phi +  q^*  X_{1} \underline{\mathcal{G}} \gamma + \eta^* \underline{\mathcal{G}} X_{2}^{*} r \right)\chi_G{f(t,s)} \notag\\
&\qquad=: \widetilde{I}+\widetilde{II}+\widetilde{III}+\widetilde{IV}. 
\end{align*}
It suffices to illustrate the estimate of $\widetilde{I}$, as an adapt of the estimate of $I$ in (\ref{010301}). The other terms can be estimated similarly.  We again apply cumulant expansion as (\ref{010723}). In the first order term of the cumulant expansion, we will have all terms analogous to those in (\ref{010722}), with an additional term which involves the derivative of the newly added $\sqrt{n}a^*X_1b$. This terms reads
\begin{align*}
\ii \sum_{ij} \mathbb{E} \Big[ (X_2^* G v u^* )_{ji} a_ib_j\chi_G{f(t,s)} \Big]=\ii \mathbb{E}[u^*a\cdot b^* X_2^* Gv\chi_G\cdot f(t,s)]=O_\prec(q_n^{-1}),
\end{align*} 
where we used Proposition \ref{quadraticformbound}. The derivative of the $\sqrt{n}a^*X_1b$ term will also show up  in the higher order terms in the cumulant expansion, but the related terms can all be estimated easily with the aid of Proposition \ref{quadraticformbound} and Cauchy Schwarz inequality. Hence, we omit the details and conclude (\ref{010710}).

\section{First order  of the heavy-tailed case: Proof of Theorem \ref{thm: first order opt}} \label{sec: 997}

Recall the matrix model $\mathcal{Y} = \mathcal{X} + \mathcal{S}$ in (\ref{model}), and the eigendecomposition for $\mathcal{S} = W\mathfrak{D}W^*$ from (\ref{062704}). We further remark here that in this section, we work with Assumption \ref{main assump 2}.  In this part, various limiting statements will be made. In these statements, we prefer to keep the non-asymptotic ratio $p/n$ instead of its limit, in the ``limiting objects". Hence, we make a convention here that in this section, a statement $A(n)\to B(n)$ means $A(n)-B(n)\to 0$, and a statement $X(n)\stackrel{ \text{law}}\rightarrow Y(n)$ means that for any continuous and bounded function (independent of $n$), one has $\mathbb{E} f(X(n))-\mathbb{E}f(Y(n))\to 0$. 

As mentioned earlier in \eqref{102901}, when $z$ is not an eigenvalue of $\mathcal{X}$, we have
\begin{equation}\label{eq: 989}
	\det(I+\mathfrak{D}W^{*}(\mathcal{X}-z)^{-1}W) = 0,
\end{equation}
if and only if $z$ is an eigenvalue of $\mathcal{Y}$.
To invoke this eigenvalue criterion, we shall first bound the spectral radius of $\mathcal{X}$, based on the characteristic polynomial approach introduced in \cite{BCG22}.
\begin{pro}\label{prop: 626} Suppose that Assumption \ref{main assump 2} holds. 
Let $\{\lambda_{i}(\mathcal{X})\}_{i=1}^{p+n}$ be the eigenvalues of $\mathcal{X}$.
Then, for any $\epsilon>0$, we have for large $n$
\begin{equation*}
	\mathbb{P}\Big\{ \max_{1\le i\le p+n} |\lambda_{i}(\mathcal{X})| > (p/n)^{1/4} + \epsilon \Big\} = o(1).
\end{equation*}
In other words, the spectral radius of $\mathcal{X}$ is bounded above by $(p/n)^{1/4}$ in probability.
\end{pro}
\noindent
The proof of Proposition \ref{prop: 626} is deferred to Section \ref{sec: 1110}.

~

Next, we define the function
\begin{equation*}
	f(z) = \det(I+\mathfrak{D}W^{*}(\mathcal{X}-z)^{-1}W).
\end{equation*}
By Proposition \ref{prop: 626}, with probability $1-o(1)$ as $n\to\infty$, all eigenvalues of $\mathcal{X}$ lie inside the disk
$\{z\in\mathbb{C}: |z|<\sqrt{(p/n)^{1/2} + \delta/3}\}.$
For a constant $M>0$, we define $\mathcal{X}^{(M)}=(\mathcal{X}^{(M)}_{ij})$ by
\begin{equation}\label{eq: 1029}
	\mathcal{X}^{(M)}_{ij} = \mathcal{X}_{ij} \mathds{1}(|\sqrt{n}\mathcal{X}_{ij}|\le M) - \E[\mathcal{X}_{ij} \mathds{1}(|\sqrt{n}\mathcal{X}_{ij}|\le M)].
\end{equation}
We then introduce
\begin{equation*}
	f^{(M)}(z) = \det(I+\mathcal{D}W^{*}(\mathcal{X}^{(M)}-\lambda)^{-1}W). 
\end{equation*}
Since all entries of $\mathcal{X}^{(M)}$ are bounded, one can similarly apply the argument used to show Theorem \ref{thm: first order} of the light-tail regime.
Consequently, for every point $d_{i}$ ($i=\pm 1,\cdots,\pm \tilde{k}$),
there exists a unique zero of \(f^{(M)}(z)\) outside the disk
$\{z\in\mathbb{C}: |z|<\sqrt{(p/n)^{1/2} + \delta/2}\}$ such that
this zero lies arbitrarily close to $d_{i}$ whenever $n$ and $M$ are sufficiently large.
Hence, invoking Rouch\'{e}'s theorem with the eigenvalue criterion \eqref{eq: 989},
it remains only to show that, with probability $1-o(1)$,
\begin{equation}\label{eq: 570}
	\sup_{|z|\ge \sqrt{(p/n)^{1/2} + \delta/2}} |f(z)-f^{(M)}(z)| = o(1),
\end{equation}
for sufficiently large $n$ and $M$.

~

We recall the determinant representations
\begin{equation}\label{eq: 1061}
	f(z) = \frac{\det(\mathcal{X}+\mathcal{S}-z)}{\det(\mathcal{X}-z)} \quad\text{and}\quad
	f^{(M)}(z) = \frac{\det(\mathcal{X}^{(M)}+\mathcal{S}-z)}{\det(\mathcal{X}^{(M)}-z)}.
\end{equation}
Hence it follows that
\begin{align*}
	f_{n}(z) - f^{(M)}_{n}(z)
	= &\frac{\det(\mathcal{X}+\mathcal{S}-z)\big(\det(\mathcal{X}^{(M)}-z)-\det(\mathcal{X}-z)\big)}{\det(\mathcal{X}-z)\det(\mathcal{X}^{(M)}-z)} \\
	&+ \frac{\det(\mathcal{X}+\mathcal{S}-z)-\det(\mathcal{X}^{(M)}+\mathcal{S}-z)}{\det(\mathcal{X}^{(M)}-z)}.
\end{align*}
Consequently, the estimate \eqref{eq: 570} follows from the next proposition, whose proof is deferred to Section \ref{sec: 1131}.

\begin{pro}\label{prop: 576}
Let $\epsilon>0$ be any (small) constant. \\

\noindent (i) The following holds for any constant $M>0$: there exists $c_{0}>0$ such that
	\begin{equation*}
		\mathbb{P}\Big\{ \min\Big(\inf_{|z|\ge \sqrt{(p/n)^{1/2} + \delta/2}}|z^{-n}\det(\mathcal{X}-z)|,\inf_{|z|\ge \sqrt{(p/n)^{1/2} + \delta/2}}|z^{-n}\det(\mathcal{X}^{(M)}-z)|\Big) < c_{0} \Big\} < \epsilon,
	\end{equation*}
	for $n$ large enough.

\noindent (ii) There exists $C_{0}>0$ such that for every $n\ge 1$,
	\begin{equation*}
		\mathbb{P}\Big\{ \sup_{|z|\ge \sqrt{(p/n)^{1/2} + \delta/2}}|z^{-n}\det(\mathcal{X}+\mathcal{S}-z)| > C_{0}  \Big\} < \epsilon.
	\end{equation*}

\noindent (iii) Let $\tau>0$ be any (small) constant.
For sufficiently large $M$, we have for every $n\ge 1$,
	\begin{equation*}
		\mathbb{P}\Big\{ \sup_{|z|\ge \sqrt{(p/n)^{1/2} + \delta/2}}|z^{-n}\det(\mathcal{X}-z)-z^{-n}\det(\mathcal{X}^{(M)}-z)| > \tau  \Big\} < \epsilon.
	\end{equation*}
	Similarly, if $M$ is large enough, then, for every $n\ge 1$,
	\begin{equation*}
		\mathbb{P}\Big\{ \sup_{|z|\ge \sqrt{(p/n)^{1/2} + \delta/2}}|z^{-n}\det(\mathcal{X}+\mathcal{S}-z)-z^{-n}\det(\mathcal{X}^{(M)}+\mathcal{S}-z)| > \tau  \Big\} < \epsilon.
	\end{equation*}
\end{pro}

\subsection{Proof of Proposition \ref{prop: 626}} \label{sec: 1110}

We follow the strategy of \cite{BCG22}.
First define
\begin{equation}\label{eq: 779}
	q_{n}(\omega) = \det(1-(p/n)^{-1/4}\omega\mathcal{X}) = 1 + \sum_{k=1}^{p+n}(-1)^{k}(p/n)^{-k/4}\omega^{k}\mathcal{P}_{k}^{(n)},
\end{equation}
where
\begin{equation}\label{eq: 1100}
	\mathcal{P}_{k}^{(n)} = \sum_{\substack{I\subseteq\{1,\cdots,p+n\}\\|I|=k}}\det(\mathcal{X}(I)), \quad
	\mathcal{X}(I) = (\mathcal{X}_{ij})_{i,j\in I}.
\end{equation}
Notice that $\det(z-\mathcal{X})=z^{n}q_{n}(z^{-1}(p/n)^{1/4})$.
For the unit disk $\mathbb{D}=\{z\in\mathbb{C}:|z|<1\}$,
denote by $H(\mathbb{D})$ the set of holomorphic function on $\mathbb{D}$ equipped with the topology of uniform convergence on compact subsets.
As a random element of $H(\mathbb{D})$, we say that $q_{n}$ converges in law to some
random holomorphic function $q$ of $H(\mathbb{D})$ if $\E[f(q_{n})]$ converges to $\E[f(q)]$ for every bounded real continuous function $f$ on $H(\mathbb{D})$.

\begin{pro}\label{prop: 622}
Let $q_{n}$ be as in \eqref{eq: 779}.
We have
\begin{equation}\label{eq: 837}
	q_{n} \overset{\textnormal{law}}{\longrightarrow} \kappa e^{-F} \quad \text{as $n\to\infty$,}
\end{equation}
where $\kappa$ is the holomorphic function on $\mathbb{D}$ defined by
\begin{equation}\label{eq: 639}
	\kappa(\omega) = \exp\bigg(-\frac{1}{2}\sum_{j=1}^{\infty} (p/n)^{-1}j^{-1}\omega^{4j}\bigg),
\end{equation}
and $F$ is the random holomorphic function on $\mathbb{D}$ defined by
\begin{equation*}
	F(\omega) = \sum_{k=1}^{\infty}Z_{k}k^{-1}\omega^{k},
\end{equation*}
where $\{Z_{k}\}$ is a sequence of independent random variables such that
$Z_{k}$ is a standard normal random variable for even $k$ and $Z_{k}=0$ for odd $k$.
\end{pro}

Recall the fact $\det(z-\mathcal{X})=z^{n}q_{n}(z^{-1}(p/n)^{1/4})$ and from the above proposition, we notice that $\kappa(z^{-1}(p/n)^{1/4})e^{-F((z^{-1}(p/n)^{1/4}))}\neq 0$ uniformly for all $|z|\geq \sqrt{(p/n)^{1/2} + \delta/2}$. Hence, to show Proposition \ref{prop: 626}, it suffices to establish Proposition \ref{prop: 622}, which asserts the convergence of $q_{n}$.
The remaining steps coincide exactly with those in \cite[Section 2]{BCG22}; accordingly, we omit them here.

~

We now begin the proof of Proposition \ref{prop: 622}.
The lemmas below will be used in the proof.

\begin{lem}\label{lem: 956}
The sequence $\{q_{n}\}$ is tight, i.e., for every $\epsilon>0$, there exists a compact subset $K$ of $H(\mathbb{D})$ such that $\mathbb{P}\{q_{n}\in K\}>1-\epsilon$ for every $n$.
\end{lem}
\begin{proof}

As in \cite[Lemma 3.1]{BCG22},
it is enough to bound $\E[|q_{n}(\omega)|^{2}]$ by a deterministic continuous function of $\omega$ not depending on $n$.
Due to independence and mean-zero condition, we notice that
\begin{equation*}
	\E[\mathcal{P}^{(n)}_{k}\mathcal{P}^{(n)}_{l}] = 0 \quad\text{if $k\neq l$}.
\end{equation*}
It is sufficient to show that
\begin{equation}\label{eq: 697}
	\E[|\mathcal{P}_{k}^{(n)}|^{2}] \le \Big(\frac{p}{n}\Big)^{k/2}.
\end{equation}
Define
\begin{equation*}
	\mathcal{T}
	=
	\begin{pmatrix}
		0 & \mathbbm{1}_{p\times n} \\
		\mathbbm{1}_{n\times p} & 0
	\end{pmatrix}\in\mathbb{R}^{(p+n)\times (p+n)},
\end{equation*}
where $\mathbbm{1}_{p\times n}$ is a $p\times n$ all-ones matrix.
Again, by independence and mean-zero condition,
\begin{equation}\label{eq: 2260}
	\E[|\mathcal{P}_{k}^{(n)}|^{2}] = n^{-k} \sum_{\substack{I\subseteq\{1,\cdots,p+n\}\\|I|=k}} D(\mathcal{T},I),
\end{equation}
where
\begin{equation*}
	D(\mathcal{T},I) = \sum_{\sigma\in\text{Perm}(I)} \prod_{i\in I}\mathcal{T}_{i\sigma(i)},
\end{equation*}
in which $\text{Perm}(I)$ is the set of all permutations on the subset $I\subseteq\{1,2,\cdots,p+n\}$.

For odd $k$, if $|I|=k$, we have $D(\mathcal{T},I)=0$.
We focus on the case that $k$ is even.
Note that $D(\mathcal{T},I)\neq 0$ if and only if $I\cap \{1,\cdots,p\}=k/2$.
We observe that
\begin{equation*}
	\sum_{\substack{I\subseteq\{1,\cdots,p+n\}\\|I|=k}} D(\mathcal{T},I) = {p\choose k/2}{n\choose k/2}(k/2)!(k/2)!\le (pn)^{k/2},
\end{equation*}
and
\begin{equation}\label{eq: 2277}
	n^{-k} \sum_{\substack{I\subseteq\{1,\cdots,p+n\}\\|I|=k}} D(\mathcal{T},I) \le \Big(\frac{p}{n}\Big)^{k/2}.
\end{equation}
Hence the desired estimate \eqref{eq: 697} follows.

\end{proof}

\begin{lem}[{\cite[Lemma 3.2]{BCG22}}]\label{lem: 972}
Let $\{f_{n}\}$ be a tight sequence of random elements of $H(\mathbb{D})$, each with power-series expansion $f_{n}(\omega)=\sum_{k=1}^{\infty}\omega^{k}P^{(n)}_{k}$ ($\omega\in\mathbb{D}$).
Assume that for every fixed $m\ge 0$,
\begin{equation*}
	(P_{0}^{(n)},\cdots,P_{m}^{(n)}) \overset{\textnormal{law}}{\longrightarrow} (P_{0},\cdots,P_{m}) \quad
	\text{as $n\to\infty$},
\end{equation*}
where \((P_{m})_{m\ge 0}\) is a common sequence of random variables (independent of \(n\)).
Then the random holomorphic function $f(\omega)=\sum_{k=0}^{\infty}(-1)^{k}\omega^{k}P_{k}$ is well-defined in $H(\mathbb{D})$ and
\begin{equation*}
	f_{n} \overset{\textnormal{law}}{\longrightarrow} f \quad
	\text{as $n\to\infty$}.
\end{equation*}
\end{lem}

\begin{lem}\label{lem: 778}
For $M>0$, let $\mathcal{X}^{(M)}$ be as in \eqref{eq: 1029}.
We define
\begin{equation*}
	\mathcal{P}_{k}^{(n,M)} = \sum_{\substack{I\subseteq\{1,\cdots,p+n\}\\|I|=k}}\det(\mathcal{X}^{(M)}(I)), \quad
	\mathcal{X}^{(M)}(I) = (\mathcal{X}^{(M)}_{ij})_{i,j\in I}.
\end{equation*}
Consider $m\ge 1$.
Assume that there exist $\{(\mathcal{P}_{1}^{(\infty,M)},\cdots,\mathcal{P}_{m}^{(\infty,M)})\}_{M\ge 1}$ and $(\mathcal{P}_{1},\cdots,\mathcal{P}_{m})$ such that
\begin{equation*}
	(\mathcal{P}_{1}^{(n,M)},\cdots,\mathcal{P}_{m}^{(n,M)}) \overset{\textnormal{law}}{\longrightarrow} (\mathcal{P}_{1}^{(\infty,M)},\cdots,\mathcal{P}_{m}^{(\infty,M)}) \quad
	\text{as $n\to\infty$,}
\end{equation*}
and
\begin{equation*}
	(\mathcal{P}_{1}^{(\infty,M)},\cdots,\mathcal{P}_{m}^{(\infty,M)}) \overset{\textnormal{law}}{\longrightarrow} (\mathcal{P}_{1},\cdots,\mathcal{P}_{m}) \quad
	\text{as $M\to\infty$.}
\end{equation*}
Let $\mathcal{P}_{k}^{(n)}$ be as in \eqref{eq: 1100}.
Then we have
\begin{equation*}
	(\mathcal{P}_{1}^{(n)},\cdots,\mathcal{P}_{m}^{(n)}) \overset{\textnormal{law}}{\longrightarrow} (\mathcal{P}_{1},\cdots,\mathcal{P}_{m}) \quad
	\text{as $n\to\infty$}.
\end{equation*}
\end{lem}
\begin{proof}

We claim that, for each $k$,
\begin{equation}\label{eq: 1244}
	\lim_{M\to\infty}\E[|\mathcal{P}_{k}^{(n)}-\mathcal{P}_{k}^{(n,M)}|^{2}] = 0.
\end{equation}
When $k$ is odd, we find that $\mathcal{P}_{k}^{(n)}=\mathcal{P}_{k}^{(n,M)}=0$.
For $k=2m$, we have
\begin{align}\label{eq: 1249}
	\E[|\mathcal{P}_{k}^{(n)}-\mathcal{P}_{k}^{(n,M)}|^{2}]
	&= \sum_{\substack{I\subseteq\{1,\cdots,p\}\\ J\subseteq\{p+1,\cdots,p+n\}\\|I|=|J|=m}} \E[|\det(\mathcal{X}(I\cup J))-\det(\mathcal{X}^{(M)}(I\cup J))|^{2}] \nonumber\\
	&= n^{-k}{p\choose m}{n\choose m}(m)!(m)! \times \E[|\prod_{j=1}^{k}a_{1j}-\prod_{j=1}^{k}a_{1j}^{(M)}|^{2}] \nonumber\\
	&\le \Big(\frac{p}{n}\Big)^{k/2} \times \E[|\prod_{j=1}^{k}a_{1j}-\prod_{j=1}^{k}a_{1j}^{(M)}|^{2}],
\end{align}
where $a_{ij}=\sqrt{n}x_{ij}$ and $a^{(M)}_{ij}=\sqrt{n}x^{(M)}_{ij}$.
The expectation on the right-hand side vanishes as \(M\to\infty\), proving
\eqref{eq: 1244}.  Moreover, the bound in \eqref{eq: 1249} is uniform in \(n\),
so the rate of convergence does not depend on \(n\).  This suffices to
establish the lemma, exactly as in \cite[Lemma 3.3]{BCG22}.

\end{proof}
	
Due to Lemma \ref{lem: 778},
we can further assume in the rest of the proof  that all entries of $\sqrt{n}\mathcal{X}$ are bounded.
Proceeding as in the discussion following \cite[Lemma 3.3]{BCG22}, we can write
\begin{equation}\label{eq: 622}
	q_{n}(\omega) = \exp\Big(-\sum_{k=1}^{\infty}\text{Tr}(\mathcal{X}^{k})k^{-1} (p/n)^{-k/4}\omega^{k}\Big).
\end{equation}
Moreover $(P_{1}^{(n)},\cdots,P_{k}^{(n)})$ is a polynomial in $(\text{Tr}(\mathcal{X}),\cdots,\text{Tr}(\mathcal{X}^{k}))$ that is independent of $n$.
Then, everything boils down to the convergence of traces:
\begin{equation*}
	(\text{Tr}(\mathcal{X}), \text{Tr}(\mathcal{X}^{2}), \cdots, \text{Tr}(\mathcal{X}^{k})), \quad k\ge 1.
\end{equation*}

~

\noindent
Differently from \cite{BCG22}, we do not prove the convergence of the above random vector by combinatorial argument. Instead, we can easily prove such a convergence again based on the cumulant expansion approach. To this end, we first write
\begin{equation*}
	\text{Tr}(\mathcal{X}^{k}) = -\frac{1}{2\pi\mathrm{i}}\int_{\Gamma}\omega^{k}\text{Tr}(\mathcal{X}-\omega)^{-1}\mathrm{d}\omega,
\end{equation*}
where a positively oriented Jordan curve $\Gamma\subseteq\mathbb{C}$ encloses the spectrum of $X$.
Note that
\begin{equation*}
	\text{Tr}(\mathcal{X}-\omega)^{-1} = \omega\text{Tr}\,G(\omega^{2}) + \omega\text{Tr}\,\mathcal{G}(\omega^{2}),
\end{equation*}
where we recalled the notations defined in (\ref{070701}). 
Then,
\begin{equation*}
	\text{Tr}(\mathcal{X}^{k}) = -\frac{1}{2\pi\mathrm{i}}\oint_{\Gamma}\omega^{k+1}(\text{Tr}\,G(\omega^{2})+\text{Tr}\,\mathcal{G}(\omega^{2}))\mathrm{d}\omega.
\end{equation*}
Further, using the centered Green functions defined in (\ref{def of G under}), one notices that for $k\geq 1$, 
\begin{equation*}
	\text{Tr}(\mathcal{X}^{k}) = -\frac{1}{2\pi\mathrm{i}}\oint_{\Gamma}\omega^{k+1}(\text{Tr}\,\underbar{G}(\omega^{2})+\text{Tr}\,\underline{\mathcal{G}}(\omega^{2}))\mathrm{d}\omega.
\end{equation*}

\noindent
Consider the centering of $\text{Tr}(\mathcal{X}^{k})$. Write
\begin{align*}
	L_{k} &= \text{Tr}(\mathcal{X}^{k}) - \E[\text{Tr}(\mathcal{X}^{k})] \\
	 &= -\frac{1}{2\pi\mathrm{i}} \Big(\oint_{\Gamma} \omega^{k+1} (\text{Tr}\,\underbar{G}(\omega^{2})-\E[\text{Tr}\,\underbar{G}(\omega^{2})]) \mathrm{d}\omega+\oint_{\Gamma} \omega^{k+1} (\text{Tr}\,\underline{\mathcal{G}}(\omega^{2})-\E[\text{Tr}\,\underline{\mathcal{G}}(\omega^{2})]) \mathrm{d}\omega\Big)\\
	 &=  -\frac{1}{\pi\mathrm{i}} \oint_{\Gamma} \omega^{k+1} (\text{Tr}\,\underbar{G}(\omega^{2})-\E[\text{Tr}\,\underbar{G}(\omega^{2})]) \mathrm{d}\omega.
\end{align*}

\begin{lem} \label{lem.fluctuation}
Suppose all entries of $\sqrt{n}\mathcal{X}$ are bounded.
Choose any given $k\ge 1$.
Then we have
\begin{equation}\label{eq: 981}
	(L_{1},\cdots,L_{k}) \overset{\textnormal{law}}{\longrightarrow}
	\big((p/n)^{1/4}Z_{1},\cdots,(p/n)^{k/4}Z_{k}\big) \quad \text{as $n\to\infty$,}
\end{equation}
where $\{Z_{k}\}$ is a sequence of independent random variables as in Proposition \ref{prop: 622}.
\end{lem}

Apart from the fluctuation of $\Tr \mathcal{X}^k$, we also need to consider the convergence of $\E[\textnormal{Tr}(\mathcal{X}^{k})]$.
See the following lemma.
\begin{lem} \label{lem.expectation}
Suppose all entries of $\sqrt{n}\mathcal{X}$ are bounded.
For each $k\ge 1$, we have
\begin{equation*}
	\E[\textnormal{Tr}(\mathcal{X}^{k})]\to w_k,
\end{equation*}
where
\begin{equation*}
	w_{k} 
	= \begin{cases}
		2\Big(\frac{p}{n}\Big)^{m+1}, &\text{if $k=4m$,} \\
		0, &\text{otherwise.}
	\end{cases}
\end{equation*}
\end{lem}
%
%
%

The detailed proofs of the above two lemmas are stated in Appendix \label{app.lemmas}. 
Combining the above two lemmas yields
\begin{equation}\label{eq: 839}
	(\text{Tr}(\mathcal{X}), \cdots, \text{Tr}(\mathcal{X}^{k})) \overset{\textnormal{law}}{\longrightarrow}
	\big((p/n)^{1/4}Z_{1}+w_{1}, \cdots, (p/n)^{k/4}Z_{k}+w_{k}\big).
\end{equation}
Due to \eqref{eq: 622}, together with Lemma \ref{lem: 956} and Lemma \ref{lem: 972}, one can obtain \eqref{eq: 837}, according to the argument in \cite{BCG22}.

\subsection{Proof of Proposition \ref{prop: 576}} \label{sec: 1131}

(i) we will make use of \eqref{eq: 837}.
Let $c>0$ to be chosen later.  By continuous mapping theorem, we have
\begin{multline*}
	\mathbb{P}\{ \inf_{|z|>\sqrt{(p/n)^{1/2}+\delta/2}}|z^{-n}\det(\mathcal{X}-z)| < c \} \\
	= \mathbb{P}\Big\{ \inf_{|z|>\sqrt{(p/n)^{1/2}+\delta/2}} |\kappa(z^{-1}(p/n)^{1/4})| \exp\Big(-\Re\Big(\sum_{k=1}^{\infty}Z_{k}k^{-1}(p/n)^{k/4}z^{-k}\Big)\Big) < c \Big\} + o(1).
\end{multline*}
Note that
\begin{multline*}
	\inf_{|z|>\sqrt{(p/n)^{1/2}+\delta/2}} |\kappa(z^{-1}(p/n)^{1/4})| \exp\Big(-\Re\Big(\sum_{k=1}^{\infty}Z_{k}k^{-1}(p/n)^{k/4}z^{-k}\Big)\Big) \\
	\ge \Big(\inf_{|z|>\sqrt{(p/n)^{1/2}+\delta/2}} |\kappa(z^{-1}(p/n)^{1/4})|\Big) \\
	\times \exp\Big(-\sup_{|z|>\sqrt{(p/n)^{1/2}+\delta/2}}\Re\Big(\sum_{k=1}^{\infty}Z_{k}k^{-1}(p/n)^{k/4}z^{-k}\Big)\Big).
\end{multline*}
In light of \eqref{eq: 639}, one can check that $$\inf_{|z|>\sqrt{(p/n)^{1/2}+\delta/2}} |\kappa(z^{-1}(p/n)^{1/4})|> c_{\delta}$$ for some constant $c_{\delta}$ depending on $\delta$.
Further, we have 
\begin{multline*}
	\mathbb{P}\Big\{ \sup_{|z|>\sqrt{(p/n)^{1/2}+\delta/2}} \Re\Big(\sum_{k=1}^{\infty}Z_{k}k^{-1}(p/n)^{k/4}z^{-k}\Big) > \log\Big(\frac{c_{\delta}}{c}\Big) \Big\} \\
	\le \frac{\E\big( \big|\sum_{k=1}^{\infty}|Z_{k}|k^{-1}(p/n)^{k/4}((p/n)^{1/2}+\delta/2)^{-k/2}\big|^{2} \big)}{\log^{2}(c_{\delta}/c)}.
\end{multline*}
One can easily check that $$\E\bigg[\Big|\sum_{k=1}^{\infty}|Z_{k}|k^{-1}(p/n)^{k/4}((p/n)^{1/2}+\delta/2)^{-k/2}\Big|^{2}\bigg]<\infty.$$
Then we can complete the proof by taking $c>0$ sufficiently small.
The same argument works for $z^{-n}\det(\mathcal{X}^{(M)}-zI_{n})$.

~

(ii) We consider
\begin{equation*}
	\sup_{|z|\ge\sqrt{(p/n)^{1/2}+\delta/2}}|z^{-n}\det(\mathcal{X}+\mathcal{S}-z)|.
\end{equation*}
For a positive integer $\mathfrak{p}$, let us write $[\mathfrak{p}]=\{1,\dots,\mathfrak{p}\}$.
We first expand the determinant of the sum of matrices using their minors as in \eqref{eq: 488}:
\begin{equation*}
	\det(1-z^{-1}\mathcal{X}-z^{-1}\mathcal{S}) = \sum_{\substack{I,J\subseteq[p+n]\\|I|=|J|}} (-1)^{s(I,J)}\det([-z^{-1}\mathcal{X}]_{I,J})\det([1-z^{-1}\mathcal{S}]_{I^c,J^c}).
\end{equation*}
Fix $I\subseteq[p+n]$ and write $m=|I|(=|J|)$.
Further perform the expansion
\begin{equation}
	\det([1-z^{-1}\mathcal{S}]_{I^c,J^c})
	= \sum_{\substack{\tilde{I}\subseteq I^{c}, \tilde{J}\subseteq J^{c}\\|\tilde{I}|=|\tilde{J}|}} (-1)^{s(\tilde{I},\tilde{J})}\det([1_{p+n}]_{\tilde{I},\tilde{J}})\det([-z^{-1}\mathcal{S}]_{I^{c}\backslash\tilde{I},J^{c}\backslash\tilde{J}}), \label{070902}
\end{equation}
where $1_{p+n}$ is the $(p+n)\times (p+n)$ identity matrix.
Since $\det([1_{p+n}]_{\tilde{I},\tilde{J}})=0$ whenever $\tilde{I}\neq\tilde{J}$,
it is enough to consider
\begin{equation*} %
	\sum_{\substack{\tilde{I}\subseteq I^{c},J^{c}}} \det([-z^{-1}\mathcal{S}]_{I^{c}\backslash\tilde{I},J^{c}\backslash\tilde{I}}).
\end{equation*}
Since $\text{rank}(\mathcal{S})=2k$, we have $\det([\mathcal{S}]_{I^{c}\backslash\tilde{I},J^{c}\backslash\tilde{I}})=0$ for all $\tilde{I}$ with $|\tilde{I}|<p+n-m-2k$.
Thus, the contributing terms are
\begin{equation}\label{eq: 1574}
	\sum_{\substack{\tilde{I}\subseteq I^{c},J^{c}\\|\tilde{I}|\ge p+n-m-2k}} \det([-z^{-1}\mathcal{S}]_{I^{c}\backslash\tilde{I},J^{c}\backslash\tilde{I}}).
\end{equation}
Because $\tilde{I}\subseteq I^{c}\cap J^{c}$, $p+n-m\ge|I^{c}\cap J^{c}|\ge|\tilde{I}|\ge p+n-m-2k$, and $|I|=|J|=m$, one can observe that the index set $J$ is obtained from $I$ by replacing at most $2k$ indices.
More precisely, there exist $\{i_{1},\dots,i_{\alpha}\}\subseteq I$ and $\{j_{1},\dots,j_{\alpha}\}\subseteq I^{c}$ with $0\le\alpha\le\min(m,2k)$ such that
\begin{equation*}
	J=(I\backslash\{i_{1},\dots,i_{\alpha}\})\cup\{j_{1},\dots,j_{\alpha}\}.
\end{equation*}
In addition, we denote by
\begin{equation*}
	(I^{c}\cap J^{c})\backslash\tilde{I}=\{\ell_{1},\dots,\ell_{\beta}\},
\end{equation*}
 with $\beta$ satisfying $0\le\alpha+\beta\le 2k$. 
Hence, for a fixed index set $I\subseteq[p+n]$,
\begin{multline*}
	\sum_{\substack{J\subseteq[p+n]\\|J|=m}} (-1)^{s(I,J)}\det([-z^{-1}\mathcal{X}]_{I,J})\det([1-z^{-1}\mathcal{S}]_{I^c,J^c}) \\
	= \sum_{\substack{0\le\alpha+\beta\le 2k\\ 0\le\alpha\le m}} (-1)^{\alpha} (-z)^{-(m+\alpha+\beta)} \sum_{\substack{\{i_{1},\cdots,i_{\alpha}\}\subseteq I\\ \{j_{1},\cdots,j_{\alpha},\ell_{1},\cdots,\ell_{\beta}\}\subseteq I^{c} }}  \det([\mathcal{X}]_{I,(I\backslash\{i_{1},\cdots,i_{\alpha}\})\cup\{j_{1},\cdots,j_{\alpha}\}}) \\
	\times \det([\mathcal{S}]_{\{j_{1},\cdots,j_{\alpha},\ell_{1},\cdots,\ell_{\beta}\},\{i_{1},\cdots,i_{\alpha},\ell_{1},\cdots,\ell_{\beta}\}}).
\end{multline*}
Then,
\begin{equation}\label{eq: 1588}
	\det(1-z^{-1}\mathcal{X}-z^{-1}\mathcal{S}) = \sum_{m=0}^{p+n}\sum_{\substack{0\le\alpha+\beta\le 2k\\ 0\le\alpha\le m}} (-1)^{\alpha} (-z)^{-(m+\alpha+\beta)} \mathcal{P}^{(n)}_{m,\alpha,\beta},
\end{equation}
where
\begin{multline}\label{eq: 1592}
	\mathcal{P}^{(n)}_{m,\alpha,\beta} = \sum_{\substack{I\subseteq[p+n]\\|I|=m}} \sum_{\substack{\{i_{1},\cdots,i_{\alpha}\}\subseteq I\\ \{j_{1},\cdots,j_{\alpha},\ell_{1},\cdots,\ell_{\beta}\}\subseteq I^{c} }}  \det([\mathcal{X}]_{I,(I\backslash\{i_{1},\cdots,i_{\alpha}\})\cup\{j_{1},\cdots,j_{\alpha}\}}) \\
	\times \det([\mathcal{S}]_{\{j_{1},\cdots,j_{\alpha},\ell_{1},\cdots,\ell_{\beta}\},\{i_{1},\cdots,i_{\alpha},\ell_{1},\cdots,\ell_{\beta}\}}).
\end{multline}
Then we simply bound the determinant as
\begin{equation*}
	|\det(1-z^{-1}\mathcal{X}-z^{-1}\mathcal{S})|
	\le
	\sum_{m=0}^{p+n}\sum_{\substack{0\le\alpha+\beta\le 2k\\ 0\le\alpha\le m}}
	|z|^{-(m+\alpha+\beta)}
	|\mathcal{P}^{(n)}_{m,\alpha,\beta}|
\end{equation*}
Hence, in order to bound the supremum (over $z$) of the determinant, it is enough to bound RHS for $|z|= \sqrt{(p/n)^{1/2}+\delta/2}$. We nevertheless proceed with the estimate for any $|z|\ge \sqrt{(p/n)^{1/2}+\delta/2}$ for notational brevity.

By Assumption \ref{main assump} (i), there is a (small) constant $\mathfrak{c}_{0}\in (0,1)$ independent of $n$ such that $\mathfrak{c}_{0}<p/n$ for all $n$.
For $|z|\ge \sqrt{(p/n)^{1/2}+\delta/2}$,
\begin{multline*}
	\E[|\det(1-z^{-1}\mathcal{X}-z^{-1}\mathcal{S})|^{2}] \\
	\le
	\sum_{\substack{0\le m\le p+n \\ 0\le m'\le p+n}}
	\sum_{\substack{0\le\alpha+\beta\le 2k \\ 0\le\alpha'+\beta'\le 2k}}
	|z|^{-(m+\alpha+\beta)}|z|^{-(m'+\alpha'+\beta')}(\E[|\mathcal{P}^{(n)}_{m,\alpha,\beta}|^{2}]\E[|\mathcal{P}^{(n)}_{m',\alpha',\beta'}|^{2}])^{1/2} \\
	\le
	\sum_{\substack{0\le m\le p+n \\ 0\le m'\le p+n}}
	(2k)^{4}((p/n)^{1/2}+\delta/2)^{-(m+m')/2}\mathfrak{c}_{0}^{-k}
	\max_{\substack{0\le\alpha+\beta\le 2k \\ 0\le\alpha'+\beta'\le 2k}}
	(\E[|\mathcal{P}^{(n)}_{m,\alpha,\beta}|^{2}]
	\E[|\mathcal{P}^{(n)}_{m',\alpha',\beta'}|^{2}])^{1/2}.
\end{multline*}
Thus it suffices to show the moment bound
\begin{equation}\label{eq: 1614}
	\max_{0\le\alpha+\beta\le 2k}\E[|\mathcal{P}^{(n)}_{m,\alpha,\beta}|^{2}] \le C_{k} \cdot (p/n)^{m/2},  
\end{equation}
for some constant $C_{k}$ depending only on $k$.

~

We observe that, for $I\neq I'$,
\begin{equation*}
	\E[ \det([\mathcal{X}]_{I,(I\backslash\{i_{1},\cdots,i_{\alpha}\})\cup\{j_{1},\cdots,j_{\alpha}\}}) \det([\mathcal{X}]_{I',(I'\backslash\{i'_{1},\cdots,i'_{\alpha}\})\cup\{j'_{1},\cdots,j'_{\alpha}\}}) ] = 0.
\end{equation*}
Moreover, for $\{i_{1},\cdots,i_{\alpha}\}, \{i'_{1},\cdots,i'_{\alpha}\}\subseteq I$ and $\{j_{1},\cdots,j_{\alpha}\}, \{j'_{1},\cdots,j'_{\alpha}\}\subseteq I^{c}$,
if $\{i_{1},\cdots,i_{\alpha}\}\neq\{i'_{1},\cdots,i'_{\alpha}\}$ or $\{j_{1},\cdots,j_{\alpha}\}\neq\{j'_{1},\cdots,j'_{\alpha}\}$, it follows that
\begin{equation*}
	\E[ \det([\mathcal{X}]_{I,(I\backslash\{i_{1},\cdots,i_{\alpha}\})\cup\{j_{1},\cdots,j_{\alpha}\}}) \det([\mathcal{X}]_{I,(I\backslash\{i'_{1},\cdots,i'_{\alpha}\})\cup\{j'_{1},\cdots,j'_{\alpha}\}}) ] = 0.
\end{equation*}
Then, from (\ref{eq: 1592}) we have
\begin{multline}\label{eq: 1615}
	\E[|\mathcal{P}^{(n)}_{m,\alpha,\beta}|^{2}]
	= \sum_{|I|=m} \sum_{\substack{\{i_{1},\cdots,i_{\alpha}\}\subseteq I\\ \{j_{1},\cdots,j_{\alpha}\}\subseteq I^{c} }}
	\sum_{\{\ell_{1},\cdots,\ell_{\beta}\},\{\ell'_{1},\cdots,\ell'_{\beta}\}\subseteq I^{c}\backslash\{j_{1},\cdots,j_{\alpha}\}} \E[|\det([\mathcal{X}]_{I,(I\backslash\{i_{1},\cdots,i_{\alpha}\})\cup\{j_{1},\cdots,j_{\alpha}\}})|^{2}]\\
	\times \det([\mathcal{S}]_{\{j_{1},\cdots,j_{\alpha},\ell_{1},\cdots,\ell_{\beta}\},\{i_{1},\cdots,i_{\alpha},\ell_{1},\cdots,\ell_{\beta}\}})
	\times \overline{\det([\mathcal{S}]_{\{j_{1},\cdots,j_{\alpha},\ell'_{1},\cdots,\ell'_{\beta}\},\{i_{1},\cdots,i_{\alpha},\ell'_{1},\cdots,\ell'_{\beta}\}})}.
\end{multline}
For each entry of $\mathcal{S}$, in light of (\ref{062704}), we can write
\begin{equation}
	\mathcal{S}(i,j) = \sum_{s=\pm 1, \ldots, \pm k} d_{s}w_{s}(i)w_{s}(j) . \label{071001}
\end{equation}
Let $[\mathcal{S}]_{(m_{1},\cdots,m_{\gamma}),(m'_{1},\cdots,m'_{\gamma})}$ be the matrix defined by
\begin{equation*}
	[\mathcal{S}]_{(m_{1},\cdots,m_{\gamma}),(m'_{1},\cdots,m'_{\gamma})}(i,j) = \mathcal{S}(m_{i},m'_{j}),
	\quad 1\le i,j\le \gamma.
\end{equation*}
Note that $[\mathcal{S}]_{(m_{1},\cdots,m_{\gamma}),(m'_{1},\cdots,m'_{\gamma})}$ can be obtained from $[\mathcal{S}]_{\{m_{1},\cdots,m_{\gamma}\},\{m'_{1},\cdots,m'_{\gamma}\}}$ by reordering rows and columns.
Moreover, we notice that
\begin{equation}\label{eq: 1632}
	|\det([\mathcal{S}]_{(m_{1},\cdots,m_{\gamma}),(m'_{1},\cdots,m'_{\gamma})})|
	= |\det([\mathcal{S}]_{\{m_{1},\cdots,m_{\gamma}\},\{m'_{1},\cdots,m'_{\gamma}\}})|.
\end{equation}
For each tuple $(i_{1},\cdots,i_{\alpha},\ell_{1},\cdots,\ell_{\beta})$ and $\sigma\in\text{Perm}([\alpha+\beta])$, we define the one-to-one correspondence $\tilde{\sigma}=\tilde{\sigma}_{(i_{1},\cdots,i_{\alpha},\ell_{1},\cdots,\ell_{\beta})}: \{1,\cdots,\alpha+\beta\}\to\{i_{1},\cdots,i_{\alpha},\ell_{1},\cdots,\ell_{\beta}\}$ by
\begin{equation*}
	\tilde{\sigma}(m) =
	\begin{cases}
		i_{\sigma(m)} & 1\le\sigma(m)\le\alpha, \\
		\ell_{\sigma(m)-\alpha} & \alpha+1\le\sigma(m)\le\alpha+\beta.
	\end{cases}
\end{equation*}
According to (\ref{071001}), it is not difficult to verify that the determinant 
$$\det(\mathcal{S}_{(j_{1},\cdots,j_{\alpha},\ell_{1},\cdots,\ell_{\beta}),(i_{1},\cdots,i_{\alpha},\ell_{1},\cdots,\ell_{\beta})})$$
consists of $(2k)^{\alpha+\beta}\times(\alpha+\beta)!$ terms of the form
\begin{equation}\label{eq: 2823}
	\prod_{\zeta=1}^{\alpha} d_{b_{\zeta}}w_{b_{\zeta}}(j_{\zeta})w_{b_{\zeta}}(\tilde{\sigma}(\zeta))
	\prod_{\xi=1}^{\beta} d_{b_{\alpha+\xi}}w_{b_{\alpha+\xi}}(\ell_{\xi})w_{b_{\alpha+\xi}}(\tilde{\sigma}(\alpha+\xi))
\end{equation}
where $b_{1},\cdots,b_{\alpha+\beta}\in[\pm k]=\{1,\cdots,k\}\cup\{-1,\cdots,-k\}$.
Recalling \eqref{eq: 1632} and using $\sigma^{-1}$ to align
$(\tilde{\sigma}(1),\cdots,\tilde{\sigma}(\alpha+\beta))$
with
$(i_{1},\cdots,i_{\alpha},\ell_{1},\cdots,\ell_{\beta}),$
we have
\begin{multline*}
	\sum_{\substack{\{i_{1},\cdots,i_{\alpha}\}\subseteq I\\ \{j_{1},\cdots,j_{\alpha}\}\subseteq I^{c} }}
	\sum_{\{\ell_{1},\cdots,\ell_{\beta}\},\{\ell'_{1},\cdots,\ell'_{\beta}\}\subseteq I^{c}\backslash\{j_{1},\cdots,j_{\alpha}\}} |\det([\mathcal{S}]_{\{j_{1},\cdots,j_{\alpha},\ell_{1},\cdots,\ell_{\beta}\},\{i_{1},\cdots,i_{\alpha},\ell_{1},\cdots,\ell_{\beta}\}})| \\
	\times |\det([\mathcal{S}]_{\{j_{1},\cdots,j_{\alpha},\ell'_{1},\cdots,\ell'_{\beta}\},\{i_{1},\cdots,i_{\alpha},\ell'_{1},\cdots,\ell'_{\beta}\}})|\\
	\le \sum_{\substack{b_{1},\cdots,b_{\alpha+\beta}\in[\pm k]\\\sigma\in\text{Perm}([\alpha+\beta])}} \sum_{\substack{b'_{1},\cdots,b'_{\alpha+\beta}\in[\pm k]\\\sigma'\in\text{Perm}([\alpha+\beta])}} \sum_{\substack{\{i_{1},\cdots,i_{\alpha}\}\subseteq I\\ \{j_{1},\cdots,j_{\alpha}\}\subseteq I^{c} }}
	\sum_{\{\ell_{1},\cdots,\ell_{\beta}\},\{\ell'_{1},\cdots,\ell'_{\beta}\}\subseteq I^{c}\backslash\{j_{1},\cdots,j_{\alpha}\}}
	\prod_{\zeta=1}^{\alpha+\beta}|d_{b_{\zeta}}d_{b'_{\zeta}}| \\
	\times
	\prod_{\zeta=1}^{\alpha} |w_{b_{\zeta}}(j_{\zeta})w_{b_{\sigma^{-1}(\zeta)}}(i_{\zeta})|
	\prod_{\xi=1}^{\beta} |w_{b_{\alpha+\xi}}(\ell_{\xi})w_{b_{\sigma^{-1}(\alpha+\xi)}}(\ell_{\xi})| \\
	\times
	\prod_{\zeta=1}^{\alpha} |w_{b'_{\zeta}}(j_{\zeta})w_{b'_{(\sigma')^{-1}(\zeta)}}(i_{\zeta})|
	\prod_{\xi=1}^{\beta} |w_{b'_{\alpha+\xi}}(\ell'_{\xi})w_{b'_{(\sigma')^{-1}(\alpha+\xi)}}(\ell'_{\xi})|.
\end{multline*}
Notice that
\begin{multline*}
	\prod_{\zeta=1}^{\alpha} |w_{b_{\zeta}}(j_{\zeta})w_{b_{\sigma^{-1}(\zeta)}}(i_{\zeta})|
	\prod_{\xi=1}^{\beta} |w_{b_{\alpha+\xi}}(\ell_{\xi})w_{b_{\sigma^{-1}(\alpha+\xi)}}(\ell_{\xi})| \\
	\times
	\prod_{\zeta=1}^{\alpha} |w_{b'_{\zeta}}(j_{\zeta})w_{b'_{(\sigma')^{-1}(\zeta)}}(i_{\zeta})|
	\prod_{\xi=1}^{\beta} |w_{b'_{\alpha+\xi}}(\ell'_{\xi})w_{b'_{(\sigma')^{-1}(\alpha+\xi)}}(\ell'_{\xi})| \\
	\le
	\prod_{\zeta=1}^{\alpha} |w_{b_{\zeta}}(j_{\zeta})w_{b_{\sigma^{-1}(\zeta)}}(i_{\zeta})|^{2}
	\prod_{\xi=1}^{\beta} |w_{b_{\alpha+\xi}}(\ell_{\xi})w_{b'_{\alpha+\xi}}(\ell'_{\xi})|^{2}
	\\
	+ \prod_{\zeta=1}^{\alpha} |w_{b'_{\zeta}}(j_{\zeta})w_{b'_{(\sigma')^{-1}(\zeta)}}(i_{\zeta})|^{2}
	\prod_{\xi=1}^{\beta} |w_{b_{\sigma^{-1}(\alpha+\xi)}}(\ell_{\xi})w_{b'_{(\sigma')^{-1}(\alpha+\xi)}}(\ell'_{\xi})|^{2}.
\end{multline*}
By summing over $i_{1},\cdots,i_{\alpha},j_{1},\cdots,j_{\alpha},\ell_{1},\cdots,\ell_{\beta},\ell'_{1},\cdots,\ell'_{\beta}$ (noting that, for a fixed $I$, we have $i_{\cdot}\in I$ and $j_{\cdot},\ell_{\cdot},\ell'_{\cdot}\in I^{c}$), we can see that (recalling that $|d_{i}|\le C$ for all $i$)
\begin{multline}\label{eq: 2890}
	\sum_{\substack{\{i_{1},\cdots,i_{\alpha}\}\subseteq I\\ \{j_{1},\cdots,j_{\alpha}\}\subseteq I^{c} }}
	\sum_{\{\ell_{1},\cdots,\ell_{\beta}\},\{\ell'_{1},\cdots,\ell'_{\beta}\}\subseteq I^{c}\backslash\{j_{1},\cdots,j_{\alpha}\}} |\det([\mathcal{S}]_{\{j_{1},\cdots,j_{\alpha},\ell_{1},\cdots,\ell_{\beta}\},\{i_{1},\cdots,i_{\alpha},\ell_{1},\cdots,\ell_{\beta}\}})| \\
	\times |\det([\mathcal{S}]_{\{j_{1},\cdots,j_{\alpha},\ell'_{1},\cdots,\ell'_{\beta}\},\{i_{1},\cdots,i_{\alpha},\ell'_{1},\cdots,\ell'_{\beta}\}})|
	\le 2C^{2(\alpha+\beta)}(2k)^{2(\alpha+\beta)}[(\alpha+\beta)!]^{2}.
\end{multline}

We claim that
\begin{equation}\label{eq: 2898}
	\det([\mathcal{X}]_{I,(I\backslash\{i_{1},\cdots,i_{\alpha}\})\cup\{j_{1},\cdots,j_{\alpha}\}})=0 \quad
	\text{if $||I\cap \{1,\cdots,p\}|-m/2| > \alpha/2$.}
\end{equation}
This vanishing property follows from the block structure
\begin{equation*}
	\mathcal{X}
	=
	\begin{pmatrix}
		0 & X_{1} \\
		X_{2} & 0
	\end{pmatrix},
\end{equation*}
which implies $\det([\mathcal{X}]_{I,J})=0$ if $|I\cap \{1,\cdots,p\}|\neq |J\cap \{p+1,\cdots,p+n\}|$.
Taking
$J=(I\backslash\{i_{1},\cdots,i_{\alpha}\})\cup\{j_{1},\cdots,j_{\alpha}\}$
and noting that $|I|=m$, one can find the stated condition
$\bigl||I\cap \{1,\cdots,p\}|-m/2\bigr|>\alpha/2$.
In addition, for any choice of indices $\{i_{1},\cdots,i_{\alpha}\}\subseteq I$ and $\{j_{1},\cdots,j_{\alpha}\}\subseteq I^{c}$, one has
\begin{equation}\label{eq: 2917}  
	\E[|\det([\mathcal{X}]_{I,(I\backslash\{i_{1},\cdots,i_{\alpha}\})\cup\{j_{1},\cdots,j_{\alpha}\}})|^{2}]\le n^{-m} \mathfrak{a}! (m-\mathfrak{a})!,
\end{equation}
for some integer $\mathfrak{a}$ satisfying $|\mathfrak{a}-m/2|\le\alpha/2$.
Combining \eqref{eq: 2823}, \eqref{eq: 2890}, \eqref{eq: 2898}, and \eqref{eq: 2917},
\begin{multline*}
	\E[|\mathcal{P}^{(n)}_{m,\alpha,\beta}|^{2}]
	\le
	\sum_{\gamma=-\lceil\alpha/2\rceil}^{\lceil\alpha/2\rceil}
	\sum_{\substack{|I|=m\\|I\cap[p]|=\lfloor m/2\rfloor+\gamma}}
	\sum_{\substack{\{i_{1},\cdots,i_{\alpha}\}\subseteq I\\ \{j_{1},\cdots,j_{\alpha}\}\subseteq I^{c} }}
	\sum_{\{\ell_{1},\cdots,\ell_{\beta}\},\{\ell'_{1},\cdots,\ell'_{\beta}\}\subseteq I^{c}\backslash\{j_{1},\cdots,j_{\alpha}\}} \\
	n^{-m} (\lfloor m/2\rfloor+\gamma)!(m-(\lfloor m/2\rfloor-\gamma))!
	\times \det([\mathcal{S}]_{\{j_{1},\cdots,j_{\alpha},\ell_{1},\cdots,\ell_{\beta}\},\{i_{1},\cdots,i_{\alpha},\ell_{1},\cdots,\ell_{\beta}\}}) \\
	\times \overline{\det([\mathcal{S}]_{\{j_{1},\cdots,j_{\alpha},\ell'_{1},\cdots,\ell'_{\beta}\},\{i_{1},\cdots,i_{\alpha},\ell'_{1},\cdots,\ell'_{\beta}\}})} \\
	\le
	\sum_{\gamma=-\lceil\alpha/2\rceil}^{\lceil\alpha/2\rceil}
	{p\choose \lfloor m/2\rfloor+\gamma} {n\choose m-\lfloor m/2\rfloor - \gamma}
	n^{-m}
	(\lfloor m/2\rfloor+\gamma)!(m-(\lfloor m/2\rfloor-\gamma))! \\
	\times 2C^{2(\alpha+\beta)}(2k)^{2(\alpha+\beta)}[(\alpha+\beta)!]^{2}.
\end{multline*}
By Assumption \ref{main assump} (i), we have $p/n < 2c$ for large $n$.
Therefore, together with the simple fact ${n\choose a}a!\le n^{a}$,
there exists a constant $K_{c,k}>0$, depending on $c$ and $k$, such that
\begin{align}\label{eq: 1717}
	\E[|\mathcal{P}^{(n)}_{m,\alpha,\beta}|^{2}] &\le 4 \alpha K_{c,k} \Big(\frac{p}{n}\Big)^{m/2} C^{2(\alpha+\beta)} (2k)^{2(\alpha+\beta)}[(\alpha+\beta)!]^{2} \nonumber \\
	&\le 2^{4k+2} K_{c,k} \Big(\frac{p}{n}\Big)^{m/2} C^{4k} k^{4k+1}[(2k)!]^{2},
\end{align}
which establishes \eqref{eq: 1614} and completes the proof of part (ii).

~

(iii)
Recall that
\begin{equation*}
	\det(1-z^{-1}\mathcal{X}) - \det(1-z^{-1}\mathcal{X}^{(M)}) = \sum_{m=1}^{p+n} (-1)^{m}z^{-m}
	(\mathcal{P}^{(n)}_{m}-\mathcal{P}^{(n,M)}_{m}).
\end{equation*}
For a constant $m_{0}>1$, by Markov's inequality,
\begin{multline*}
	\mathbb{P}\Big\{\sum_{m>m_{0}} (\sqrt{p/n}+\delta/2)^{-m/2}|\mathcal{P}_{m}^{(n)}| > c/3 \Big\} \\
	\le (c/3)^{-2} \sum_{m,m'>m_{0}} (\sqrt{p/n}+\delta/2)^{-(m+m')/2} \E[|\mathcal{P}_{m}^{(n)} \mathcal{P}_{m'}^{(n)}|].
\end{multline*}
Due to \eqref{eq: 697}, it follows that
\begin{equation*}
	\E[|\mathcal{P}_{m}^{(n)} \mathcal{P}_{m'}^{(n)}|]
	\le \E[|\mathcal{P}_{m}^{(n)}|^{2}]^{1/2} \E[|\mathcal{P}_{m'}^{(n)}|^{2}]^{1/2}
	\le \Big(\frac{p}{n}\Big)^{(m+m')/4}.
\end{equation*}
Thus one can notice that
\begin{equation*}
	\mathbb{P}\Big\{\sum_{m>m_{0}} (\sqrt{p/n}+\delta/2)^{-m/2}|\mathcal{P}_{m}^{(n)}| > c/3 \Big\}
	=
	O\bigg( (c/3)^{-2} \bigg(\frac{(p/n)^{1/4}}{(\sqrt{p/n}+\delta/2)^{1/2}}\bigg)^{2m_{0}} \bigg).
\end{equation*}
By choosing $m_{0}$ sufficiently large, we have
\begin{equation*}
	\mathbb{P}\Big\{\sum_{m>m_{0}} (\sqrt{p/n}+\delta/2)^{-m/2}|\mathcal{P}_{m}^{(n)}| > c/3 \Big\} 
	< \epsilon/3.
\end{equation*}
Similarly,
\begin{equation*}
	\mathbb{P}\Big\{\sum_{m>m_{0}} (\sqrt{p/n}+\delta/2)^{-m/2}|\mathcal{P}_{m}^{(n,M)}| > c/3 \Big\} 
	< \epsilon/3.
\end{equation*}
Thus what remains is to bound the following sum:
\begin{equation*}
	\sum_{m=1}^{m_{0}} (\sqrt{p/n}+\delta/2)^{-m/2} |\mathcal{P}_{m}^{(n)}-\mathcal{P}_{m}^{(n,M)}|.
\end{equation*}
Since \eqref{eq: 1244} has already been established, the first claim in part (iii) follows immediately.

~

Applying the same reasoning together with \eqref{eq: 1717},
the second claim of part (iii) also boils down to showing the following result:
\begin{equation}\label{eq: 805}
	\lim_{M\to\infty}\E[|\mathcal{P}^{(n)}_{m,\alpha,\beta}-\mathcal{P}^{(n,M)}_{m,\alpha,\beta}|^{2}] = 0.
\end{equation}
Applying the reasoning used in part (ii), we notice that
\begin{multline*}
	\E[|\mathcal{P}^{(n)}_{m,\alpha,\beta}-\mathcal{P}^{(n,M)}_{m,\alpha,\beta}|^{2}]
	= \sum_{|I|=m} \sum_{\substack{\{i_{1},\cdots,i_{\alpha}\}\subseteq I\\ \{j_{1},\cdots,j_{\alpha}\}\subseteq I^{c} }}
	\sum_{\{\ell_{1},\cdots,\ell_{\beta}\},\{\ell'_{1},\cdots,\ell'_{\beta}\}\subseteq I^{c}\backslash\{j_{1},\cdots,j_{\alpha}\}} \\
	\E[|\det([\mathcal{X}]_{I,(I\backslash\{i_{1},\cdots,i_{\alpha}\})\cup\{j_{1},\cdots,j_{\alpha}\}})-\det([\mathcal{X}^{(M)}]_{I,(I\backslash\{i_{1},\cdots,i_{\alpha}\})\cup\{j_{1},\cdots,j_{\alpha}\}})|^{2}]\\
	\times \det([\mathcal{S}]_{\{j_{1},\cdots,j_{\alpha},\ell_{1},\cdots,\ell_{\beta}\},\{i_{1},\cdots,i_{\alpha},\ell_{1},\cdots,\ell_{\beta}\}})
	\times \overline{\det([\mathcal{S}]_{\{j_{1},\cdots,j_{\alpha},\ell'_{1},\cdots,\ell'_{\beta}\},\{i_{1},\cdots,i_{\alpha},\ell'_{1},\cdots,\ell'_{\beta}\}})}.
\end{multline*}
By the same argument as in \eqref{eq: 1717}, we conclude that 
\begin{equation}\label{eq: 819}
	\E[|\mathcal{P}^{(n)}_{m,\alpha,\beta}-\mathcal{P}^{(n,M)}_{m,\alpha,\beta}|^{2}]
	\le 2^{4k+1} K_{c,k} \Big(\frac{p}{n}\Big)^{m/2} C^{4k} k^{4k+1}[(2k)!]^{2} \times \E\big(\Big|\prod_{j=1}^{m}a_{1j}-\prod_{j=1}^{m}a_{1j}^{(M)}\Big|^{2}\big),
\end{equation}
where $a_{ij}=\sqrt{n}x_{ij}$ and $a^{(M)}_{ij}=\sqrt{n}x^{(M)}_{ij}$.
This provides \eqref{eq: 805} so completes the proof.

\section*{Acknowledgement} Z.~Bao is grateful to Johannes Alt and Torben Kr\"{u}ger for explaining the work \cite{AEKN18} and several related works.  We would also like to thank Zheng Tracy Ke for many helpful discussions and Xiucai Ding for reference.  Z.~Bao is supported by Hong Kong RGC Grant GRF 16304724, NSFC12222121 and NSFC12271475. K.~Cheong, J.~Lee and Y.~Li are supported by Hong Kong RGC Grant 16303922.


%
%
%
%
%
%
%
%

\begin{appendix}

\section{Spectral norm estimate: Proof of Proposition \ref{pro.spectralradius}} \label{s.spectralradius}

 In this section we will prove Proposition \ref{pro.spectralradius}.  We will first prove the result on $X_1X_2^*$ without using Proposition \ref{quadraticformbound}; see Proposition \ref{prop_upper_bound} and its proof below. Then, we prove those bounds for $\Sigma X_1X_2^* \Sigma^*$ with the aid of Proposition \ref{quadraticformbound}; see Proposition \ref{prop_upper_bound_spike} and its proof below. We remark here that the proof of Proposition \ref{quadraticformbound} in Section \ref{s.proof of quadratic form bound}  will need those bounds in Proposition \ref{prop_upper_bound}.

Recall the linearization of $X_1X_2^*$ from (\ref{def of X_0}) and its variance profile (\ref{variance profile}). For simplicity, in this section we denote by 
\begin{align}
\mathbb{T}=n^{-1}T \label{bbT}
\end{align} 
the variance profile of $X_1$ and $X_2$ and we recall the flatness assumption in (\ref{flatness_assumption}).

We will follow the strategy developed in \cite{AEKN18, AEK18, AEK21}. Especially, our matrix $\mathcal{X}_0$ can be regarded as a special case of the general model considered in \cite{AEKN18}. Based on the result for $\mathcal{X}_0$, we then derive the results for the model $\mathcal{X}$ via a perturbation argument.  

A standard strategy to study the spectrum of a non-Hermitian random matrix is via the Girko's Hermitization/linearization. Specifically, the spectrum of the $(n+p)\times (n+p)$-matrix $\mathcal{X}_0$ can be studied by analyzing the following $2(n+p)\times 2(n+p)$ Hermitian matrix
\begin{equation*}
	\mathbf{H}^z=\left(
\begin{array}{ccc}
~ & \mathcal{X}_0-z\\
\mathcal{X}_0^*-\bar{z}
 &~
\end{array}
\right),
\end{equation*}
where $z$ is a generic complex number. It is known that the possible eigenvalues of $\mathcal{X}_0$ around $z$ can be studied via the spectrum of $\mathbf{H}^z$ around $0$. Heuristically, if the eigenvalues of $\mathbf{H}^z$ are aways from $0$, the eigenvalues of $\mathcal{X}_0$ will be away from $z$.  The spectrum of $\mathbf{H}^z$ can then be studied via its Green function $\mathbf{G}^z(\omega)=(\mathbf{H}^z-\omega)^{-1}$.  Following the study in \cite{AEKN18, AEK18, AEK21}, when the variance profile $\mathcal{V}$ is general, one needs to consider the solution to the following Matrix Dyson Equation, which shall be regarded as an approximation of the Green function $\mathbf{G}^z(i \eta)$ for any $\eta>0$.  It is defined as
\begin{equation}\label{MDE}
	-\mathbf{M}^z(i\eta)^{-1}=i\eta \mathbf{1}-\textbf{A}^z+\mathcal{V}[\mathbf{M}^z(i\eta)]
\end{equation}
where in our case
\begin{equation*}
	\textbf{A}^z=\left(
\begin{array}{ccc}
0 & -z\\
-\bar{z} &0
\end{array}
\right).
\end{equation*}
Here the functional $\mathcal{V}[\cdot]$ is defined as 
\begin{equation*}
	\mathcal{V}[\textbf{W}]=\left(
\begin{array}{ccc}
\text{diag}(\mathcal{V}w_2) & 0\\
0 &\text{diag}(\mathcal{V}^*w_1)
\end{array}
\right)
\end{equation*}
for any $2(n+p)\times 2(n+p)$ matrix $\textbf{W}$ such that
\begin{align*}
	&\textbf{W}=(w_{ij})_{i,j=1}^{2(n+p)}\in \mathbb{C}^{2(n+p)\times 2(n+p)}, \notag\\
	&w_1=(w_{ii})_{i=1}^{n+p}\in \mathbb{C}^{n+p}, \qquad w_2=(w_{ii})_{i=n+p+1}^{2(n+p)}\in \mathbb{C}^{n+p}. 
\end{align*} 
It is shown in \cite{HFS07} that (\ref{MDE}) has a  unique solution under the constraint that $\Im \mathbf{M}^z:=(\mathbf{M}^z-(\mathbf{M}^z)^*)/2i$ is positive definite.
According to \cite{AEK19}, we define the self-consistent density of states $\rho^z$ of $\mathbf{H}^z$ as the unique measure whose Stietjes transform is $\frac{1}{2(p+n)}\Tr \mathbf{M^z}(\omega)$. More precisely, the solution $\mathbf{M}^z(i\eta)$ to (\ref{MDE}) has the Stieltjes transform representation
\begin{equation*}
	\mathbf{M}^z(i\eta)=\int_{\mathbb{R}}\frac{\mathbf{V}({\rm d}x)}{{x}-i\eta}
\end{equation*}
where $\mathbf{V}$ is a matrix-valued, compactly supported measure on $\mathbb{R}$. Then, 
\begin{equation*}
	\rho^z(d{x}):=\frac{1}{2(n+p)}\Tr \mathbf{V}({\rm d}x).
\end{equation*}
Let $m^z_j(i\eta)$ be the $j$-th diagonal entry of the matrix $\mathbf{M}^z(i\eta)$. For any $\tau>0$, define the set
\begin{equation*}
	\mathbb{D}_{\tau}:=\{z\in \mathbb{C}:\text{dist}(0,\text{supp }\rho^z)\leq \tau\}
\end{equation*}
and
\begin{equation}\label{tilde_D}
	\widetilde{\mathbb{D}}_{\tau}:=\{z:\limsup_{\eta \to 0}\frac{1}{\eta}\max_j |\im m_j^z(i\eta)|\geq \frac{1}{\tau}\}
\end{equation}
$\mathbb{D}_{\tau}$ is called the self-consistent $\tau$-pseudospectrum of $\mathcal{X}_0$. It is shown in \cite{AEKN18} that eigenvalues of $\mathcal{X}_0$ will concentrate on $\mathbb{D}_{\tau}$ for any fixed $\tau>0$. It is also shown in \cite{AEKN18} that $\mathbb{D}_{\tau}$ and $\widetilde{\mathbb{D}}_{\tau}$
are comparable in the sense that for any $\tau >0$, we have $\mathbb{D}_{\tau_1}\subseteq \widetilde{\mathbb{D}}_{\tau}\subseteq\mathbb{D}_{\tau_2}$ for certain $\tau_1$, $\tau_2>0$. 

Further, in our case, the matrix equation (\ref{MDE}) implies that $\mathbf{M}^z$ has the block structure such that its top-left, top-right, lower-left and lower-right $(n+p)\times(n+p)$ blocks are all diagonal matrices. After simplification, the Matrix Dyson Equation (\ref{MDE}) admit the following form for some vectors $u$ and $v$,  
\begin{equation}\label{MDE_Solution}
	\mathbf{M}^z(i\eta)=\left(
\begin{array}{ccc}
\text{diag}(iu) & -z\text{diag}(\frac{u}{\eta+\mathcal{V}^*u})\\
-\bar z \text{diag}(\frac{v}{\eta+\mathcal{V}v}) & \text{diag}(iv)
\end{array}
\right)
\end{equation}
Therefore, to determine $\widetilde{\mathbb{D}}_{\tau}$, as well as $\mathbb{D}_{\tau}$, we only need to analyze the following coupled vector equations derived from (\ref{MDE}) and (\ref{MDE_Solution}) via Schur complement 
\begin{align}\label{MDE1}
	\frac{1}{u}&=\eta +\mathcal{V}v+\frac{|z|^2}{\eta +\mathcal{V}^*u}, \nonumber\\
	\frac{1}{v}&=\eta +\mathcal{V}^*u+\frac{|z|^2}{\eta +\mathcal{V}v}.
\end{align}
Here $u,v\in \mathbb{R}^{p+n}_+$ are the unique solutions to the vector equations. The uniqueness and existence of the positive solutions to this vector equation is a consequence of the solution to the Matrix Dyson Equation. Bijection between the solution to the Matrix Dyson Equation (\ref{MDE}) with positive definite $\Im \mathbf{M}^z$ and the positive solutions to the vector equation (\ref{MDE1}) is proved in \cite{AEK18}. One can check from (\ref{MDE_Solution}) that the diagonal part of $\mathbf{M}^z$ is in fact purely imaginary and the imaginary part $\Im m^z_j(i\eta)$ form the vector $u$ and $v$. Using the block structure of $\mathcal{V}$, one can further write the system of vector equations (\ref{MDE1}) as a system of four equations, with the notation (\ref{bbT}), 
\begin{align}\label{MDE2}
	\frac{1}{u_1}&=\eta +\mathbb{T}v_2+\frac{|z|^2}{\eta +\mathbb{T}u_2} \nonumber\\
	\frac{1}{u_2}&=\eta +\mathbb{T}^*v_1+\frac{|z|^2}{\eta +\mathbb{T}^*u_1} \nonumber\\
	\frac{1}{v_1}&=\eta +\mathbb{T}u_2+\frac{|z|^2}{\eta +\mathbb{T}v_2} \nonumber\\
	\frac{1}{v_2}&=\eta +\mathbb{T}^*u_1+\frac{|z|^2}{\eta +\mathbb{T}^*v_1}
\end{align}  
Here 
\begin{align*}
u=\binom{u_1}{u_2}, \qquad v=\binom{v_1}{v_2}, \qquad u_1, v_1 \in \mathbb{R}^p_+, \quad u_2, v_2\in \mathbb{R}^n_+.
\end{align*}

According to Theorem 2.4 and Remark 2.5 (iv) in \cite{AEKN18} , to get the upper bound 
\begin{align}
\rho(\mathcal{X}_0)\leq \sqrt{\rho(\mathcal{V})}+\delta, \label{rho bound}
\end{align} 
it suffices to show that 
\begin{align}
\limsup_{\eta \to 0}\frac{1}{\eta}\max_j |\im m_j^z(i\eta)|< \frac{1}{\delta'} \label{012301}
\end{align}
 for some $\delta'>0$, given that $|z|\geq \sqrt{\rho(\mathcal{V})}+\delta$. If (\ref{012301}) holds, by the equivalence of $\mathbb{D}_{\tau}$ and $\widetilde{\mathbb{D}}_{\tau}$, the eigenvalues of $\mathbf{H}^z$ is away from zero by a distance of $\delta'$. We can then conclude from Theorem 4.7 in \cite{AEKN18} that $\mathbf{H}^z$ is invertible and the resolvent at 0 is bounded by a constant, if (\ref{012301}) is granted. Specifically,  we will have the following bound on the operator norm of $(\mathcal{X}_0-z)^{-1}$ and the Green functions $G(z^2)=(X_1X_2^*-z^2)^{-1}$, $\mathcal{G}(z^2)=(X_2^*X_1-z^2)^{-1}$.
 \begin{pro}\label{prop_upper_bound}
 	With high probability, $\mathcal{X}_0-z$ is invertible uniformly in  $|z|\geq\sqrt{\rho(\mathcal{V})}+\delta'$ for some $\delta'>0$, and in addition
 	\begin{equation}
 		||(\mathcal{X}_0-z)^{-1}||_{op}< \frac{1}{\delta}, \qquad  ||G(z^2)||_{op}<\frac{1}{\delta}, \qquad  ||\mathcal{G}(z^2)||_{op}< \frac{1}{\delta} \label{042301}
 	\end{equation}
 	hold for some constant $\delta>0$, uniformly in $|z|\geq\sqrt{\rho(\mathcal{V})}+\delta'$. 
 \end{pro}

 Here, we first give the proof of Proposition \ref{prop_upper_bound}, assuming (\ref{012301}) is satisfied. Since the system of equations ($\ref{MDE1}$) is scaling invariant, in the following we may assume $\rho (\mathcal{V})=1$.
 \begin{proof}[Proof of Proposition \ref{prop_upper_bound}]
Since one can find $\delta'>0$ such that (\ref{012301}) holds, 	
then $z\notin \widetilde{\mathbb{D}}_{\delta'}$. This implies, by a slight abuse of notations,
\begin{equation*}
	\text{dist}(0,\text{supp } \rho^z)>\delta'.
\end{equation*}
Further, Theorem 4.7 in \cite{AEKN18} tells us that uniformly in $z$, no eigenvalue of $H^z$	is away from $\text{supp }\rho^z$ with high probability. Specifically, for a fixed $z$, there exist constants $\delta_0, C>0$ such that
\begin{equation*}
	\mathbb{P}\left( \text{Spec}(\mathbf{H}^z)\subseteq\{{x} \in \mathbb{R}: \text{dist}({x}, \text{supp }\rho^z)\leq N^{-\delta_0}\}\right)\geq 1-\frac{C}{N}.
\end{equation*}
Hence, there exists $\delta>0$ such that the smallest singular value of $\mathcal{X}_0-z$ satisfies
\begin{equation*}
	\mathbb{P} \left( \sigma_{min}(\mathcal{X}_0-z)>\delta \right)\geq 1-\frac{C}{N}.
\end{equation*}	
Then with high probability, $\mathcal{X}_0-z$ is invertible and
\begin{equation}\label{op_bound}
	||(\mathcal{X}_0-z)^{-1}||_{op}< \frac{1}{\delta}.
\end{equation}
According to Schur's complement 
\begin{equation*}
	(\mathcal{X}_0-z)^{-1}=\left(
\begin{array}{ccc}
(-z+\frac{1}{z}X_1X_2^*)^{-1} & ~\\
~&(-z+\frac{1}{z}X_2^*X_1)^{-1}
\end{array}
\right).
\end{equation*}
$zG(z^2)$ is a submatrix of $\mathcal{X}_0-z$, then by ($\ref{op_bound}$), for $|z|^2>1$
\begin{equation*}
	||G(z^2)||_{op}<\frac{1}{\delta}.
\end{equation*}

The uniformity of the estimates follow simply by applying Neumann expansion. For instance, for sufficiently large $z$, say, $|z|\geq N^C$ for some large constant $C$, one can expand $G(z)$ around $-z^{-1}$, by applying a crude order $1$ upper bound of the operator norm $X_1X_2^*$. For those $\sqrt{\rho(\mathcal{V})}+\delta'\leq |z|\leq N^C$, we can apply a standard $\epsilon$-net argument. We can find an $\epsilon$-net of this domain with cardinality $N^{O(C)}$ such that for each $z$ in this domain, one can find an $z'$ in the $\epsilon$-net, so that $|z-z'|\leq N^{-C}$. By the definition of stochastic dominance in Definition \ref{stochastic dom}, one can readily conclude that the estimates in (\ref{042301}) hold uniformly on the $\epsilon$-net, with high probability. Then for any other $z$ satisfying $\sqrt{\rho(\mathcal{V})}+\delta'\leq |z|\leq N^C$, we can expand $G(z)$ around $G(z')$ using Neumann expansion, where $z'$ is a point in the $\epsilon$-net and $|z-z'|\leq N^{-2C}$. Then by the boundedness of the operator norm $G(z')$, one can conclude the proof of the uniformity for all $|z|\geq  \sqrt{\rho(\mathcal{V})}+\delta'$. 
 
 \end{proof}

 In the sequel, we prove (\ref{012301}), which follows from the lemma below, according to (\ref{MDE_Solution}). 
\begin{lem}\label{lem_uv_est}
	The solution of (\ref{MDE1}) satisfies
	\begin{equation}\label{Sol_Aver_Equ}
		\langle u(\eta)\rangle=\langle v(\eta)\rangle
	\end{equation}
	for all $\eta>0$, where $\langle a(\eta)\rangle =(n+p)^{-1}\sum_{i=1}^{n+p}a_i(\eta)$ for $a=u,v$. Uniformly in $0<\eta\leq 1$ and $|z|^2>1+\delta'$, we have the estimate
	\begin{equation}\label{Sol_uv_est}
		u(\eta)\sim v(\eta) \sim \frac{\eta}{|z|^2-1+\eta^{2/3}}.
	\end{equation}
\end{lem}
Here we follow the proof strategy of Proposition 3.2 in \cite{AEK18}. The main difference is that our variance matrix $\mathcal{V}$  has zero blocks and the entries of $\mathcal{V}$ do not satisfy the flatness assumption, although the off-diagonal blocks do satisfy the flatness assumption; see (\ref{flatness_assumption}). Necessary modifications will be made in the following proof. 
\begin{proof}[Proof of Lemma \ref{lem_uv_est}]
First,  multiplying on both sides of two equations in ($\ref{MDE1}$) by $\eta+\mathcal{V}^*u$ and $\eta+\mathcal{V}v$ respectively, we get 
	\begin{equation}
		\frac{u}{\eta +\mathcal{V}^*u}=\frac{v}{\eta +\mathcal{V}v}
	\end{equation}
	which leads to 
	\begin{equation}
		0=\eta(u-v)+(u\mathcal{V}v-v\mathcal{V}^*u)
	\end{equation}
	Taking average on both sides and using the fact that $\langle u\mathcal{V}v \rangle=\langle v\mathcal{V}^*u \rangle$, we obtain ($\ref{Sol_Aver_Equ}$).
	
	The matrix $\mathcal{V}$ does not satisfy the flatness assumption, namely we cannot deduce the estimate immediately from assumption (\ref{flatness_assumption}) on $S$ that
	\begin{equation}\label{Sol_est}
		\mathcal{V}v\sim \langle v(\eta)\rangle, \mathcal{V}u\sim \langle u(\eta)\rangle
	\end{equation}
	 But we can still make use of the block structure of $\mathcal{V}$ to prove (\ref{Sol_est}). Note that by the definition in (\ref{variance profile}) and
	\begin{equation}\label{Sol_vec_form}
		\mathcal{V}v=\binom{\mathbb{T}v_2}{\mathbb{T}^*v_1},
	\end{equation}
	 we immediately get from (\ref{flatness_assumption}) the estimate 
	\begin{equation}\label{Sol_vec_est}
		\langle v_1\rangle \sim \mathbb{T}^*v_1, \langle v_2\rangle \sim \mathbb{T}v_2, \qquad 
		\langle u_1\rangle \sim \mathbb{T}^*u_1, \langle u_2\rangle \sim \mathbb{T}u_2
	\end{equation}
	Now together with ($\ref{Sol_vec_form}$), ($\ref{Sol_vec_est}$) and the fact that 
	\begin{align}\label{Sol_combin}
		\langle u \rangle=\frac{p}{n+p}\langle u_1 \rangle+\frac{n}{n+p}\langle u_2 \rangle, \qquad 
		\langle v \rangle=\frac{p}{n+p}\langle v_1 \rangle+\frac{n}{n+p}\langle v_2 \rangle,
	\end{align}
	in order to prove ($\ref{Sol_est}$), it suffices to show 
	\begin{equation}\label{Sol_u_v_sim}
		\langle u_1 \rangle \sim \langle u_2 \rangle \text{, } \langle v_1 \rangle \sim \langle v_2 \rangle
	\end{equation}
	We first need to show an auxiliary bound for $\langle u \rangle$ and $\langle v \rangle$
	\begin{equation}\label{Sol_auxil_bound}
		\eta \lesssim \langle u \rangle=\langle v \rangle \lesssim 1
	\end{equation}
	From the first equation in ($\ref{MDE2}$) and ($\ref{Sol_vec_est}$) we have
	\begin{equation}\label{Sol_u_1}
		u_1=\frac{\eta+Tu_2}{(\eta+Tu_2)(\eta+Tv_2)+|z|^2}\sim \frac{\eta+\langle u_2 \rangle}{(\eta+\langle u_2 \rangle)(\eta+\langle v_2 \rangle)+|z|^2}
	\end{equation}
	Suppose $\langle u \rangle=\langle v \rangle \lesssim \eta$, then $\langle u_2 \rangle,\langle v_2 \rangle \lesssim \eta$ follows immediately from ($\ref{Sol_combin}$). Then ($\ref{Sol_u_1}$) gives $\langle u_1 \rangle \sim \eta$. The lower bound follows.
	
	To show the upper bound for $\langle u \rangle$ and $\langle v \rangle$, from ($\ref{MDE2}$) one can obtain
	\begin{align*}
		\frac{u_1}{\eta+\mathbb{T}u_2}=\frac{v_1}{\eta+\mathbb{T}v_2}, \qquad 
		\frac{u_2}{\eta+\mathbb{T}^*u_1}=\frac{v_2}{\eta+\mathbb{T}^*v_1}. 
	\end{align*}
Multiplying $\mathbb{T}$, $\mathbb{T}^*$ on both sides of the above two equations respectively,  and using ($\ref{Sol_vec_est}$) gives
\begin{align}
	\frac{\langle u_1 \rangle}{\eta+\langle u_2\rangle}\sim \frac{\langle v_1\rangle}{\eta+\langle v_2\rangle}, \qquad
		\frac{\langle u_2\rangle}{\eta+\langle u_1\rangle}\sim \frac{\langle v_2\rangle}{\eta+\langle v_1\rangle}. \label{022001}
\end{align}
Using the lower bound on $\langle u \rangle$ and ($\ref{Sol_combin}$), we may assume $\langle u_2 \rangle\gtrsim \eta$. In case $\langle v_2\rangle \gtrsim \eta$ also holds,  the first estimation becomes
\begin{equation}
	\frac{\langle u_1 \rangle}{\langle u_2\rangle}\sim \frac{\langle v_1\rangle}{\langle v_2\rangle} \label{012302}
\end{equation}
If $\langle v_2\rangle \lesssim \eta$, we can easily get from (\ref{Sol_u_1}) that $\langle u_1\rangle\gtrsim \eta$. This together with the second equation in (\ref{022001}) that (\ref{012302}) is still valid. 

It then follows from (\ref{Sol_Aver_Equ}), (\ref{Sol_combin}) and (\ref{012302})  that
\begin{equation}\label{Sol_u1_v1_sim}
	\langle u_1 \rangle \sim \langle v_1 \rangle, \langle u_2 \rangle \sim \langle v_2 \rangle. 
\end{equation}
Now, from the first equation in ($\ref{MDE2}$),  we know 
	\begin{equation}
		1=\eta u_1+u_1\mathbb{T}v_2+\frac{|z|^2u_1}{\eta+\mathbb{T}u_2}\geq u_1\mathbb{T}v_2
	\end{equation}
	Taking average gives 
	\begin{equation}
		1\geq \langle u_1\mathbb{T}v_2 \rangle \gtrsim \langle u_1 \rangle \langle v_2 \rangle
	\end{equation}
	Similarly, using the second equation in ($\ref{MDE2}$), $1 \gtrsim \langle u_2 \rangle \langle v_1 \rangle$. If $\langle u_1 \rangle, \langle v_2 \rangle\lesssim 1$, together with ($\ref{Sol_combin}$), ($\ref{Sol_u1_v1_sim}$), the upper bound follows. Otherwise, suppose $\langle u_2 \rangle$ is of order greater than 1, then $\langle u_1 \rangle \sim \langle v_1 \rangle$ is of order smaller than 1. However, the second equation in ($\ref{MDE2}$) gives $\langle u_2 \rangle\lesssim \eta +\langle u_1 \rangle$, which leads to a contradiction. Hence,  ($\ref{Sol_auxil_bound}$) is proved.
	
	Using ($\ref{Sol_auxil_bound}$), ($\ref{Sol_u_1}$) and ($\ref{Sol_u1_v1_sim}$) and supposing $\langle u_2 \rangle \gtrsim \eta $, we obtain 
	\begin{equation}
		u_1\sim \frac{1}{\eta +\langle u_2 \rangle+\frac{|z|^2}{\eta +\langle u_2 \rangle}}\sim \langle u_2 \rangle.
	\end{equation}
	Hence, ($\ref{Sol_u_v_sim}$) follows by taking average. Consequently, we have (\ref{Sol_est}).
	
	With ($\ref{Sol_est}$),  the remaining proof can be done  similarly to  that in \cite{AEK18}. Using the first equation of ($\ref{MDE1}$) gives
	\begin{equation}\label{Sol_u_order1}
		\eta = u(\eta +\mathcal{V}v)(\eta+\mathcal{V}^*u)+|z|^2u-\mathcal{V}^*u.
	\end{equation}
	By the Perron-Frobenius theorem, there exists a vector $\varrho \in \mathbb{R}^{n+p}_+$  such that
	\begin{equation}
		\mathcal{V}\varrho =\varrho , \langle \varrho  \rangle=1, \varrho \sim 1
	\end{equation} 
	Note that $\mathcal{V}$ has zero blocks and indeed we apply the Perron-Frobenius theorem to $\mathbb{T}\mathbb{T}^*$ and $\mathbb{T}^*\mathbb{T}$ to get $\varrho \in \mathbb{R}_+^{n+p}$.
	Taking scalar product of ($\ref{Sol_u_order1}$) with $\varrho $, we get
	\begin{equation}
		\eta=\langle \varrho  u(\eta +\mathcal{V}v)(\eta+\mathcal{V}^*u)\rangle+(|z|^2-1)\langle \varrho  u \rangle \sim \langle u \rangle^3+(|z|^2-1)\langle u \rangle
	\end{equation}
	where in the last step we also use ($\ref{Sol_Aver_Equ}$) and ($\ref{Sol_est}$). Now one can conclude ($\ref{Sol_uv_est}$) for $u(\eta)$, and similarly for $v(\eta)$. It further implies (\ref{012301}). 
	
	Then we can conclude the proof of Proposition \ref{prop_upper_bound}. 

\end{proof}

Next, we extend the conclusions in Proposition \ref{prop_upper_bound} from the matrix $\mathcal{X}_0$ to $\mathcal{X}$. We have the following proposition. 

\begin{pro}\label{prop_upper_bound_spike}  With high probability, $\mathcal{X}-z$ is invertible uniformly in all $|z|\geq\sqrt{\rho(\mathcal{V})}+\delta'$ for some $\delta'>0$.  
\end{pro}

\begin{proof} We first show the proof for the invertibility for one fixed $z$ which is larger than $\sqrt{\rho(\mathcal{V})}+\delta'$ in magnitude. Later, we will prove the uniformity. 
From Proposition \ref{prop_upper_bound}, we  know 
\begin{equation}\label{02151403}
	\det(\mathcal{X}_0-z)\not =0
\end{equation}
 with high probability. It suffices to show that  $\det(\mathcal{X}-z)\not =0$.  Recall the definition from (\ref{model}) 
\begin{equation*}
	\mathcal{X}=\left(
\begin{array}{ccc}
~ & \Sigma X_1\\
(\Sigma X_2)^* &~
\end{array}
\right)
\end{equation*}
We use the following basic identity of  the determinant of a block matrix
\begin{equation}
	\det \left(
\begin{array}{ccc}
\mathcal{A} & \mathcal{B}\\
\mathcal{C} & \mathcal{D}
\end{array}
\right)=\det (\mathcal{A}-\mathcal{B}\mathcal{D}^{-1}\mathcal{C})\cdot \det(\mathcal{D})
\end{equation}
which holds when $\mathcal{D}$ is invertible. One can then easily check that 
\begin{align}\label{02151445}
	\det (\mathcal{X}-z)&=(-1)^p\det(\Sigma X_1 X_2^*\Sigma^*-z^2)\nonumber\\
	&=(-1)^p\det(\Sigma\Sigma^*)\cdot\det(X_1X_2^*-z^2(\Sigma^*\Sigma)^{-1})
\end{align}
Recall from (\ref{032401}) that $\Gamma$ is a low rank diagonal matrix which  satisfies $\det(I-\Gamma)=1/\det(\Sigma^*\Sigma)$. Since (\ref{02151403}) implies that $\det(X_1X_2^*-z^2)\not =0$, we have
\begin{align}\label{02151446}
	\det(X_1X_2^*-z^2(\Sigma^*\Sigma)^{-1})&=\det(X_1X_2^*-z^2+z^2L\Gamma L^*) \nonumber\\
	&=\det(X_1X_2^*-z^2)\cdot\det(I+z^2\Gamma L^*(X_1X_2^*-z^2)^{-1}L)
\end{align}
By our assumption, $1-\Gamma_{ii}$ is no smaller than $\lambda^{-1}_{max}(\Sigma^*\Sigma)$. Further using the fact that $\lambda_{max}(\Sigma^*\Sigma)\cdot q_n^{-1}=o(1)$, we have that $1-\Gamma_{ii}$ is of order greater than $q_n^{-1}$ for all $i$. Moreover, using the fact that $\Gamma$ is of low rank and diagonal, we have for some fixed constant $k$ 
\begin{align}\label{02151447}
	\det(I+z^2\Gamma L^*(X_1X_2^*-z^2)^{-1}L)&=(1+o(1))\cdot\prod_{i=1}^{k}(1+z^2\Gamma_{ii}c_i^*(X_1X_2-z^2)^{-1}c_i)\nonumber\\
	&=(1+o(1))\cdot\prod_{i=1}^{k}(1+z^2\Gamma_{ii}(-\frac{1}{z^2}+O_{\prec} (q_n^{-1})))\nonumber\\
	&=(1+o(1))\cdot\prod_{i=1}^{k}(1-\Gamma_{ii}+O_{\prec}(q_n^{-1}))
\end{align}
where the first two steps follow from Proposition \ref{quadraticformbound}, and especially in the first step we used  the fact that all the off-diagonal entries are of order $q_n^{-1}\ll 1-\Gamma_{ii}$. Now together with (\ref{02151445}), (\ref{02151446}) and (\ref{02151447}),  $\det(\mathcal{X}-z)\not =0$ holds with high probability, for a fixed $z$. 
Next, we show that the fact $\det(\mathcal{X}-z)\not =0$ holds uniformly in $|z|\geq \sqrt{\rho(\mathcal{V})}+\delta'$, with high probability. It suffices to have  Proposition \ref{quadraticformbound} uniformly in $z$.   Again,  the uniformity of the estimates in Proposition \ref{quadraticformbound} follows simply by applying Neumann expansion. For instance, for sufficiently large $z$, say, $|z|\geq N^C$ for some large constant $C$, one can expand $G(z)$ around $-z^{-1}$, and apply a crude upper bound of the operator norm $X_1X_2^*$, to see  that Proposition \ref{quadraticformbound} is true uniformly in $|z|\geq N^C$, with high probability. For those $\sqrt{\rho(\mathcal{V})}+\delta'\leq |z|\leq N^C$, we can apply a standard $\epsilon$-net argument. We can find an $\epsilon$-net of this domain with cardinality $N^{O(C)}$ such that for each $z$ in this domain, one can find an $z'$ in the $\epsilon$-net, so that $|z-z'|\leq N^{-C}$. By the definition of stochastic dominance in Definition \ref{stochastic dom}, one can readily conclude from Proposition \ref{quadraticformbound} that the estimates therein holds uniformly on the $\epsilon$-net, with high probability. Then for any other $z$ satisfying $\sqrt{\rho(\mathcal{V})}+\delta'\leq |z|\leq N^C$, we can expand $G(z)$ around $G(z')$ using Neumann expansion, where $z'$ is a point in the $\epsilon$-net and $|z-z'|\leq N^{-2C}$. Then by the boundedness of the operator norm $G(z')$ from Proposition \ref{prop_upper_bound}, one can conclude the proof of the uniformity of the estimate in Proposition \ref{quadraticformbound} for all $|z|\geq  \sqrt{\rho(\mathcal{V})}+\delta'$, which holds with high probability. 
\end{proof}

\begin{proof}[Proof of Proposition \ref{pro.spectralradius}] With Propositions \ref{prop_upper_bound} and \ref{prop_upper_bound_spike}, we can conclude the proof of Proposition \ref{pro.spectralradius}.
\end{proof}

\section{A Priori bound for Green function quadratic forms: Proof of Proposition \ref{quadraticformbound}}\label{s.proof of quadratic form bound}
\begin{proof}[Proof of Proposition \ref{quadraticformbound}] For brevity, in the sequel, we show the proof of the first  bound in (\ref{011510}). The other terms can be bounded in a similar manner and thus we omit the details. 
Recall the definition of $q_n$ from (\ref{def of q}). We denote by 
\begin{align*}
\mathcal{P}:=zq_n u^* \underline{G}v \chi_G=q_n u^* X_1X_2^* Gv \chi_G,
\end{align*}
where $\chi_G$ is defined in (\ref{011502}). We also refer to the discussion in (\ref{011502})-(\ref{010722}) regarding the $\chi_G$ factor. Therefore, in the sequel, we may simply regard $\chi_G$ as $1$ and neglect all the terms involving its derivatives. 

Notice that from the definition in (\ref{def of G under}) we have the trivial bound 
\begin{align}
|\mathcal{P}|=|q_n u^*v+zq_n u^* Gv|\chi_G\prec q_n \label{122001}
\end{align}
and also the deterministic bound $|\mathcal{P}|\leq CKn$ for some constant $C>0$; see (\ref{101501}).

We aim for a recursive moment estimate for $\mathbb{E} (\mathcal{P})^{2k}$.  By cumulant expansion formula \cite[Lemma 1.3]{HK17}, we have 
\begin{align}
\mathbb{E} (\mathcal{P})^{2k} &=q_n\sum_{ij} \mathbb{E}\Big[x_{1, ij} (X_2^*  Gvu^*)_{ji} \chi_G(\mathcal{P})^{2k-1} \Big]\notag\\
&=q_n \sum_{\alpha=1}^m \sum_{ij} \frac{\kappa_{\alpha+1}(x_{1,ij})}{\alpha!} \mathbb{E} \partial_{1,ij}^{\alpha} \Big[(X_2^*  Gvu^* )_{ji}\chi_G (\mathcal{P})^{2k-1} \Big]+\mathcal{R}, \label{103101}
\end{align}
where we again use $\mathcal{R}$ to denote the remainder term, with certain abuse of notation. Specifically,  here $\mathcal{R}$ is given by 
\begin{align}
|\mathcal{R}|&\leq C\mathbb{E} |x_{1,ij}|^{m+2} \mathbb{E} \Big(\sup_{|x_{1,ij}|\leq c}\Big| \partial^{m+1}_{1,ij} \Big[ (X_2^*  Gvu^* )_{ji} \chi_G(\mathcal{P})^{2k-1} \Big]\Big|\Big)\notag\\
&+C\mathbb{E} \Big(|x_{1,ij}|^{m+2}\mathbf{1}(|x_{1,ij}|>c) \Big)\mathbb{E} \Big(\sup_{x_{1,ij}\in \mathbb{R}}\Big| \partial^{m+1}_{1,ij} \Big[ (X_2^*  Gvu^*)_{ji} \chi_G(\mathcal{P})^{2k-1} \Big]\Big|\Big). \label{042502}
\end{align}
As the derivative can only generate matrix entries which are $O_\prec(1)$, and there are totally $2k-1$ $q_n$-factors from $\mathcal{P}^{2k-1}$, we can trivially bound  
\begin{align*}
\Big|\partial^{m+1}_{1,ij} \Big[ (X_2^*  Gvu^* )_{ji} \chi_G (\mathcal{P}^{2k-1})\Big|\prec q_n^{ 2k-1}.
\end{align*}
In order to bound the remainder terms in the cumulant expansion, we also need to have the above bound when we take supremum over one $X$ entry. This technical problem can be handled by simply truncating the matrix entries at $n^{-1/2+\varepsilon}$ at the beginning. We omit the details.   
Further notice that $\mathbb{E}|x_{1,ij}|^{m+2}\sim n^{-\frac{m+2}{2}}.$ 
Hence, it is easy to show that when $m$ is sufficiently large, 
\begin{align}
|\mathcal{R}|\prec n^{-C}  \label{remainderterm}
\end{align}
for a constant $C=C(m, \varepsilon)>0$.  From the first $m$ terms in (\ref{103101}), we start from $\alpha=1$. After an elementary calculation, we arrive at 
\begin{align*}
&q_n  \sum_{ij} \kappa_{2}(x_{1,ij}) \mathbb{E} \partial_{1,ij}^{1} \Big[(X_2^*  Gvu^* )_{ji} \chi_G(\mathcal{P})^{2k-1} \Big]\notag\\
&=\frac{q_n}{n} \mathbb{E}\Big[ v^*G^* X_2X_2^*Gu \chi_G \cdot \mathcal{P}^{2k-1} \Big]-z\frac{q_n^2}{n} \mathbb{E}\Big[v^*G^* X_2X_2^* G v\cdot u^* G^* u \chi_G\cdot \mathcal{P}^{2k-2} \Big]\notag\\
&=\frac{q_n}{n} \mathbb{E}\Big[ O_\prec(1) \mathcal{P}^{2k-1} \Big]+\frac{q_n^2}{n} \mathbb{E}\Big[O_\prec(1) \mathcal{P}^{2k-2} \Big]
\end{align*}
For $\alpha\geq 0$, we first notice that each term in $\partial_{1,ij}^{\alpha} (X_2^*  Gvu^* )_{ji}$ must contain the factor 
\begin{align*}
(X_2^*  Gvu^* )_{ji}
\end{align*}
and the other factors can be simply bounded by $O_\prec(1)$. Further, we notice that each term in $\partial_{1,ij}^{\alpha} \mathcal{P}$ must contain  the factor
\begin{align*}
u^* G e_i\cdot e_j^* X_2^* G v
\end{align*}
and the other factors can be simply bounded by $O_\prec(1)$. Apart from the $q_n$ factors, the $ij$ sum of all the derivatives can be simply bounded by 
\begin{align*}
&\sum_{ij} |e_j^*(X_2^*  Gvu^* )e_i| |e_j^* X_2^*  G v| |u^* Ge_i | \\
&\leq \sum_{i}  |u^* Ge_i | \sqrt{\sum_j |e_j^*(X_2^*  Gvu^* )e_i|^2}\sqrt{\sum_j|e_j^* X_2^*  G v|^2}\notag \\
&\leq \sum_i |u^* Ge_i||u^* e_i| \sqrt{v^* G^* X_2X_2^* G v} \sqrt{v^* G^* X_2X_2^* G v} \\
&\prec  \sum_i |u^* Ge_i||u^*  e_i|\prec 1,
\end{align*}
where the last step follows from Cauchy Schwarz inequality. Then, if we further consider the $q_n$ factors generated from the derivative of $\mathcal{P}$, the most dangerous terms from $\partial_{1,ij}^{\alpha} \mathcal{P}^{2k-1}$ would be 
\begin{align*}
(\partial_{1,ij}^1 \mathcal{P})^{\alpha} \mathcal{P}^{2k-1-\alpha}
\end{align*}
which creates a factor $q_n^{\alpha}$. The contribution of such term to (\ref{103101}) would be 
\begin{align*}
\Big(\frac{q_n}{\sqrt{n}}\Big)^{\alpha+1} \mathcal{P}^{2k-1-\alpha}.
\end{align*}
Putting all these estimates together, we arrive at the recursive moment estimate
\begin{align*}
\mathbb{E} |\mathcal{P}|^{2k}= \sum_{\alpha=1}^{2k} \Big(\frac{q_n}{\sqrt{n}}\Big)^{\alpha} \mathbb{E}\Big[O_\prec (1) \mathcal{P}^{2k-\alpha}\Big]+O(n^{-C}). 
\end{align*}
Applying a Young's inequality, we can immediately get 
\begin{align*}
\mathbb{E} |\mathcal{P}|^{2k}=o(1). 
\end{align*}
Due to arbitrariness of $k$, we obtain the first estimate in (\ref{011510}). The other terms can be estimated similarly. Hence, we conclude the proof of Proposition \ref{quadraticformbound}.
\end{proof}

\section{Proofs of Lemmas \ref{lem.fluctuation} and \ref{lem.expectation}} \label{app.lemmas}

\begin{proof}[Proof of Lemma \ref{lem.fluctuation}]
Similarly to the discussion in Section \ref{section. proof of CLT}, for brevity, we state the proof for a single $L_k$ in the sequel. The joint distribution of all $L_i$'s can be done similarly. 
First we write
\begin{align*}
	g(t) = \exp( \mathrm{i} t L_{k} ) 
	= \exp\Big( -\frac{t}{\pi} \oint_{\Gamma} \omega^{k+1} (\text{Tr}\,\underbar{G}(\omega^{2})-\E[\text{Tr}\,\underbar{G}(\omega^{2})]) \mathrm{d}\omega \Big).
\end{align*}
Consider the derivative
\begin{align}
	g'(t) = -\frac{1}{\pi} \oint_{\Gamma} \omega^{k+1} (\text{Tr}\,\underbar{G}(\omega^{2})-\E[\text{Tr}\,\underbar{G}(\omega^{2})]) f(t) \mathrm{d}\omega.  \label{071016}
	\end{align}
We aim to show
\begin{equation*}
	\E[g'(t)] = -t\sigma^{2}\E[g(t)] + o(1), \quad \text{for some $\sigma>0$.}
\end{equation*}
We will need the following lemma as a technical input. 
\begin{lem}\label{l.2 point function} Let $|w_1|, |w_2|\geq  (p/n)^{1/2} + \delta/2$. We have 
\begin{align*}
	&\mathbb{E} \Tr\big((G(w_1)^\top G(w_2)\big)= p\left(w_1w_2 - \frac{p^2}{n^2}\right)^{-1} + O_{\prec}(nq_n^{-1}),\notag\\
	&\mathbb{E}\big(\Tr(G(w_1)^\top G(w_2)) f(t)\big)=p\left(w_1w_2 - \frac{p^2}{n^2}\right)^{-1}\mathbb{E} f(t)+ O_{\prec}(nq_n^{-1}). 
\end{align*}
\end{lem}
 With the above lemma, we proceed with the estimate of (\ref{071016}). 
First, applying the cumulant expansion w.r.t. $X_1$ variables, we can get 
\begin{align}
	\mathbb{E}\Tr \underline{G}(w^2) &= \frac{1}{w^2}\mathbb{E}\Tr X_1X_2^*G(w^2) \notag \\
	&= \frac{1}{nw^2}\sum_{a,b}\mathbb{E} \partial_{1,ab} (X_2^*G)_{ba} + \frac{1}{w^2}\sum_{a, b}\sum_{\alpha=2}^m\frac{\kappa_{\alpha+1}(x_{1, ab})}{\alpha!}\mathbb{E}\partial_{1, ab}^{\alpha} (X_2^*G)_{ba}+\mathcal{R} \notag \\
	&= -\frac{1}{nw^2}\sum_{a,b}\mathbb{E} (X_2^*G)_{ba}(X_2^*G)_{ba} + O_{\prec}(\sqrt{n}q_n^{-3}) \label{070710}
\end{align}
Further applying the cumulant expansion w.r.t. $X_2$ entries, and 
applying Lemma \ref{l.2 point function} with $w_1=w_2$, we will get 
\begin{align}
		\mathbb{E}\Tr \underline{G}(w^2) &= -\frac{p}{n^2w^2}\sum_{a,k}\mathbb{E} (G_{ka}G_{ka}) + O_{\prec}(\sqrt{n}q_n^{-3})= -\frac{p^2}{n^2w^2}\left(w^4 - \frac{p^2}{n^2}\right)^{-1} + O_{\prec}(q_n^{-1}).  \label{070711}
\end{align}
Here we omitted the estimate of the error terms as it is similar to the discussion in Section \ref{section. proof of CLT}. 

Similarly, we have 
\begin{align}
	&\mathbb{E} \big( \Tr \underline{G}(w^2) f(t) \big) = \frac{1}{w^2} \mathbb{E} \big(\Tr(X_1X_2^*G(w^2))f(t) \big) \notag \\
	&= \frac{1}{nw^2} \sum_{a,b}\mathbb{E} \big(\partial_{1, ab} (X_2^*G)_{ba}f(t) \big) + \frac{1}{w^2}\sum_{a, b}\sum_{\alpha=2}^m\frac{\kappa_{\alpha+1}(x_{1, ab})}{\alpha!}\mathbb{E}\big(\partial_{1, ab}^{\alpha} (X_2^*G)_{ba}f(t) \big)+\mathcal{R} \notag \\
	&= -\frac{1}{nw^2}\sum_{a,b}\mathbb{E} \big((X_2^*G)_{ba}(X_2^*G)_{ba}f(t) \big) \notag \\
	& \qquad - \frac{t}{2n\pi w^2} \sum_{a,b}\mathbb{E}\big( (X_2^*G_1)_{ba} f(t)\oint_{\Gamma_2} w_2^{k+1}\partial_{1, ab} (1-\mathbb{E})\Tr(\underline{G_2}) \mathrm{d}w_2 \big) + O_{\prec}(\sqrt{n}q_n^{-3}). \label{071013}
\end{align}
For the first term in the RHS above, similarly to (\ref{070710}) and (\ref{070711}), we can get via performing cumulant expansion w.r.t. $X_2$ variables
\begin{align}
-\frac{1}{nw^2}\sum_{a,b}\mathbb{E} \big((X_2^*G)_{ba}(X_2^*G)_{ba}f(t) \big)=  -\frac{p^2}{n^2w^2}\left(w^4 - \frac{p^2}{n^2}\right)^{-1}\mathbb{E} f(t) + O_{\prec}(\sqrt{n}q_n^{-3}). \label{071014}
\end{align}

Plugging (\ref{070711}), (\ref{071013}) and (\ref{071014}) into (\ref{071016}), we obtain 
\begin{align*}
	&\mathbb{E} g'(t) = \frac{t}{4n\pi^2}\oint_{\Gamma_1}\oint_{\Gamma_2}w_1^{k-1}w_2^{k-1} \sum_{a,b} \mathbb{E}\big( (X_2^*G_1)_{ba} f(t) \partial_{1, ab}\Tr(X_1X_2^*G_2) \big)\mathrm{d}w_2 + O_{\prec}(\sqrt{n}q_n^{-3}) \notag \\
	&= \frac{t}{4n\pi^2}\oint_{\Gamma_1}\oint_{\Gamma_2}w_1^{k-1}w_2^{k-1}\sum_{a,b}\mathbb{E}\Big((X_2^*G_1)_{ba} \big((X_2^*G_2)_{ba} - w_2^2(X_2^*G_2\underline{G_2})_{ba}\big)f(t)    \Big) \mathrm{d}w_2 + O_{\prec}(\sqrt{n}q_n^{-3}) 
\end{align*}
By a further cumulant expansion w.r.t. $X_2$-variables, we can arrive at
\begin{align*}
	&\sum_{a,b}\mathbb{E}\Big((X_2^*G_1)_{ba}\big((X_2^*G_2)_{ba} - w_2^2(X_2^*G_2\underline{G_2})_{ba}\big) f(t) \Big) \notag \\
	&\qquad = \frac{p}{n} \mathbb{E}\big(\Tr(G_1^\top G_2) f(t)\big) + O_{\prec}(nq_n^{-1})
\end{align*}
Again, by Lemma \ref{l.2 point function}, we have 
\begin{align*}
	\mathbb{E} g'(t) &= \frac{pt}{4n^2\pi^2}\oint_{\Gamma_1}\oint_{\Gamma_2}w_1^{k-1}w_2^{k-1}\mathbb{E}\Big(\Tr(G_1^\top G_2) f(t)\Big)\mathrm{d}w_1\mathrm{d}w_2 + O_{\prec}(q_n^{-1}) \notag \\
	&= \frac{p^2t}{4n^2\pi^2}\oint_{\Gamma_1}\oint_{\Gamma_2}w_1^{k-1}w_2^{k-1}\left(w_1^2w_2^2 - \frac{p^2}{n^2}\right)^{-1}\mathrm{d}w_1\mathrm{d}w_2 \mathbb{E}f(t) + O_{\prec}(q_n^{-1}).
\end{align*}
Since we are in the regime $|w_1^2|, |w_2^2| >  p/n$, we can write  
\begin{align*}
	\mathbb{E} g'(t) = \frac{t}{4\pi^2}\sum_{\ell =1}^{\infty} \left(\frac{p}{n}\right)^{2\ell}\oint_{\Gamma_1}\oint_{\Gamma_2}w_1^{k-1-2\ell}w_2^{k-1-2\ell}
	  \mathrm{d}w_1\mathrm{d}w_2 \mathbb{E}f(t) + O_{\prec}(q_n^{-1}).
\end{align*}
Now, consider 
\begin{align*}
	\oint_{\Gamma_1}\oint_{\Gamma_2}w_1^{k-1-2\ell}w_1^{k-1-2\ell}
	  \mathrm{d}w_1\mathrm{d}w_2 = 
	  \begin{cases}
	  	2\pi i \text{Res}(1 , 0) (\oint_{\Gamma_1}w_1^{k-1-2\ell}\mathrm{d}w_1) = -4\pi^2 , &\text{when } 2\ell = k, \\
	  	0, &\text{otherwise }
	  \end{cases}
\end{align*}
Thus, 
\begin{align*}
	\mathbb{E} g'(t) = \begin{cases}
		-t\left(\frac{p}{n}\right)^{k}\mathbb{E}g(t) + O_{\prec}(q_n^{-1}), &\text{if } k \text{ is even}, \\
		O_{\prec}(q_n^{-1}),  &\text{if } k \text{ is odd }.
	\end{cases}
\end{align*}
This implies that the cumulants $\{\kappa_{j}\}$ of $L_{k}$ satisfy
$\kappa_{2}(L_{k})\sim(p/n)^{k/2}$ for even $k$ while
$\kappa_{m}(L_{k})=o(1)$ for every $m\ge3$ (and for all $m\ge2$ when $k$ is odd).
Moreover, mixed cumulants with distinct orders are negligible,
$\kappa_{m_{1},\dots,m_{r}}(L_{k_{1}},\dots,L_{k_{r}})=o(1)$ whenever
$k_{1},\dots,k_{r}$ are not all equal and even.
Hence the method of cumulants (or characteristic functions) yields that
\eqref{eq: 981} holds for a sequence $\{Z_{k}\}$ of independent random variables
as in Proposition \ref{prop: 622}.

\end{proof}

Next, we prove Lemma \ref{l.2 point function}. 
\begin{proof}[Proof of Lemma \ref{l.2 point function}] For brevity, we denote by $G(w_a)=G_a$ for $a=1,2$ in this proof.  First, we rewrite
\begin{align*}
	\mathbb{E} \Tr(G_1^\top G_2) = \sum_{ij}\mathbb{E} \left(G_{1, j i} G_{2, j i}\right)  = \frac{1}{w_{2}} \sum_{ij}\mathbb{E} \big(G_{1, j i}\left(X_{1} X_{2}^{*} G_{2}\right)_{ji}\big)+\frac{p}{w_{1}w_{2}} + O_{\prec}(n q_n^{-1}). 
\end{align*}
By cumulant expansion, we have
\begin{align*}
	& \sum_{ij}\mathbb{E} \big(G_{1, j i}\left(X_{1} X_{2}^{*} G_{2}\right)_{j i}\big)=\mathbb{E} \sum_{ijb} \big(X_{1, jb} \left(X_{2}^{*} G_{2} \right)_{bi}G_{1, j i}\big) \notag \\
	&=\frac{1}{n} \mathbb{E} \sum_{ijb}  \partial_{1, j b}\big(\left(X_{2}^{*} G_{2}\right)_{b i} G_{1, ji}\big) + \sum_{ijb}\sum_{\alpha=2}^m \frac{\kappa_{\alpha+1}(x_{1, jb})}{\alpha!}\mathbb{E} \partial_{1,jb}^{\alpha}\big(\left(X_{2}^{*} G_{2}\right)_{b i} G_{1, ji}\big) + \mathcal{R} \notag \\
&= -\frac{1}{n} \mathbb{E} \sum_{ijb}  \big((X_2^*G_2)_{bj}(X_2^*G_2)_{bi}G_{1, ji} + (X_2^*G_2)_{bi}G_{1, jj}(X_2^*G_1)_{bi}\big) + O_{\prec}(n^{3/2}q_n^{-3}) \notag \\
&= \frac{1}{n w_1}\mathbb{E} \sum_{ib}  \big((X_2^*G_2)_{bi}(X_2^*G_2)_{bi}\big) + \frac{p}{n w_1}\mathbb{E} \sum_{ib} \big((X_2^*G_2)_{bi}(X_2^*G_1)_{bi}\big) + O_{\prec}(n^{3/2}q_n^{-3}),
\end{align*}
where the error terms can be estimated similarly to the discussion in Section \ref{section. proof of CLT}. Next, we can further perform a cumulant expansion w.r.t. $X_2$ entries, and obtain
\begin{align*}
&\sum_{ij}\mathbb{E} \big(G_{1, j i}\left(X_{1} X_{2}^{*} G_{2}\right)_{j i}\big) 
= \frac{1}{n^2 w_1}\mathbb{E} \sum_{ibk}  \big(G_{2, ki}G_{2, ki} - (G_2X_1)_{kb}G_{2, ki}(X_2^*G_2)_{bi} - w_2 \underline{\mathcal{G}}_{2, bb} G_{2, ki} G_{2, ki}\big) \notag \\
& \qquad + \frac{p}{n^2 w_1}\mathbb{E} \sum_{ibk} \big(G_{1, ki}G_{2, ki} - (G_1X_1)_{kb} G_{1,ki}(X_2^*G_2)_{bi} - w_2 \underline{\mathcal{G}}_{2, bb} G_{1, ki}G_{2, ki}\big) + O_{\prec}(n^{3/2}q_n^{-3}) \notag \\
&= \frac{p}{n^2 w_1}\mathbb{E} \sum_{i}\big((G_2^\top G_2)_{ii} - w_2 q_n^{-1} (G_2^\top G_2)_{ii} \big) +  \frac{p^2}{n^2 w_1}\mathbb{E}  \sum_{i}\big((G_1^\top G_2)_{ii} - w_2 q_n^{-1} (G_1^\top G_2)_{ii}\big) + O_{\prec}(n^{3/2}q_n^{-3}) \notag \\
&= \frac{p^2}{n^2 w_1}\mathbb{E}\Tr(G_1^\top G_2) + O_{\prec}(n q_n^{-1}).
\end{align*}
This implies
\begin{align*}
	\mathbb{E} \Tr(G_1^\top G_2) = \frac{p^2}{n^2 w_1w_{2}}\mathbb{E} \Tr(G_1^\top G_2) +\frac{p}{w_{1} w_{2}}  + O_{\prec}(nq_n^{-1}).
\end{align*}

Hence, we concluded the proof of the first equation in Lemma \ref{l.2 point function}. The second statement can be proved similarly. The only additional complication is that when one performs the cumulant expansion, the derivatives will also hit the $f(t)$ factor. It is easy to show that all terms involving the derivatives are negligible, and thus we omitted the details. 

\end{proof}
 
\begin{proof}[Proof of Lemma \ref{lem.expectation}]

Using (\ref{070711}), one can derive that
\begin{align*}
	\frac{1}{2 \pi i} \oint_{\Gamma}w^{k+1}\mathbb{E}\Tr \underline{G}(w^2) \mathrm{d}w &= - \frac{p^2}{2 \pi i n^2} \oint_{\Gamma}w^{k-1} \left(w^4 - \frac{p^2}{n^2}\right)^{-1} \mathrm{d}w + O_{\prec}(q_n^{-1}) \notag \\
	&= - \frac{1}{2 \pi i} \sum_{\ell=1}^{\infty}\left(\frac{p}{n}\right)^{2\ell}\oint_{\Gamma}w^{k-1-4\ell} \mathrm{d}w + O_{\prec}(q_n^{-1}) \notag \\
	&= \begin{cases}
		-\left(\frac{p}{n}\right)^{k/2}, &\text{if } k \text{ is a multiple of }4, \\
		+ O_{\prec}(q_n^{-1}) , &\text{otherwise }.
	\end{cases}
\end{align*}
%
%

\end{proof}

\section{Proof of Theorem \ref{thm: out iid}} \label{sec: 2425}

By singular value decomposition, one may write
\begin{equation}\label{eq: 1805}
	C = \sum_{i=1}^{k} s_{i} u_{i}v_{i}^{*},
\end{equation}
where $\{s_{i}\}_{i=1}^{k}$ is the set of nonzero singular values of $X$, and $\{u_{i}\}_{i=1}^{k},\{v_{i}\}_{i=1}^{k}\subseteq\mathbb{C}^{n}$ are the sets of orthogonal unit vectors.
Note that
\begin{equation*}
	\max_{1\le i\le m}|s_{i}|\le s^{*}.
\end{equation*}
Define the $n\times k$ matrix $Q$ by taking its \(j\)-th column to be \(u_{j}\) for \(1\le j\le k\).
Likewise, define $k\times n$ matrix $R$ so that its \(j\)-th row equals \(s_{j}v_{j}^{*}\).
With these choices we have the factorization
\[
C \;=\; QR.
\]

As recalled from \eqref{eq: 989}, for any $z$ lying outside the spectrum of $X$, the number $z$ is an eigenvalue of $X+C$ if and only if
\begin{equation}\label{eq: 1821}
	\det(1+R(X-z)^{-1}Q) = 0.
\end{equation}

\begin{theorem}[{\cite[Theorem 1.1]{BCG22}}]
	The spectral radius of $X$ converges to $1$ in probability.
\end{theorem}
By this spectral radius result, with probability tending to $1$, the spectrum of $X$ is contained in the disk $\{z\in\mathbb{C}:|z|\le 1+\delta/3\}$.
Following the approach in Section \ref{sec: 997}, we first define two functions $\tilde{f}(z)$ and $\tilde{f}^{(M)}(z)$ (for $M>0$) by setting
\begin{equation*}
	\tilde{f}(z) = \det(1+R(X-z)^{-1}Q) \quad\text{and}\quad
	\tilde{f}^{(M)}(z) = \det(1+R(X^{(M)}-z)^{-1}Q),
\end{equation*}
where
\[
X^{(M)}=\bigl(x_{ij}^{(M)}\bigr),\quad
x_{ij}^{(M)}
    \;=\;
    x_{ij}\,\mathbf 1(|\sqrt n\,x_{ij}|\le M)
    \;-\;
    \mathbb E[x_{ij}\,\mathbf 1(|\sqrt n\,x_{ij}|\le M)].
\]
Because the entries of $X^{(M)}$ are bounded, relying on the argument used to prove \cite[Theorem 1.7]{Tao} of the light-tail regime, one can find that, 
for every point $c_{i}$ ($i=1,\cdots,\tilde{k}$),
there exists a unique zero of $\tilde{f}^{(M)}(z)$ outside the disk
$\{z\in\mathbb{C}: |z|< 1 + \delta/2\}$ such that
this zero is located arbitrarily close to $c_{i}$ if $n$ and $M$ are large enough.
Thus, by Rouch\'{e}'s theorem with the eigenvalue criterion \eqref{eq: 1821},
it is enough to show that, for sufficiently large \(n\) and \(M\),
\begin{equation}\label{eq: 1870}
\sup_{|z|\ge 1+\delta/2}\bigl|\tilde{f}(z)-\tilde{f}^{(M)}(z)\bigr|
    \;=\; o(1),
\end{equation}
with probability $1-o(1)$.
Recalling \eqref{eq: 1061}, we can rewrite
\begin{equation*}
	\tilde{f}(z) = \frac{\det(X+C-z)}{\det(X-z)} \quad\text{and}\quad
	\tilde{f}^{(M)}(z) = \frac{\det(X^{(M)}+C-z)}{\det(X^{(M)}-z)}.
\end{equation*}
Then \eqref{eq: 1870} follows directly from the next proposition.

\begin{pro}\label{062702}
	Let $\eps>0$ be any (small) constant. \\
\noindent (i) The following holds for any constant $M>0$: there exists $c_{0}>0$ such that
	\begin{equation*}
		\mathbb{P}\Big\{ \min\Big(\inf_{|z|\ge 1+\delta/2}|z^{-n}\det(X-z)|,\inf_{|z|\ge 1+\delta/2}|z^{-n}\det(X^{(M)}-z)|\Big) < c_{0} \Big\} < \eps,
	\end{equation*}
	for $n$ large enough.
	
\noindent (ii) There exists $C_{0}>0$ such that for every $n\ge 1$,
	\begin{equation*}
		\mathbb{P}\Big\{ \sup_{|z|\ge 1+\delta/2}|z^{-n}\det(X+C-z)| > C_{0}  \Big\} < \eps.
	\end{equation*}
\noindent (iii) Let $c>0$ be any (small) constant. For sufficiently large $M$, we have for every $n\ge 1$,
	\begin{equation*}
		\mathbb{P}\Big\{ \sup_{|z|\ge 1+\delta/2}|z^{-n}\det(X-z)-z^{-n}\det(X^{(M)}-z)| > c  \Big\} < \eps.
	\end{equation*}
	Similarly, if $M$ is large enough, then, for every $n\ge 1$,
	\begin{equation*}
		\mathbb{P}\Big\{ \sup_{|z|\ge 1+\delta/2}|z^{-n}\det(X+C-z)-z^{-n}\det(X^{(M)}+C-z)| > c  \Big\} < \eps.
	\end{equation*}
\end{pro}

\begin{proof}

(i)
By \cite[Theorem 1.2]{BCG22}, it follows that
\begin{equation}\label{eq: 297}
	\inf_{|z|\ge 1+\delta/2}\det(1-z^{-1}X) \underset{n\to\infty}{\overset{\text{law}}{\longrightarrow}} \inf_{|z|\ge 1+\delta/2}\sqrt{1-z^{-2}} \exp\Big(-\sum_{k=1}^{\infty}k^{-1/2}Y_{k}z^{-k}\Big),
\end{equation}
where $\{Y_{k}\}$ are independent Gaussian random variables such that
\begin{equation*}
	\E Y_{k}=0 \quad\text{and}\quad \E |Y_{k}|^{2}=1.
\end{equation*}
The remaining steps coincide with the proof of Proposition \ref{prop: 576} (i) and are therefore omitted.

~

(ii)
Recalling \eqref{eq: 1588} and \eqref{eq: 1592}, we likewise deduce that
\begin{equation}\label{eq: 1926}
	\det(1-z^{-1}X-z^{-1}C) = \sum_{m=0}^{n}\sum_{\substack{0\le\alpha+\beta\le k\\ 0\le\alpha\le m}} (-1)^{\alpha} (-z)^{-(m+\alpha+\beta)} P^{(n)}_{m,\alpha,\beta},
\end{equation}
where
\begin{multline}\label{eq: 1930}
	P^{(n)}_{m,\alpha,\beta} = \sum_{|I|=m} \sum_{\substack{\{i_{1},\cdots,i_{\alpha}\}\subseteq I\\ \{j_{1},\cdots,j_{\alpha},\ell_{1},\cdots,\ell_{\beta}\}\subseteq I^{c} }}  \det([X]_{I,(I\backslash\{i_{1},\cdots,i_{\alpha}\})\cup\{j_{1},\cdots,j_{\alpha}\}}) \\
	\times \det([C]_{\{j_{1},\cdots,j_{\alpha},\ell_{1},\cdots,\ell_{\beta}\},\{i_{1},\cdots,i_{\alpha},\ell_{1},\cdots,\ell_{\beta}\}}).
\end{multline}
Then, as in \eqref{eq: 1615}, it follows that
\begin{multline}\label{eq: 1935}
	\E[|P^{(n)}_{m,\alpha,\beta}|^{2}] \\
	= \sum_{|I|=m} \sum_{\substack{\{i_{1},\cdots,i_{\alpha}\}\subseteq I\\ \{j_{1},\cdots,j_{\alpha}\}\subseteq I^{c} }}
	\sum_{\{\ell_{1},\cdots,\ell_{\beta}\},\{\ell'_{1},\cdots,\ell'_{\beta}\}\subseteq I^{c}\backslash\{j_{1},\cdots,j_{\alpha}\}}
	\E[|\det([X]_{I,(I\backslash\{i_{1},\cdots,i_{\alpha}\})\cup\{j_{1},\cdots,j_{\alpha}\}})|^{2}] \\
	\times \det([C]_{\{j_{1},\cdots,j_{\alpha},\ell_{1},\cdots,\ell_{\beta}\},\{i_{1},\cdots,i_{\alpha},\ell_{1},\cdots,\ell_{\beta}\}})
	\times \overline{\det([C]_{\{j_{1},\cdots,j_{\alpha},\ell'_{1},\cdots,\ell'_{\beta}\},\{i_{1},\cdots,i_{\alpha},\ell'_{1},\cdots,\ell'_{\beta}\}})}.
\end{multline}
Recalling the singular value decomposition \eqref{eq: 1805} of perturbation $C$ together with the assumption that $|s_{i}|\le s^{*}$ for all $i$, we find that
\begin{multline*}
	\sum_{\substack{\{i_{1},\cdots,i_{\alpha}\}\subseteq I\\ \{j_{1},\cdots,j_{\alpha}\}\subseteq I^{c} }}
	\sum_{\{\ell_{1},\cdots,\ell_{\beta}\},\{\ell'_{1},\cdots,\ell'_{\beta}\}\subseteq I^{c}\backslash\{j_{1},\cdots,j_{\alpha}\}} |\det([D]_{\{j_{1},\cdots,j_{\alpha},\ell_{1},\cdots,\ell_{\beta}\},\{i_{1},\cdots,i_{\alpha},\ell_{1},\cdots,\ell_{\beta}\}})| \\
	\times |\det([D]_{\{j_{1},\cdots,j_{\alpha},\ell'_{1},\cdots,\ell'_{\beta}\},\{i_{1},\cdots,i_{\alpha},\ell'_{1},\cdots,\ell'_{\beta}\}})|
	\le  2 (s^{*})^{2(\alpha+\beta)} k^{2(\alpha+\beta)}[(\alpha+\beta)!]^{2},
\end{multline*}
where the estimate follows by the same reasoning used in the proof of Proposition \ref{prop: 576} (ii).
When $|I|=m$, for any index sets $\{i_{1},\cdots,i_{\alpha}\}\subseteq I$ and $\{j_{1},\cdots,j_{\alpha}\}\subseteq I^{c}$ with $\alpha\le m$,
we have $\E[|\det([X]_{I,(I\backslash\{i_{1},\cdots,i_{\alpha}\})\cup\{j_{1},\cdots,j_{\alpha}\}})|^{2}]\le n^{-m}m!$.
Since ${n\choose m }\le n^{m}/m!$, we conclude
\begin{equation}\label{eq: 1946}
	\E[|P^{(n)}_{m,\alpha,\beta}|^{2}] \le 2 (s^{*})^{2(\alpha+\beta)} k^{2(\alpha+\beta)}[(\alpha+\beta)!]^{2}\le 2 (s^{*})^{2k} k^{2k}[(k)!]^{2},
\end{equation}
which completes the proof of this part.

~

(iii)
Analogous to \eqref{eq: 1926} and \eqref{eq: 1930}, in the simpler case of $\det(1-z^{-1}X)$,
\begin{equation*}
	\det(1-z^{-1}X) = 1 + \sum_{m=1}^{n} (-1)^{m}z^{-m}P^{(n)}_{m},
\end{equation*}
where
\begin{equation*}
	P_{m}^{(n)} = \sum_{\substack{I\subseteq\{1,\cdots,n\}\\|I|=m}}\det(X(I)), \quad
	X(I) = (x_{ij})_{i,j\in I}.
\end{equation*}
Similarly,
\begin{equation*}
	\det(1-z^{-1}X^{(M)}) = 1 + \sum_{m=1}^{n} (-1)^{m}z^{-m}P^{(n,M)}_{m},
\end{equation*}
where
\begin{equation*}
	P_{m}^{(n,M)} = \sum_{\substack{I\subseteq\{1,\cdots,n\}\\|I|=m}}\det(X^{(M)}(I)), \quad
	X^{(M)}(I) = (x^{(M)}_{ij})_{i,j\in I}.
\end{equation*}
Since we have \eqref{eq: 1946} and the bound $\E[|P^{(n)}_{m}|^{2}]\le 1$ (referring to \cite[Section 4.1]{BCG22}),
following the same argument as in the proof of Proposition \ref{prop: 576} (iii),
we aim to show 
\begin{equation}\label{eq: 1982}
	\lim_{M\to\infty} \E[|P^{(n)}_{m} - P^{(n,M)}_{m}|^{2}] = 0,
\end{equation}
which is enough to prove the first claim of part (iii).
Note that
\begin{align*}
	\E[|P_{m}^{(n)}-P_{m}^{(n,M)}|^{2}]
	&= \sum_{\substack{I\subseteq\{1,\cdots,n\}\\|I|=m}} \E[|\det(X(I))-\det(X^{(M)}(I))|^{2}] \\
	&= n^{-m}{n\choose m}m! \times \E[|\prod_{j=1}^{m}a_{1j}-\prod_{j=1}^{m}a_{1j}^{(M)}|^{2}],
\end{align*}
where $a_{ij}=\sqrt{n}x_{ij}$ and $a^{(M)}_{ij}=\sqrt{n}x^{(M)}_{ij}$.
Hence we establish \eqref{eq: 1982}.

For the case of perturbation, we introduce
\begin{multline*}
	P^{(n,M)}_{m,\alpha,\beta} = \sum_{|I|=m} \sum_{\substack{\{i_{1},\cdots,i_{\alpha}\}\subseteq I\\ \{j_{1},\cdots,j_{\alpha},\ell_{1},\cdots,\ell_{\beta}\}\subseteq I^{c} }}  \det([X^{(M)}]_{I,(I\backslash\{i_{1},\cdots,i_{\alpha}\})\cup\{j_{1},\cdots,j_{\alpha}\}}) \\
	\times \det([C]_{\{j_{1},\cdots,j_{\alpha},\ell_{1},\cdots,\ell_{\beta}\},\{i_{1},\cdots,i_{\alpha},\ell_{1},\cdots,\ell_{\beta}\}}).
\end{multline*}
Similarly, it boils down to showing
\begin{equation}\label{eq: 402}
	\lim_{M\to\infty} \E[|P^{(n)}_{m,\alpha,\beta} - P^{(n,M)}_{m,\alpha,\beta}|^{2}] = 0.
\end{equation}
We see that
\begin{multline*}
	P^{(n)}_{m,\alpha,\beta} - P^{(n,M)}_{m,\alpha,\beta} \\= \sum_{|I|=m} \sum_{\substack{\{i_{1},\cdots,i_{\alpha}\}\subseteq I\\ \{j_{1},\cdots,j_{\alpha},\ell_{1},\cdots,\ell_{\beta}\}\subseteq I^{c} }} 
	\Big(\det([X]_{I,(I\backslash\{i_{1},\cdots,i_{\alpha}\})\cup\{j_{1},\cdots,j_{\alpha}\}}) - \det([X^{(M)}]_{I,(I\backslash\{i_{1},\cdots,i_{\alpha}\})\cup\{j_{1},\cdots,j_{\alpha}\}})\Big) \\
	\times \det([C]_{\{j_{1},\cdots,j_{\alpha},\ell_{1},\cdots,\ell_{\beta}\},\{i_{1},\cdots,i_{\alpha},\ell_{1},\cdots,\ell_{\beta}\}}).
\end{multline*}
Then, by the independence and mean-zero condition of the entries,
\begin{multline*}
	\E[|P^{(n)}_{m,\alpha,\beta} - P^{(n,M)}_{m,\alpha,\beta}|^{2}]\\
	= \sum_{|I|=m} \sum_{\substack{\{i_{1},\cdots,i_{\alpha}\}\subseteq I\\ \{j_{1},\cdots,j_{\alpha}\}\subseteq I^{c} }}
	\sum_{\{\ell_{1},\cdots,\ell_{\beta}\},\{\ell'_{1},\cdots,\ell'_{\beta}\}\subseteq I^{c}\backslash\{j_{1},\cdots,j_{\alpha}\}}\\
	\E[|\det([X]_{I,(I\backslash\{i_{1},\cdots,i_{\alpha}\})\cup\{j_{1},\cdots,j_{\alpha}\}})-\det([X^{(M)}]_{I,(I\backslash\{i_{1},\cdots,i_{\alpha}\})\cup\{j_{1},\cdots,j_{\alpha}\}})|^{2}]\\
	\times \det([C]_{\{j_{1},\cdots,j_{\alpha},\ell_{1},\cdots,\ell_{\beta}\},\{i_{1},\cdots,i_{\alpha},\ell_{1},\cdots,\ell_{\beta}\}})
	\times \overline{\det([C]_{\{j_{1},\cdots,j_{\alpha},\ell'_{1},\cdots,\ell'_{\beta}\},\{i_{1},\cdots,i_{\alpha},\ell'_{1},\cdots,\ell'_{\beta}\}})}.
\end{multline*}
Again, applying the argument from Proposition \ref{prop: 576} (iii) analogously, and invoking the bound \eqref{eq: 1946},
\begin{equation*}
	\E[|P^{(n)}_{m,\alpha,\beta} - P^{(n,M)}_{m,\alpha,\beta}|^{2}]
	\le 2 (s^{*})^{2k} k^{2k}[(k)!]^{2} \times \E[|\prod_{j=1}^{m}a_{1j}-\prod_{j=1}^{m}a_{1j}^{(M)}|^{2}],
\end{equation*}
which indeed implies \eqref{eq: 402}.

\end{proof}

\section{Eigenvector projection}\label{s eigenvector}

Apart from the study of the singular value of $S$, it is also natural to consider the left and right singular vectors of $S$, $u_i$'s and $v_i$'s. In this appendix, we state and prove some results for eigenvector projections for the heteroscedastic case. 

We write the spectral decomposition of  $\mathcal{Y}$ as 
\begin{align}\label{biorthogonal decomp}
\mathcal{Y} =P\Lambda P^{-1} =\sum_{i=\pm 1, \ldots, \pm n\wedge p} \lambda_i \tilde{w}_i \hat{w}_i^*.
\end{align}
Notice that, the left and right eigenvectors $\tilde{w}_i$'s and  $\hat{w}_i$'s are not necessarily orthonormal, but they are biorthogonal, i.e, $\hat{w}_i^*\tilde{w}_j=\delta_{ij}$. Due the non-Hermitian nature, one cannot use  $\tilde{w}_i \hat{w}_i^*$ to estimate the Hermitian one $w_iw_i^*$ for $i= 1, \ldots, \tilde{k}$, where $w_{i}=\frac{1}{\sqrt{2}}\binom{u_i}{ v_i}$. Nevertheless, for any given direction $a$, we can still use $a^* \tilde{w}_i \hat{w}_i^* a$ to estimate $a^* w_i w_i^* a$ rather precisely.  In the sequel, we state such an estimate regarding the projection of the signal onto any given direction. 

 As in the assumptions for Theorem \ref{thm: first order}, here we also allow multiple $d_i$'s. More specifically, suppose that there exist an integer $k_0, 1 \leq k_0 \leq \tilde{k}$, such that $|\{d_1, \ldots, d_k\}| = k_0$. For $i \in \{1, \ldots, k_0\}$, we denote the index set
\begin{align*}
	\mathcal{I}_{i}= \{j \in \{ 1, \ldots  \tilde{k}\}: d_j = d_i\}.
\end{align*}
The eigenspace formed by the signals can be denoted as the projection matrix $P_i = \sum_{j \in \mathcal{I}_{i}}w_j w_j^*$ associated to signal $d_i$. To ensure the identifiability of the eigenspaces formed by distinct $d_i$'s, we will need an additional non-overlapping condition / eigen-gap requirement.   
\begin{assumption}\label{ass: eigen-gap} We assume that Assumption \ref{main assump} holds.  Additionally, we assume that all distinct $d_i$'s satisfy a minimum separation condition 
\begin{align}\label{eigengap}
	\delta_{i} \deq \min_{j \in \mathcal{I}_{i}^c} |d_i - d_j| > \sigma_{\max}^2 n^{-1/2+\delta},
\end{align}
for some $\delta > 0$.
\end{assumption}

\begin{thm}[First order for eigenvector projection--heteroscedastic case]\label{thm: eigenvector}
	Suppose that Assumption \ref{ass: eigen-gap} holds. For $i = 1, \dots, k_0$. Denote the corresponding random projection by $\tilde{P}_i \deq \sum_{j \in \mathcal{I}_{i}}\tilde{u}_j \tilde{v}_j^*$ associated to the outlying eigenvalues $\lambda_i$. For any unit vector $a \in \mathbb{S}^{p+n-1}$, with high probability, we have for any $\epsilon\in(0,\delta)$, 
	\begin{align*}
		|\langle a^*, \tilde{P}_{i}a \rangle - \langle a^*, P_{i}a \rangle| \leq  \delta_{i}^{-1}n^{-\frac12+\epsilon}\sigma_{\max}^2. 
	\end{align*}
\end{thm}

\begin{rem} Similarly to the eigenvalue case, one can also consider the fluctuation of the eigenvector projections, using similar approach used for the proof of Theorem \ref{thm: second order}. Nevertheless, as our main focus in the eigenvalues, we would prefer to leave the detailed discussion of the fluctuation of eigenvector projections and its applications to a future work.
\end{rem}

\begin{rem}The above theorem does not guarantee that $\tilde{P}_i$ is close to $P_i$, however, it tells us that the projection of any unit vector onto the subspaces defined by $\tilde{P}_i$ and $P_i$ are close. Similar study in the case when there is no noise spike and the signals are $\mu$-incoherent can be found in \cite{CWC21}.  In contrast, our result does not involve any structural assumption about $w_i$. Further, our result shows not only the eigenvalue, but also the eigenvector projection estimation are not affected by the spike noise in $\Sigma$ via the asymmetrization approach.   
\end{rem}

\begin{proof}[Proof of Theorem \ref{thm: eigenvector}]
In order to study the leading (left and right) eigenvectors of $\mathcal{Y}$, we start from its Green function. With the notation in  (\ref{biorthogonal decomp}), we can write 
\begin{align*}
	\tilde{G}(z) \deq (\mathcal{Y}-z)^{-1} = \sum_{i= \pm1, \ldots, \pm n \wedge p} \frac{\tilde{w}_i \hat{w}_i^*}{\lambda_i-z}.
\end{align*}
By the Cauchy integral formula, for $i \in \{1, \ldots, k_0 \}$, 
\begin{align}\label{070401}
	\langle a^*, \tilde{P}_{i}a \rangle = -\frac{1}{2\pi \mathrm{i}} \oint_{\Gamma_{i}} a^*\tilde{G}(z)a  ~\mathrm{d}z,
\end{align}
where $\Gamma_{i}$ is a contour enclosing $\lambda_{i}$ but not any other eigenvalues. From the result of Theorem \ref{thm: first order} and the Assumption \ref{ass: eigen-gap} on $d_i$'s, with high probability, we can choose  $\Gamma_i=\{z: |z-d_i|=\delta_i/2\}$.   Since $\tilde{G}(z)$ is the Green function of low rank perturbation of asymmetrized matrix $\mathcal{X}$, by Woodbury matrix identity, we can write 
\begin{align}\label{071201}
	\tilde{G}(z) &= (\mathcal{X} + W\mathfrak{D}W^* -z)^{-1} \notag \\
	&= (\mathcal{X}-z)^{-1} - (\mathcal{X}-z)^{-1}W\Big(\mathfrak{D}^{-1} + W^*(\mathcal{X}-z)^{-1}W\Big)^{-1}W^*(\mathcal{X}-z)^{-1} \notag \\
	&= (\mathcal{X}-z)^{-1} - (\mathcal{X}-z)^{-1}W(\mathfrak{D}^{-1} -\frac{1}{z} + W^*\underline{\mathcal{X}}W)^{-1}W^*(\mathcal{X}-z)^{-1},
\end{align}
where in the last equality, we recall that by Schur complement formula, we have
\begin{align}\label{071203}
(\mathcal{X}-\lambda)^{-1} &= -\frac{1}{\lambda} + \left(
\begin{array}{ccc}
\lambda \underline{G}(\lambda^2) + \lambda \mathcal{E}(\lambda^2) & \Sigma X_1\mathcal{G}(\lambda^2) +  X_1\mathcal{F}(\lambda^2)  \\
 \mathcal{G}(\lambda^2) X_2^*\Sigma^* + \mathcal{F}(\lambda^2)X_2^*  & \lambda \underline{\mathcal{G}}(\lambda^2) + \lambda \mathcal{F}(\lambda^2) 
 \end{array}
 \right) =: -\frac{1}{\lambda} + \underline{\mathcal{X}}.
\end{align}
The matrix $\underline{\mathcal{X}}$ collects all the sub-leading terms. Specifically, for any unit vector $a \in \mathbb{S}^{p+n-1}$, we have
\begin{align}\label{071202}
	a^*\underline{\mathcal{X}}a = O_{\prec}(\sigma_{\max}^2q_n^{-1}),
	\end{align}
which can be easily seen from Proposition \ref{quadraticformbound} and the estimation in (\ref{102903}) -  (\ref{def of E underline}). For notational brevity, we denote $L(z) = (\mathfrak{D}^{-1} -\frac{1}{z})^{-1}$. As in (\ref{021801}), we will use the expansion
\begin{align*}
	(\mathfrak{D}^{-1} -\frac{1}{z} + W^*\underline{\mathcal{X}}W)^{-1} = L - LW^*\underline{\mathcal{X}}WL + (LW^*\underline{\mathcal{X}}W)^2(\mathfrak{D}^{-1} +W^*(\mathcal{X}-z)^{-1}W)^{-1}. 
\end{align*}
Plugging the above expansion together with (\ref{071203}) and the estimate in (\ref{071202}) into (\ref{071201}), we obtain 
\begin{align*}
	 a^*\tilde{G}(z)a  &= -\frac{1}{z} - \frac{1}{z^2}a^*WLW^*a +\frac{1}{z} a^*\underline{\mathcal{X}}WLW^*a +\frac{1}{z} a^*WLW^*\underline{\mathcal{X}}a-a^*\underline{\mathcal{X}}WLW^*\underline{\mathcal{X}}a \notag \\
	 &\qquad + a^*\left(-\frac{1}{z}+\underline{\mathcal{X}}\right)WLW^*\underline{\mathcal{X}}WLW^*\left(-\frac{1}{z}+\underline{\mathcal{X}}\right)a \notag \\
	 & \qquad - a^*\left(-\frac{1}{z}+\underline{\mathcal{X}}\right)W(LW^*\underline{\mathcal{X}}W)^2(\mathfrak{D}^{-1}  +W^*(\mathcal{X}-z)^{-1}W)^{-1}W^*\left(-\frac{1}{z}+\underline{\mathcal{X}}\right)a \notag \\
	 & \qquad O_{\prec}(\sigma_{\max}^2q_n^{-1}).
\end{align*}
We will later show that except the first two terms in $a^*\tilde{G}(z)a$ above, all the remaining terms are sub-leading when considering the Cauchy integral $\oint_{\Gamma_{i}} a^*\tilde{G}(z)a ~\mathrm{d}z$. Since $L(z)$ is a diagonal matrix with entries 
\begin{align*}
	L_{jj} = (\mathfrak{D}^{-1} -\frac{1}{z})^{-1}_{jj} = \frac{zd_j}{z-d_j},
\end{align*}
we can write 
\begin{align*}
	WLW^* = \sum_{j= \pm 1, \ldots,  \pm k}\frac{w_j w_j^*}{d_j^{-1} - z^{-1}} = \sum_{j=\pm 1, \ldots, \pm k}\frac{z d_i w_j w_j^*}{z - d_j}.
\end{align*}
We first estimate those sub-leading terms in $\oint_{\Gamma_{i}} a^*\tilde{G}(z)a ~\mathrm{d}z$.  First, it is easy to see
\begin{align*}
	\oint_{\Gamma_{i}} \frac{1}{z} a^*\underline{\mathcal{X}}WLW^*a ~\mathrm{d}z &= \sum_{j= \pm 1, \ldots,  \pm k}\langle a, w_j \rangle \oint_{\Gamma_{i}} \frac{d_j}{z-d_j}a^*\underline{\mathcal{X}}w_j~\mathrm{d}z = O_{\prec}(\sigma_{\max}^2q_n^{-1}), \notag \\
	\oint_{\Gamma_{i}} a^*\underline{\mathcal{X}}WLW^*\underline{\mathcal{X}}a ~\mathrm{d}z &= \sum_{j= \pm 1, \ldots,  \pm k} \oint_{\Gamma_{i}} \frac{zd_j }{z-d_j} a^*\underline{\mathcal{X}}w_jw_j^*\underline{\mathcal{X}}a ~\mathrm{d}z = O_{\prec}(\sigma_{\max}^4q_n^{-2}),
\end{align*}
and
\begin{align*}
	&\oint_{\Gamma_{i}} \frac{1}{z^2}a^*WLW^*\underline{\mathcal{X}}WLW^*a ~\mathrm{d}z  = \sum_{j, \ell = \pm 1, \ldots, \pm k} \langle a, w_j \rangle \langle a, w_{\ell} \rangle \oint_{\Gamma_{i}} \frac{d_jd_{\ell}}{(z-d_j)(z-d_{\ell})}w_j^*\underline{\mathcal{X}(z)}w_{\ell} ~\mathrm{d}z \notag \\	
	&\qquad = O_{\prec}(\sigma_{\max}^2q_n^{-1}) \Bigg(\sum_{j, \ell \in \mathcal{I}_{i}} \langle a, w_j \rangle \langle a, w_{\ell} \rangle\oint_{\Gamma_{i}}\frac{d_i^2}{(z-d_i)^2} ~\mathrm{d}z \notag \\
	&\qquad \qquad + 2\sum_{j \in \mathcal{I}_{i}^c}\sum_{\ell \in \mathcal{I}_{i}} \langle a, w_j \rangle \langle a, w_{\ell} \rangle\oint_{\Gamma_{i}} \frac{d_id_j}{(z-d_i)(z-d_j)} ~\mathrm{d}z \Bigg)
	\notag \\
	&\qquad = O_{\prec}(\delta_i^{-1}\sigma_{\max}^2q_n^{-1}),
\end{align*}
Similarly, we have 
\begin{align*}
	\oint_{\Gamma_{i}} \frac{1}{z}a^*\underline{\mathcal{X}}WLW^*\underline{\mathcal{X}}WLW^*a ~\mathrm{d}z &= O_{\prec}(\delta_i^{-1}\sigma_{\max}^4q_n^{-2}), \notag \\
	\oint_{\Gamma_{i}} a^*\underline{\mathcal{X}}WLW^*\underline{\mathcal{X}}WLW^*\underline{\mathcal{X}}a ~\mathrm{d}z &= O_{\prec}(\delta_i^{-1}\sigma_{\max}^6q_n^{-3}),
\end{align*}
and
\begin{align*}
	&\oint_{\Gamma_{i}}a^*\left(-\frac{1}{z}+\underline{\mathcal{X}}\right)W(LW^*\underline{\mathcal{X}}W)^2(\mathfrak{D}^{-1}  + W^*(\mathcal{X}-z)^{-1}W)^{-1}W^*\left(-\frac{1}{z}+\underline{\mathcal{X}}\right)a\Big] ~\mathrm{d}z \notag \\
	& \qquad = O_{\prec}(\delta_i^{-1}\sigma_{\max}^2q_n^{-1}).
\end{align*}
Consequently, in light of (\ref{070401}) and the above estimation, we have
\begin{align*}
	\langle a^*, \tilde{P}_{i}a \rangle &=-\frac{1}{2\pi \mathrm{i}} \oint_{\Gamma_{i}} (-\frac{1}{z} - \frac{1}{z^2}a^*WLW^*a)~\mathrm{d}z + O_{\prec}(\delta_i^{-1}\sigma_{\max}^2q_n^{-1}) \notag \\
	 & = \frac{1}{2\pi \mathrm{i}} \sum_{j = \pm 1, \ldots, \pm k} |\langle a, w_j \rangle|^2 \oint_{\Gamma_{i}} \frac{d_j}{z(z - d_j)} ~\mathrm{d}z  + O_{\prec}(\delta_i^{-1}\sigma_{\max}^2q_n^{-1}) \notag \\
	 & = \sum_{j \in \mathcal{I}_{i}}|\langle a, w_j \rangle|^2 + O_{\prec}(\delta_i^{-1}\sigma_{\max}^2q_n^{-1}) \notag \\
	 & =  \langle a^*, P_{i}a \rangle + O_{\prec}(\delta_i^{-1}\sigma_{\max}^2q_n^{-1}).
\end{align*}
Hence, we complete the proof. 
\end{proof}

\section{Figures}\label{s.fig}
In this appendix, we display all the simulation figures.

\begin{figure}[htbp]


        \centering
        \includegraphics[width=\linewidth]{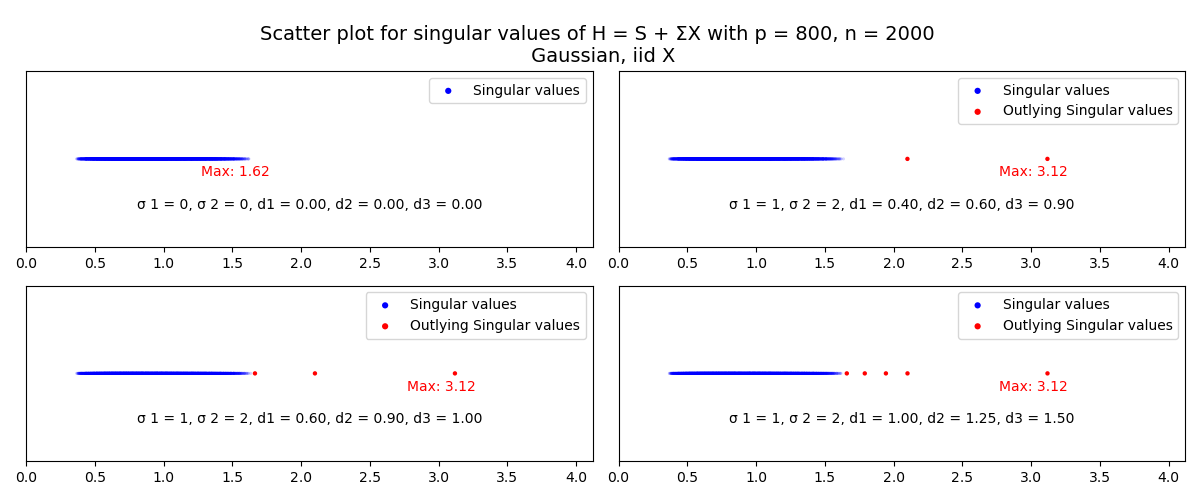} 
       \captionof{figure}{}
        \label{fig:Gaussian_iid_SV}

\end{figure}

\begin{figure}[htbp]

        \centering
        \includegraphics[width=\linewidth]{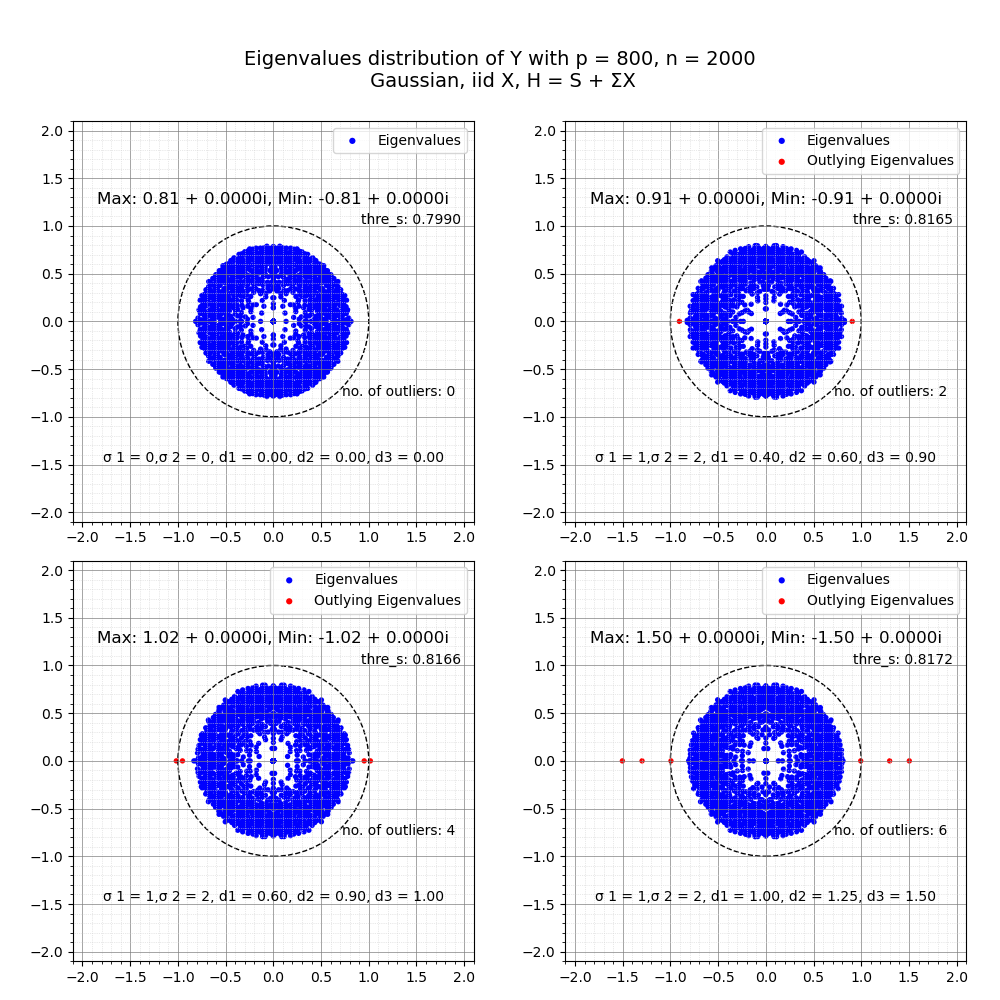}
        \captionof{figure}{}
        \label{fig:Gaussian_iid_EV}

\end{figure}

\begin{figure}[htbp]


        \centering
        \includegraphics[width=\linewidth]{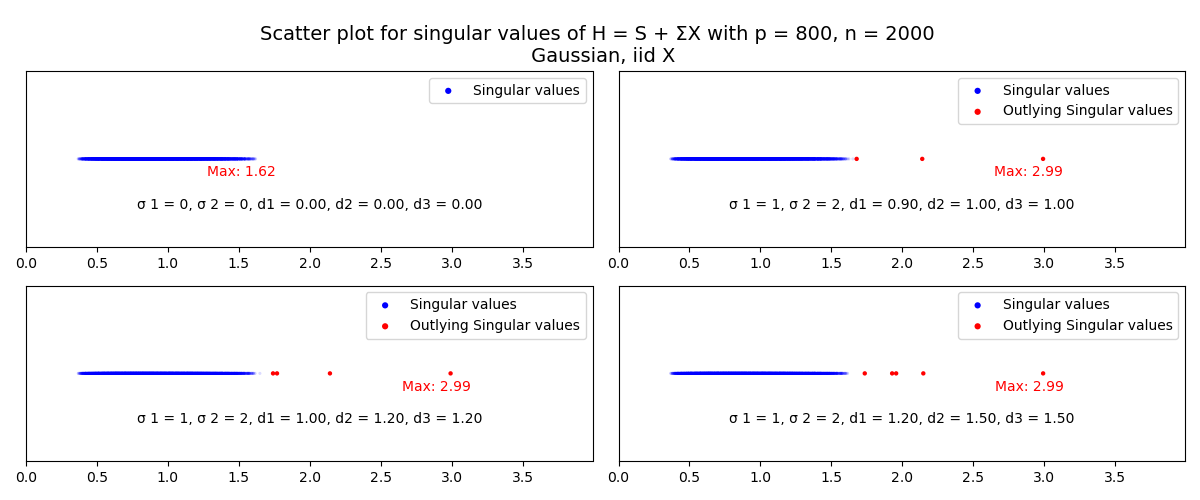} 
        \captionof{figure}{}
        \label{fig:Gaussian_iid_SV_multiple}

\end{figure}

\begin{figure}[htbp]

        \centering
        \includegraphics[width=\linewidth]{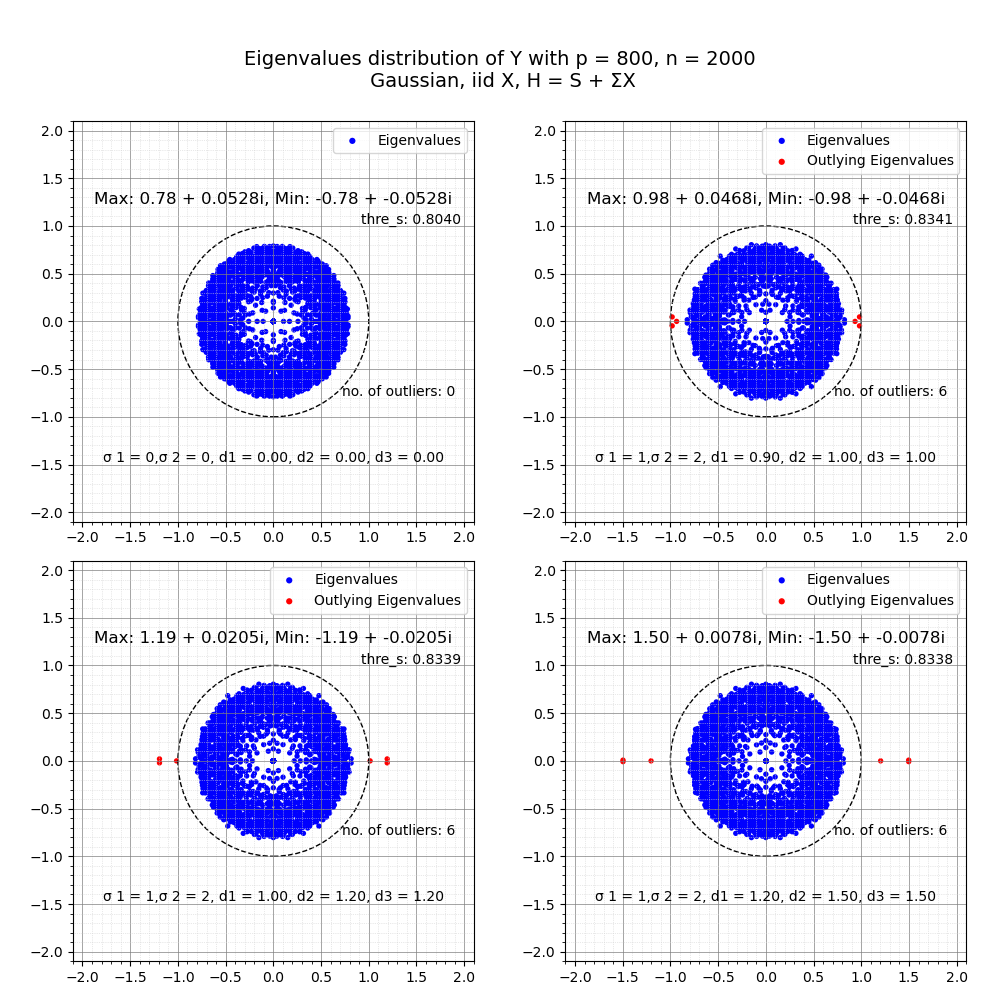}
        \captionof{figure}{}
        \label{fig:Gaussian_iid_EV_multiple}

\end{figure}

\begin{figure}[htbp]


        \centering
        \includegraphics[width=\linewidth]{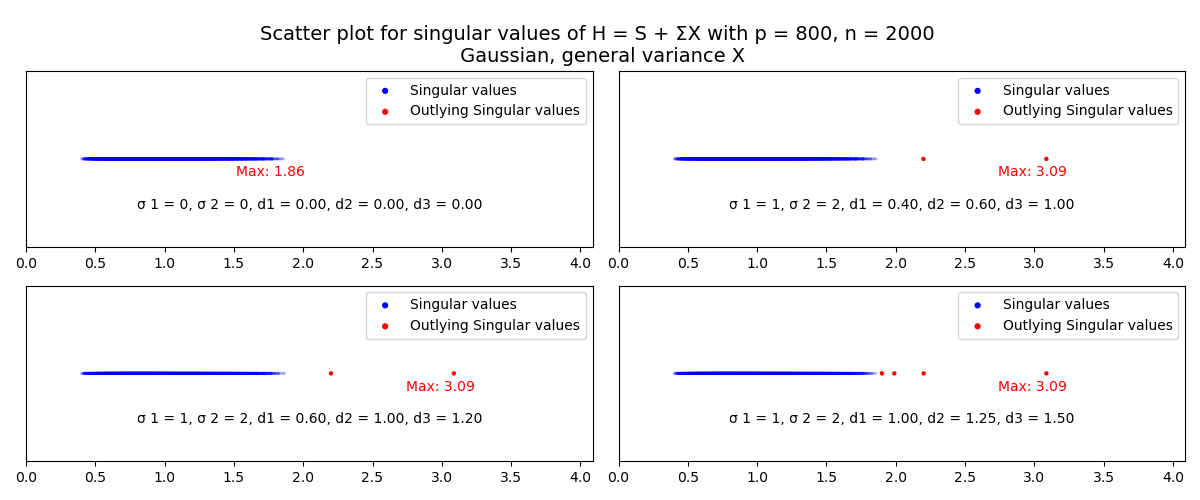} 
        \captionof{figure}{}
        \label{fig:Gaussian_general_SV}

\end{figure}

\begin{figure}[htbp]

        \centering
        \includegraphics[width=\linewidth]{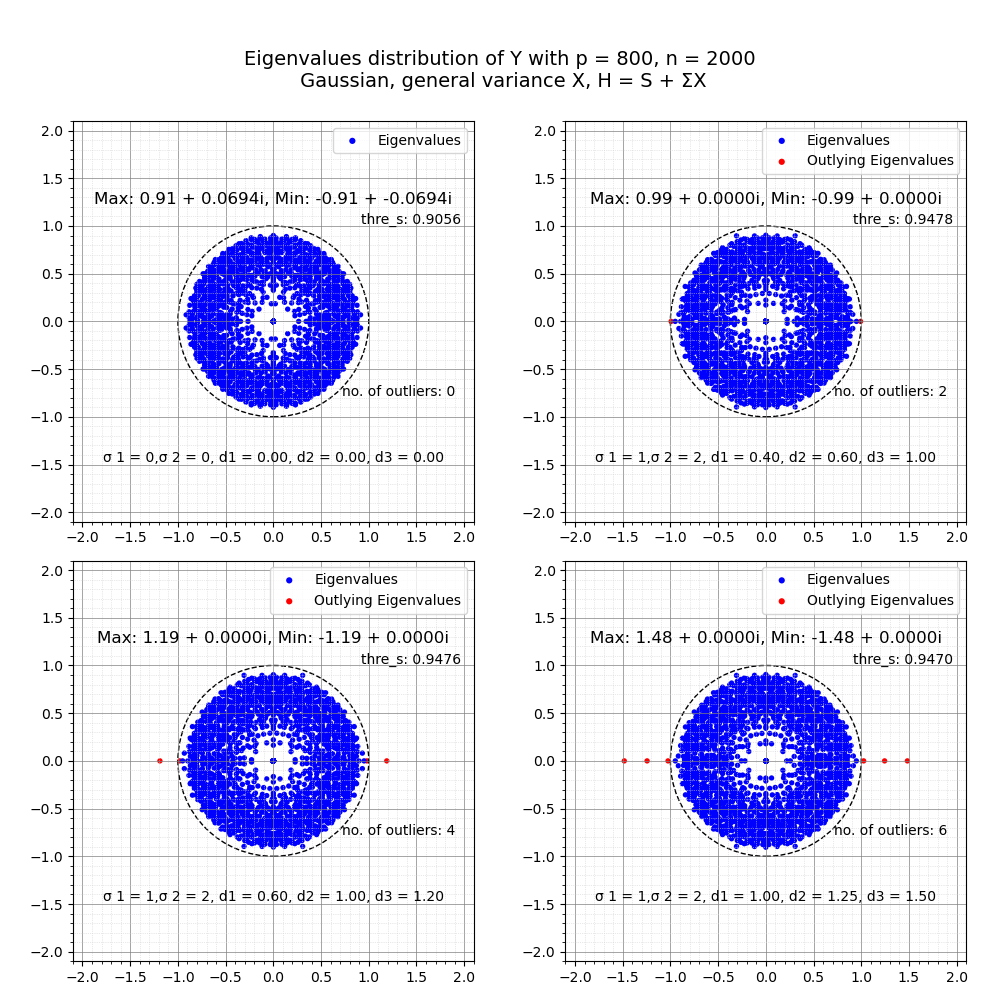}
        \captionof{figure}{}
        \label{fig:Gaussian_general_EV}

\end{figure}

\begin{figure}[htbp]


        \centering
        \includegraphics[width=\linewidth]{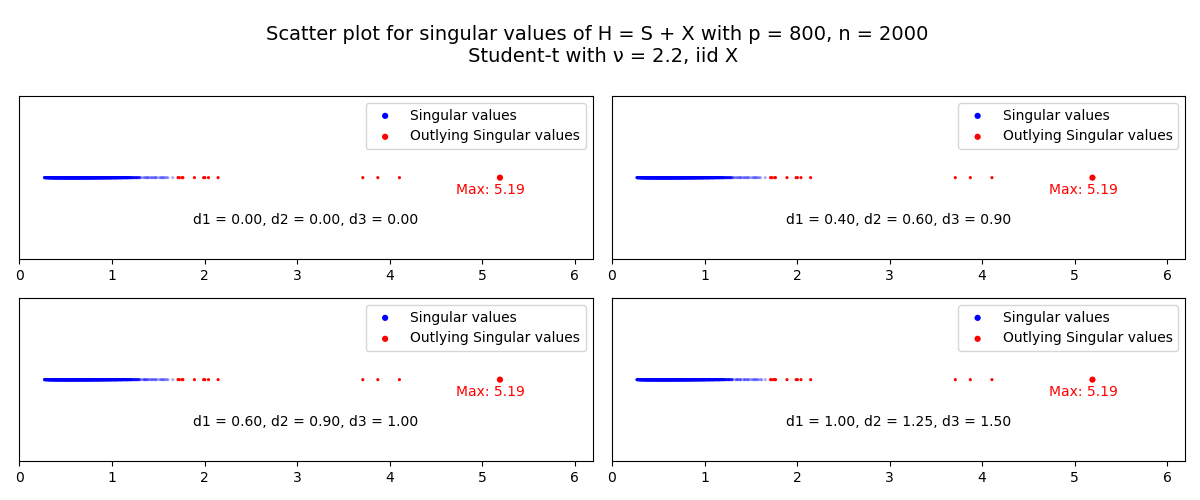} 
        \captionof{figure}{}
        \label{fig:Heavy_iid_SV}

\end{figure}

\begin{figure}[htbp]

        \centering
        \includegraphics[width=\linewidth]{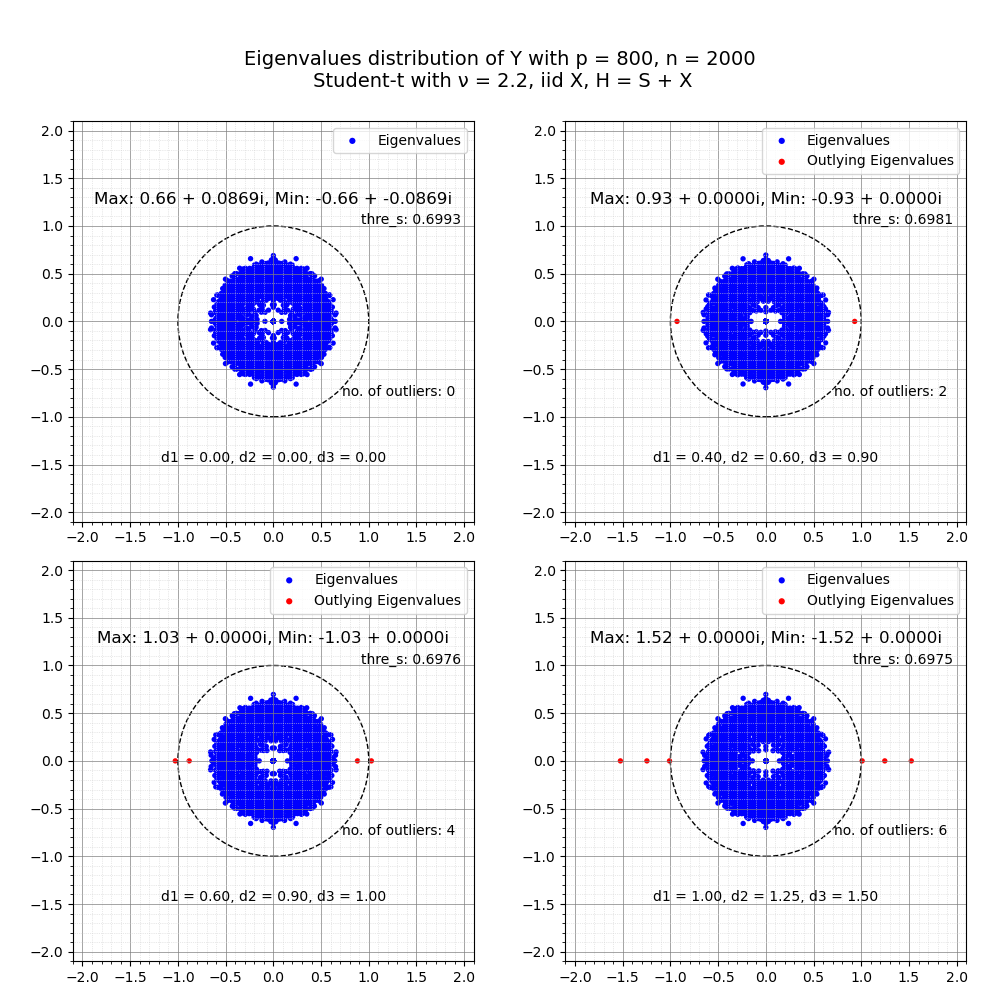}
        \captionof{figure}{}
        \label{fig:Heavy_iid_EV}

\end{figure}

\end{appendix}

\end{document}